\documentclass[reqno,11pt]{amsart}
\usepackage{color} 
\usepackage[left=1 in, right=1 in,top=1 in, bottom=1 in]{geometry}
\usepackage{amsfonts,amsmath,amsthm,amssymb,mathrsfs,environ}
\newtheorem{theorem}{Theorem}[section]

\newtheorem{lemma}[theorem]{Lemma}

\newtheorem{remark}[theorem]{Remark}

\numberwithin{equation}{section}

\usepackage[pdfstartview=FitH,colorlinks=true,backref=page]{hyperref}
\hypersetup{urlcolor=red, linkcolor=blue,citecolor=red, CJKbookmarks=true}
\allowdisplaybreaks
\makeatletter
\newcommand*{\rom}[1]{\expandafter\@slowromancap\romannumeral #1@}

\newcommand{\bd}{\mathrm{d}}
\newcommand{\grad}{\nabla_{x}}
\newcommand{\vdot}{v\!\cdot\!\nabla_x}
\newcommand{\sdot}{\!\cdot\!}
\newcommand{\weakc}{\rightharpoonup}
\newcommand{\m}{\dm}
\newcommand{\bdv}{\dm \bd v}
\newcommand{\dm}{\mathrm{M}}
\newcommand{\bdp}{\mathcal{P}}
\newcommand{\bdpv}{\mathcal{P}_{\!\!\perp}}
\newcommand{\bds}{\mathrm{S}}
\newcommand{\divg}{\mathrm{div}}
\newcommand{\Q}{\mathcal{Q}}
\newcommand{\vpbg}{\mathcal{G}_{\epsilon}}
\newcommand{\intps}{\int_{\scriptstyle \mathbb{T}^3}\!\int_{\scriptstyle \mathbb{R}^3} }
\newcommand{\intpsz}{\int_{\scriptstyle \mathbb{I}_z}\int_{\scriptstyle \mathbb{T}^3}\!\int_{\scriptstyle \mathbb{R}^3} }
\newcommand{\intv}{\int_{\mathbb{R}^3}}
\newcommand{\intt}{\int_{\mathbb{T}^3}}

\newcommand{\intz}{\int_{\mathbb{I}_z}}
\newcommand{\reps}{\tfrac{1}{\epsilon}}
\newcommand{\repst}{\tfrac{1}{\epsilon^2}}
\newcommand{\repsta}{\tfrac{a_2}{\epsilon^2}}

\newcommand{\llv}{L^2_{\Lambda_{v}}}

\newcommand{\dt}{\frac{\mathrm{d}}{2 \mathrm{d}t}}
\newcommand{\tdt}{\frac{\mathrm{d}}{ \mathrm{d}t}}
\newcommand{\hszt}{L^2L^2_z}

\newcommand{\lzi}{L^\infty_z}
\newcommand{\hszti}{L^{2\cap\infty}_z}
\newcommand{\eps}{\epsilon}

\makeatother
\renewcommand{\ge}{\geqslant}
\renewcommand{\le}{\leqslant}

\NewEnviron{aligno}{%
\begin{align}\begin{aligned}
  \BODY
\end{aligned}\end{align}
}

\title[NS-VPB-UQ]{ Sensitivity analysis and incompressible Navier-Stokes-Poisson limit of Vlasov-Poisson-Boltzmann equations with uncertainty }
\author[Ning Jiang]{Ning Jiang}
\address[Ning Jiang]{\newline School of Mathematics and Statistics, Wuhan University, Wuhan, 430072, P. R. China}
\email{njiang@whu.edu.cn}

\author[X. Zhang]{Xu Zhang}
\address[Xu Zhang]{\newline School of Mathematics and Statistics, zhengzhou University, Zhengzhou, 450001, P. R. China}
\email{xuzhang889@zzu.edu.cn}

\begin{document}
	\begin{abstract}
For the Vlasov-Poisson-Boltzmann equations with random uncertainties from the initial data or collision kernels, we proved the sensitivity analysis and energy estimates uniformly with respect to the Knudsen number in the diffusive scaling using hypocoercivity method. As a consequence, we also justified the incompressible Navier-Stokes-Poisson limit with random inputs. In particular, for the first time, we obtain the precise convergence rate {\em without} employing any results based on Hilbert expansion. We not only generalized the previous deterministic Navier-Stokes-Poisson limits to random initial data case, also improve the previous uncertainty quantification results to the case where the initial data include both kinetic and fluid parts.
	\end{abstract}
\maketitle	
\section{Introduction}
\subsection{Vlasov-Poisson-Boltzmann system}
Dilute electrons in the absence of a magnetic field can be described by the Vlasov-Poisson-Boltzmann system, which is fundamental in plasma physics. In this paper, we consider the one-species VPB, which is a coupling a Boltzmann equation with a poisson equation arising from electric field:
\begin{equation}\label{vpb}
    \left\{
    \begin{array}{c}
\partial_t f + \vdot f + \grad\phi\sdot\nabla_v f = \Q(f,f), \\[2mm]
\Delta_x \phi=\int_{\mathbb{R}^3} f \bd v - \rho_b\,.
        \end{array}
  \right.
\end{equation}
It models the dynamics of dilute electrons under the self-consistent electric field while the   ion is seen as the background charge. Here the non-negative $f(t,x,v)$ represents the density function of electrons in position $x\in \mathbb{T}^3$ with velocity $v\in\mathbb{R}^3$ at time $t>0$. In \eqref{vpb}, the first equation is the classical Boltzmann equation, and coupled with the $\rho_b$ which is the density of  the background charge, the self-consistent electrostatic field $\nabla_{\!x}\phi(t,x)$ is generated and described by the Poisson equation. The collision between particles is given by the standard Boltzmann collision operator $\Q(f,f)$ with hard-sphere interaction:
\begin{equation}\nonumber
  \Q(f,f)(v)= \int_{\mathbb{R}^3 \times \mathbb{S}^2} ( f^\prime f_1^\prime - f f_1 ) | ( v - v_1 ) \cdot \omega | \bd \omega \bd v_1 \,,
\end{equation}
where
\begin{equation*}
  f = f(t,x,v) \,, \quad f^\prime = f ( t , x , v^\prime ) \,, \quad f_1 = f ( t , x , v_1 ) , \quad f^\prime_1 = f ( t , x , v_1^\prime ) \,,
\end{equation*}
\begin{equation*}
  v^\prime = v - [ ( v - v_1 ) \cdot \omega ] \omega \,, \quad v_1^\prime = v_1 + [ ( v - v_1 ) \cdot \omega ] \omega \,, \quad \omega \in \mathbb{S}^2 \,.
\end{equation*}
It is well-known that the constant state (denoted by $\mathrm{M}(v)$, and called global Maxwellian) which makes the collision operator $\Q(\dm, \dm)=0$ has the following form (after normalization):
\begin{equation}
  \dm(v)= \tfrac{1}{(2\pi)^{3/2}}e^{-|v|^2/2}\,.
\end{equation}

There are many contributions to the mathematical analysis of VPB, or more generally Boltzmann type equations. Basically, there are two types of results on the well-posedness of the kinetic equations: the first is ``large" initial data renormalized solutions introduced by DiPerna-Lions, see \cite{DL-Annals1989} and \cite{Lions-Tokyo1994-1, Lions-Tokyo1994-2}. The second is the classical solutions near global Mexwellian, initialed from Ukai \cite{Ukai}, and later developed by Guo using the so-called nonlinear energy method, \cite{Guo-vpb, guo-2003-vmb-invention}. Based on these seminal contributions, there are lots of improvements in the past two decades from the perspectives of boundary conditions {\cite{Mischler}}, collision kernels (from cut-off to non-cutoff kernels) and regularity, which we will not list here.

The main goals of this paper is two-folds: 1. We connect the one-species VPB \eqref{vpb} in the diffusive scalings (see the scaled form \eqref{vpb-ns-scaling}) to the incompressible Navier-Stokes-Fourier-Poisson (NSFP) equations. We will justify this limit in the context of classical solutions. 2. We consider the uncertainty quantification (UQ) problem for VPB under this diffusive scaling. In this paper the random effects come from the initial data and collision kernels, and we will analyze the sensitivity of the uncertainty. In the following subsections, we describe the above problems in details.

\subsection{Incompressible NSFP limit}
One of the most important features of the Boltzmann-type equations (or more generally, kinetic equations) is their connection to the fluid equations. The so-called fluid regimes of the kinetic equation are those of asymptotic dynamics of the scaled kinetic equations when the Knudsen number $\eps$ (which is the ratio of the mean free path to the macroscopic length scale) is very small. Justifying these limiting processes rigorously has been an active research field from late 70's. Among many results obtained, the most complete contributions (at least for global in time and high spatial dimension cases) are the incompressible Navier-Stokes limits. There are two types of results in this field:
\begin{enumerate}
  \item Start the solutions of the scaled kinetic equations with estimates {\em uniformly} in the Knudsen number $\eps$, then extract a convergent (sub)sequence converging to the solutions of the fluid equations as $\eps\rightarrow 0\,;$.
  \item Start from the solutions for the limiting fluid equations, then construct a sequence of {\em special} solutions of the scaled kinetic equations for small Knudsen number $\eps$.
\end{enumerate}
The key difference between the results of type (1) and (2) are: in type (1), the solutions of the fluid equations are {\em not} known a priori, and are completely obtained from taking limits from the kinetic equations. In short, it is ``from kinetic to fluid", or ``bottum-up" (the term used in Mischler-Mouhot in \cite{Mischler-Mouhot-Invent2013}). This approach in fact gives a mesoscopic proof of the existence for the limiting macroscopic equations. In type (2), the solutions of the fluid equations are known {\em first}. In short, it is ``from fluid to kinetic", or ``top-down".

Most of the type (2) results are based on the Hilbert expansion  and obtained in the context of classical solutions. It was started from Nishida and Caflisch's work on the compressible Euler limit \cite{Nishida, Caflisch, KMN}. Their approach was revisitied by Guo, Jang and Jiang, combining with nonlinear energy method to apply to the acoustic limit \cite{GJJ-KRM2009, GJJ-CPAM2010, JJ-DCDS2009}. After then this process was used for the incompressible limits, for examples, \cite{DEL-89} and \cite{guo2006NSlimit}. In \cite{DEL-89}, De Masi-Esposito-Lebowitz considered  Navier-Stokes limit in dimension 2. More recently, using the nonlinear energy method, in \cite{guo2006NSlimit} Guo justified the Navier-Stokes limit (and beyond, i.e. higher order terms in Hilbert expansion). This result was extended in \cite{JX-SIMA2015} to more general initial data which allow the fast acoustic waves.

The most successful example of type (1) result is the so-called BGL program (named after Bardos-Golse-Levermore's work \cite{bgl1991formal, bgl1993convergence}), which justifies weak limit starting from DiPer-Lions solutions of Boltzmann equations to weak solutions of fluid equations (mostly, incompressible Navier-Stokes and Euler equations). This program started from \cite{bgl1991formal, bgl1993convergence} and completed  in \cite{Go-Sai04, Go-Sai09} by Golse and Saint-Raymond. For the limits in BGL program for bounded domain, see \cite{MSRM-CPAM2003} and \cite{JM-CPAM2017}.

In another line of research, there are also a few type (1) results in the context of classical solutions. The first result in this direction might be Bardos-Ukai \cite{b-u} in which the semigroup method was employed to justify the incompressible Navier-Stokes limit from the hard sphere Boltzmann equation. Recently, on the torus, hypocoercivity method was used by Briant \cite{briant-2015-be-to-ns} and Briant, Merino-Aceituno and Mouhot \cite{BMMouhot} to provide a constructive proof of the incompressible Navier-Stokes limit. Their methods can be considered as an improvement of semigroup method, based on the functional analysis breakthrough of Gualdani-Mischler-Mouhot \cite{GMM}. We emphasize that in \cite{briant-2015-be-to-ns}, it was for the first time a convergence rate was explicitly provided, for hard sphere kernel, and for special initial data which coincide with that needed in the Hilbert expansion method \cite{guo2006NSlimit}. To obtain this convergence rate, it essentially needed Guo's result \cite{guo2006NSlimit} which was based on Hilbert expansion.

In the deterministic part of the current paper, our strategy was inspired by the above works, in particular \cite{briant-2015-be-to-ns}, to give a constructive proof of incompressible NSFP limit of the one-species VPB. Note that for the proof based energy method, this limit has been justified by Guo-Jiang-Luo in \cite{Guo-Jiang-Luo}, using the similar method for Boltzmann equation \cite{ns-limit-2018}. (After we finish the draft of this paper, we heard that Li-Zhong-Yang also prove this limit using the method in the spirit of Ukai's semigroup method). Comparing to the corresponding Navier-Stokes limit of the Boltzmann equation in \cite{briant-2015-be-to-ns, BMMouhot}, there are some key difference and difficulties. The first is for the linearized VPB operator, the fluid and kinetic parts are not orthogonal. The second, which we consider as more important, is to obtain an explicit convergence rate, we delicately decompose the fluid and kinetic parts in the semigroup, and use the structure of the NSFP equations, then avoid any Hilbert expansion type results (which was used in \cite{briant-2015-be-to-ns}, as mentioned above). This is one of the novelty of this paper.

\subsection{UQ of VPB in diffusive limit}
Another main goal of this paper is to study the above kinetic equation and its diffusive approximation under the influence of random uncertainty. Since kinetic equations are not first-principle physical equations, there are inevitably modeling errors, incomplete knowledge of the interaction mechanism, and imprecise measurement of the initial and boundary data, which contribute uncertainties to the equations. Understanding the impact of these uncertainties is crucial to the simulation and validation of the models, in order to provide more reliable predictions, calibrations, and improvements of the models. In this paper we consider the uncertainty coming from initial data and collision kernels. The uncertainty is described by the random variable $z$, which lies in the random space $I_z$ with a probability measure $\pi(z)\mathrm{d}z$, then the solution $f = f(t,x,v,z)$ depends on $z$. The {\em sensitivity analysis} aims to study how randomness of the initial data and collision kernel (the  ``input") propagates in time and how it affects the solution in the long time (the ``output"). It is an essential part of the so-called uncertainty quantification for kinetic equations. Among many previous contributions in this direction, we only list \cite{hujin-2017-uq-review, liujin-2018-kinetic, Jin-Liu-Ma-2017, Jin-Zhu, Hu-Jin-Shu}.

In \cite{liujin-2018-kinetic}, Liu and Jin prove that the solutions of nonlinear kinetic equations with random inputs tend to the deterministic global equilibrium under the acoustic and Navier-Stokes (NS) scalings, in the functional analysis setting of \cite{briant-2015-be-to-ns}. In \cite{liujin-2018-kinetic}, they proved the sensitivity analysis of the Boltzmann equation in the acoustic and NS scalings, but their initial data did not contain the corresponding acoustic or Navier-Stokes parts. Thus, at later time, the fluid parts are totally ignored. The current paper not only generalize the result of \cite{liujin-2018-kinetic} from Boltzmann to VPB, more importanlyly, but also includes the fluid parts and the interactions between the fluid and kinetic parts under the random inputs. This is the first UQ result for spatially high dimension VPB equations in diffusive limits containing Navier-Stokes-Fourier-Poisson dynamics.

\section{Settings on the kernels}
In this section, we will state the assumpution on kernel for both deterministic case and random case respectively.

\subsection{Linearization around Maxwellian}
First, in this subsectin, we derive the linearized equation of VPB system around the global Maxwellian. According to \cite[Sec. 2.3.5]{diogosrm-2019-vmb-fluid}, the scaled VPB equation is
\begin{align}
\label{vpb-ns-scaling}
\begin{cases}
\epsilon \partial_t f_\epsilon + \vdot f_\epsilon + \epsilon \grad\phi_\epsilon\sdot\nabla_v f_\epsilon = \tfrac{1}{\epsilon}\Q(f_\epsilon,f_\epsilon), \\
\epsilon \Delta_x \phi_\epsilon(t,x)=\intv f_\epsilon(t,x,v) \bd v - 1,
\end{cases}
\end{align}
with
\begin{align}
\label{vpb-linear-m}
f_\epsilon= \m(1 + \epsilon g_\epsilon),   \mathrm{M} =  \tfrac{1}{(\sqrt{2\Pi})^3}\exp(-\tfrac{|v|^2}{2}).
\end{align}
Pluging $f_\epsilon= \m(1 + \epsilon g_\epsilon)$ into \eqref{nsfp}, the equations of $g_\epsilon$ is

\begin{align}
\label{vpb-ns-scaling-ns-corrector}
\begin{cases}
  \partial_t g_\epsilon + \reps \vdot g_\epsilon +  \repst \mathcal{L}(g_\epsilon) +  (\nabla_v g_\epsilon   -  {v}  g_\epsilon) \sdot \nabla\phi_\epsilon - \tfrac{v}{\epsilon} \sdot \nabla \phi_\epsilon = \reps {\Gamma}(g_\epsilon,g_\epsilon), \\
 \Delta_x \phi_\epsilon =\intv g_\epsilon \m \bd v,
\end{cases}
\end{align}
where
\begin{align}
\begin{cases}
\mathcal{L}(g )& = \m^{-1}\big[ \Q(\m,  \m g) + \Q(\m g , \m) \big]\\
               & =\int_{\mathbb{T}^3\times\mathbb{S}^{2}}(g - g' + g_* - g_*') B(v-v_*,\omega) \m(v_*)\bd v_*\bd\omega\,\\
\Gamma(g ,h )  & = \tfrac{1}{2\m}\big[ \Q(\m h, \m g) + \Q(\m g , \m h) \big]\\
 & =\tfrac{1}{2}\int_{\mathbb{T}^3\times\mathbb{S}^{2}}(g_*' h' + h_*' g'- g_* h - h_* g) B(v-v_*,\omega) \m(v_*)\bd v_*\bd\omega.
\end{cases}
\end{align}
As long as the uniform estimates with respect to $\epsilon$ is established. By weak convergence method and based on the formal analysis in \cite[Sec. 2.3.5]{diogosrm-2019-vmb-fluid}, the corresponding   Navier-Stokes-Fourier-Poisson equations
\begin{align}
\label{nsfp}
\begin{cases}
\partial_t u + u\sdot \nabla u - \nu \Delta u + \nabla P = \rho \nabla \theta,\\
\partial_t(\tfrac{3}{2}\theta - \rho) + u \sdot\nabla (\tfrac{3}{2}\theta - \rho) - \tfrac{5\kappa}{2}\Delta \theta =0,\\
\divg u = 0,~~ \Delta(\rho + \theta ) = \rho,~~E = \nabla(\rho + \theta).
\end{cases}
\end{align}
can be derived from \eqref{vpb-ns-scaling-ns-corrector}.

\begin{remark}
\label{remark-sqrt-m}
The linearization is different with \cite{briant-2015-be-to-ns} and \cite{mouhotneumann-2006-decay}. Indeed, in \cite{briant-2015-be-to-ns, mouhotneumann-2006-decay},
\begin{align}
\label{vpb-linear-sqrtm}
 f_\epsilon = \sqrt{\m}(\sqrt\m + \epsilon g_\epsilon),
\end{align}
by simple calculation, in this settings, the $\tfrac{v}{\epsilon} \sdot \nabla \phi_\epsilon$ in \eqref{vpb-ns-scaling-ns-corrector} is replaced by  $\tfrac{v}{\epsilon} \sdot \nabla \phi_\epsilon \sqrt\m $. Even for linear equation, it is compulsory to  add couple of new terms up  to  the norm used in \cite[Theorem 2.1]{briant-2015-be-to-ns} to obtain a closed inequality. The added terms hint that it is more convenient to use \eqref{vpb-linear-m} other than \eqref{vpb-linear-sqrtm}.  More explainations can be found in Remark \ref{remark-m-sqrtm}.
\end{remark}

\begin{remark}
The comparsion between the linear Boltzmann operator and linear VPB operator will be stated in Sec. \ref{sec-linear-vpb}.
\end{remark}

\subsection{Settings of kernels on the deterministic case}
\label{sec-assump-determin} For the deterministic case, the assumptions are the same as those in  \cite[Sec.2.1]{briant-2015-be-to-ns}. We copy the below for the reader's convience.  Before them, we introduce some notation to be extensively used later. $\nabla^i_x  f $($\nabla^j_vf$) denotes all the $i$-$th$($j$-$th$) derivative  of $f$ with respect to $x$($v$). Specially, $\nabla_x f$ is the gradient of scalar function $f$. $\nabla_x^1 f$ denotes all the first order derivative of $f$.
\begin{align*}
&\|f\|_{L^2_v}^2 =  \intv f^2 \bdv ,   ~~\|f\|_{L^2}^2 =  \intps f^2 \bdv \bd x,~~\|f\|_{H^s_x}^2 =  \sum\limits_{k=0}^s\|\nabla^k_x f\|_{L^2}^2, \\
& \|f\|_{L^2_\Lambda}^2 =  \intps f^2 \hat{v}\bdv \bd x,~~~~\hat{v}= 1 + |v|, |f\|_{H^s_{\Lambda_x}}^2 =  \sum\limits_{k=0}^s\|\nabla^k_x f\sqrt{\hat{v}}\|_{L^2}^2,~~\|f\|_{\llv}^2 =  \intv f^2 \hat{v}\bdv\\
& \|f\|_{H^s_\Lambda}^2 =  \sum\limits_{k=0}^s \sum\limits_{i + j =k}\|\nabla^i_x \nabla^j_v f\|_{L^2_\Lambda}^2,~~\|f\|_{H^s}^2 =  \sum\limits_{k=0}^s \sum\limits_{i + j =k}\|\nabla^i_x \nabla^j_v f\|_{L^2}^2.
\end{align*}
The assumputions on the kernels are as follows.

{\bf 1, Coercivity assumpution.} The linear Boltzmann operator $\mathcal{L}$ is a closed and self-adjoint operator from $L^2_v$ to $L^2_v$ with the following decomposition:
\begin{align}
\label{assump-h1-h}
\mathcal{L}(h) = -\mathcal{K}(h) + \Lambda(v)\sdot h.
\end{align}
Furthermore, $\Lambda$ is coercive:
\begin{itemize}
\item for any $h\in L^2_v$ and some positive constant $\tilde{a}$ and $C_u$, \[ \tilde{a} \|h\|_{\llv} \le \intv \Lambda(h)\sdot h \bdv   \lesssim  \|h\|_{\llv}, \]
and
\begin{align}
\vert\intv  \Lambda(h)\sdot g\bdv  \vert \le C_u \|h\|_{\llv}\|g\|_{\llv}.
\end{align}
\item For any $f, g \in L^2_v$,
\begin{align}
\intv f \sdot \mathcal{L}(g) \bdv  \lesssim \|f\|_{\llv} \|g\|_{\llv}.
\end{align}

\item With respect to the derivative to $v$, the operator $\Lambda$ admits ``a defect of coercivity'', i.e.,
there exist some strictly positive  constants $a_3$ and $a_4$   such that
\begin{align}
\label{constant-a3-a4}
\intv\nabla_{v} \Lambda(h)\sdot\nabla_{v} h\bdv   \geqslant a_3 \left\|\nabla_{v} h\right\|_{\llv}^{2}-a_4\|h\|_{L^{2}_v}^{2}.
\end{align}
and for the high order derivative,
\begin{align}
\label{constant-a3-a4-hi}
\intps \nabla_{v}^i \nabla_x^j \Lambda(h)\sdot\nabla_{v}^i \nabla_x^j h\bdv \bd x \geqslant a_3 \left\|\nabla_{v}^i \nabla_x^j h \right\|_{L^2_\Lambda}^{2}-a_4\|h\|_{H^{i+j-1}}^{2}.
\end{align}
\end{itemize}

{\bf 2, Mixing property in velocity}

Furthermore, for any $\delta>0$, there exists some positive constant $C_\delta$,
\begin{align}
\label{constant-delta}
\vert \intv \nabla_{v} \mathcal{K}(h) \sdot \nabla_{v} h\bdv  \vert  \leqslant C_\delta\|h\|_{L^{2}_v}^{2}+\delta\left\|\nabla_{v} h\right\|_{\llv}^{2},
\end{align}
and for high order derivative,
\begin{align}
\vert \intps \nabla_{v}^i \nabla^j_x  \mathcal{K}(h) \sdot \nabla_{v}^i \nabla^j_x h\bdv  \vert  \leqslant C_\delta\|h\|_{H^{i+j-1}}^{2}+\delta\left\|\nabla_{v}^i \nabla^j_x\right\|_{L^2_\Lambda}^{2},
\end{align}

{\bf 3, Relaxation to equilibrium}

The kernl space of $\mathcal{L}$ is spanned by
\[ 1, v_1, v_2, v_3, |v|^2. \]
Furthermore, there exists some constant $a_2$
\begin{align}
\label{constant-a2}
  \intps \mathcal{L}(h)\sdot h \bdv \bd x   \geqslant a_2 \|h^\perp\|_{L^2_\Lambda}^2,
\end{align}
where  $h^\perp$ is defined as follows:
\[ h^\perp = h -\bdp h, ~~\bdp h = \int_{\mathbb{R}^3} h \bdv + v \sdot \int_{\mathbb{R}^3} h v \bdv +  \tfrac{|v|^2 -3}{2}  \sdot \int_{\mathbb{R}^3} \tfrac{|v|^2 -3}{2}  h \bdv. \]

 {\bf 4, Assumption  on  the bilinear term}
\begin{itemize}
\item For any $g, h\in L^2$, $\Gamma(g,h) \in \mathrm{Ker}(\mathcal{L})^\perp$.
\item For the non-linear operator, $s \ge 3$,
\begin{align}
\label{constant-cn}
\begin{split}
&\vert \intps \nabla^s_x \Gamma(g,h) \sdot f \bdv \bd x \vert \lesssim \|(g,h)\|_{H^s_{x}} \|(g,h)\|_{H^s_{\Lambda_x}}\|f^\perp\|_{L^2_\Lambda},\\
&\vert \intps \nabla^j_x \nabla_v^i \Gamma(g,h) \sdot f \bdv \bd x \vert \lesssim \|(g,h)\|_{H^s} \|(g,h)\|_{H^s_{\Lambda }}\|f\|_{L^2_\Lambda},~~i\ge 1,~~ s = i+j.
\end{split}
\end{align}
\end{itemize}

\subsection{Settings of kernels on the random case}
\label{sec-assump-ran}
For the random case, the solution is dependent of $z$, i.e.,
\[ g_\epsilon = g_\epsilon(t,x,v,z),~~ z \in \mathbb{I}_z. \]
Besides, the random variable $z$ is assumed to be one dimension with bounded support.

The uncertainties  come from not only the initial data but also the collision kernels. For the collision kernels, the cross section depends on $z$, i.e.,
\begin{align*}
B(|v-v_*|, \cos\theta, z) = C_b |v-v_*|b(\cos \theta, z).
\end{align*}
Furthermore, for any $\eta \in [-1,1]$, $z \in \mathbb{I}_z$, there exists some constant $\tilde{C}_b$ such that
\[ |b(\eta, z)| \le \tilde{C}_b, |\partial_\eta b(\eta, z)| \le \tilde{C}_b,~~ |\partial_z b(\eta, z)| \le \tilde{C}_b.  \]

Based on the above assumption, for the derivative of $\mathcal{L}$ and $\Gamma$ with respect to $z$,
\begin{align*}
\mathcal{L}_z(g )=\int_{\mathbb{T}^3\times\mathbb{S}^{2}}(g - g' + g_* - g_*') \partial_z B(v-v_*,\omega,z) \m(v_*)\bd v_*\bd\omega\,
\end{align*}
we also assume that for any $h \in L^2$ and for each $z \in \mathbb{I}_z$,
\begin{align*}
\mathcal{L}_z(h) \in \mathrm{Ker}(\mathcal{L})^\perp,~~ \intps \mathcal{L}_z(h(z)) \sdot g(z) \bdv \bd x \lesssim \|h(z)\|_{L^2_\Lambda}\|g^\perp(z)\|_{L^2_\Lambda}.
\end{align*}
Furthermore, for the bilinear operator $\Gamma$,
\begin{align}
\label{constant-gamma-z}
\begin{split}
&\vert \intps \nabla^s_x \Gamma_z(g,h)(z) \sdot f(z) \bdv \bd x \bd z \vert \lesssim \|(g(z),h(z))\|_{H^s_{x}  } \|(h(z), g(z))\|_{H^s_{\Lambda_x} }\|f^\perp(z)\|_{L^2_\Lambda},\\
&\vert \intps \nabla^j_x \nabla_v^i \Gamma_z(g,h)(z) \sdot f(z) \bdv \bd x \bd z \vert \\
& \lesssim \|(h(z),g(z))\|_{H^s} \|(g(z),h(z))\|_{H^s_{\Lambda}}\|f(z)\|_{L^2_\Lambda},~~i\ge 1,~~ s = i+j.
\end{split}
\end{align}

For the random case, the linear Boltzmann operator $\mathcal{L}$ is dependent of $z$, except the above assumption, for any $z \in \mathbb{I}_z$, we assume that the $\mathcal{L}$ satisfies all the assumption in Sec. \ref{sec-assump-determin}.

\section{Main results}
This section is devoted to stating the main results of this work.

\subsection{Main results of deterministic parts}
Now we state the result for the determinsitic case, the initial data are assumed to satisfy
\begin{align}
\label{initial-mean-nonlinear-c}
\begin{split}
  { \intps g_\epsilon(0) \m \bd v \bd x = 0,     \intps v g_\epsilon(0) \m \bd v \bd x =0,}\\
  \big( 3\intps (\tfrac{|v|^2}{3} - 1) g_\epsilon(0)\m \bd v \bd x +  \epsilon \|\nabla_x \phi_\epsilon(0)\|_{L^2}^2 \big)= 0.
  \end{split}
\end{align}
\begin{theorem}[Existence and Decay rates]
\label{theoremNonlinear-decay}
Under the assumptions on kernels in Sec. \ref{sec-assump-determin} and assumptions  \eqref{initial-mean-nonlinear-c} on the initial data, there exists some small enough  constant $c_{00}$ such that for each $0 < \epsilon \le 1$,  as long as
\[  \|g_\epsilon(0)\|_{H^s_x}^2 + \epsilon^2 \|\nabla_v   g_\epsilon(0)\|_{H^{s-1}}^2   \le c_{00}, \] equations \eqref{vpb-ns-scaling-ns-corrector} admit a unique solution  $(g_\epsilon, \nabla \phi_\epsilon)$ satisfying that there exist $\bar{c}_{00}>0$ and $\bar{c_0}>0$ such that
\begin{align}
\label{estTheorem-nonlinear-exp-decay}
\|g_\epsilon(t)\|_{H^s_x}^2 + \epsilon^2 \|\nabla_v g_\epsilon(t)\|_{H^{s-1}}^2 \le \bar{c}_{00} \exp(- \bar{c_0} t).
\end{align}
\end{theorem}
\begin{remark}
Compared to the Boltzmann case, the results are similar, but there exists new difficulty from the eletric field during the proof. Indeed, by the global conservation laws, for the Boltzmann equation, if the initial data satisfy
\[  { \intps g_\epsilon(0) \m \bd v \bd x = 0,\quad \intps v g_\epsilon(0) \m \bd v \bd x =0,} \quad
     \intps (\tfrac{|v|^2}{3} - 1) g_\epsilon(0)\m \bd v \bd x  \big)= 0. \]
Then for any $t>0$, one can obtain that
\[  { \intps g_\epsilon(t) \m \bd v \bd x = 0,\quad \intps v g_\epsilon(t) \m \bd v \bd x =0,} \quad
     \intps (\tfrac{|v|^2}{3} - 1) g_\epsilon(t)\m \bd v \bd x  \big)= 0. \]
So $\intt \bdp g_\epsilon \bd x = 0$ and the Poincare's inequality can be used. The linear VPB system enjoys the same properites (see Remark \ref{remark-linear}).  But for the nonlinear VPB system, even if the initial data satisfy \eqref{initial-mean-nonlinear-c}, by the global conservation laws of VPB system, one can obtain that
\begin{align*}
{ \intps g_\epsilon(t) \m \bd v \bd x = 0,     \intps v g_\epsilon(t) \m \bd v \bd x =0,}\\
     \intps (\tfrac{|v|^2}{3} - 1) g_\epsilon(t)\m \bd v \bd x =   -\tfrac{\epsilon}{3} \|\nabla_x \phi_\epsilon(t)\|_{L^2}^2.
\end{align*}
The above relation indicates that $\intt \bdp g_\epsilon \bd x \neq 0$ and we can not directly employ the Poincare's inequality to recover the dissipative estimates of the fluids part.  To overcome this, noticing that
\begin{align*}
\intt \bdp g_\epsilon \bd x = -\tfrac{\epsilon}{3} \|\nabla_x \phi_\epsilon(t)\|_{L^2}^2 \sdot (\tfrac{|v|^2}{3} - 1),
\end{align*}
we carefully split and estimate the mean part of $\bdp g_\epsilon$.
\end{remark}

\begin{remark}
Noticing that  for the linear VPB system, i.e.,
\[ \partial_t g_\epsilon + \reps \vdot g_\epsilon +  \repst \mathcal{L}(g_\epsilon)   - \tfrac{v}{\epsilon} \sdot \nabla \phi_\epsilon =0, \]
if we apply $\nabla_x^j \nabla_v^i$ to the above equation and $i \ge 2$, the last term $\tfrac{v}{\epsilon} \sdot \nabla \phi_\epsilon$ will vanish. This is why we split the whole estimate of $g_\epsilon$, i.e.,
$\|g_\epsilon(t)\|_{H^s_x}^2 + \epsilon^2 \|\nabla_v g_\epsilon(t)\|_{H^{s-1}}^2$ into three parts during the proof (see \eqref{est-norm-es-g-2}). This is the advantage of linearizing \eqref{vpb-ns-scaling} around $\dm$.  But since there exists a weight $\dm$ in the definition of $H^s$,  the price is that we can not obtain the $L^\infty$ estimate of $g_\epsilon$ from its $H^s$ norm.
\end{remark}

\begin{remark}
\label{remark-v-derivative}
Compared to the Boltzmann case, the derivative of $g_\epsilon$ with respect to $v$ brings more  difficulty for VPB system. On one hand, for the nonlinear system \eqref{vpb-ns-scaling-ns-corrector}, for the term $(\nabla_v g_\epsilon   -  {v}  g_\epsilon) \sdot \nabla\phi_\epsilon$, there already exists derivative with respect to $v$. While taking $s$-th derivative to \eqref{vpb-ns-scaling-ns-corrector}, we should employ integration by parts to avoid the presence of $\nabla_v^{s+1}g_\epsilon$. On the other hand, from the unifrom estimates \eqref{estTheorem-nonlinear-exp-decay}, as long as there exists derivative with respect to $v$, there is   a coefficient $\epsilon^2$. This means that there is no uniform estimates of $\nabla_v g_\epsilon$ and  makes the computation more difficult and complicated.

\end{remark}

For the convergence rate, the initial data in \cite{briant-2015-be-to-ns} are independent of $\epsilon$.     The initial data $g_\epsilon(0)$  are assumed to be well prepared, i.e.,  satisfy
\begin{align}
\label{assumpIni-Con}
g_\epsilon(0) = g_0(0) + \epsilon g_{\epsilon,1}(0),~~~~ g_0(0), g_{\epsilon,1} \in H^s,
\end{align}
where   $g_0$ belongs the kernel of $\mathcal{L}$ and  satisfies the Boussinesq relation.
\begin{theorem}
\label{main-results-convergence-rate}
Under the assumption of Theorem \ref{theoremNonlinear-decay} and \eqref{assumpIni-Con},  the solution $g_\epsilon$ constructed in Theorem \ref{theoremNonlinear-decay}   converges to $g$ in distributional sense with
\[ g = \rho + u\sdot v + \tfrac{\theta}{2}(|v|^2-3), \]
where $\rho, u, \theta$  are  sulution ( in $H^s_x$ space) to  the   initial boundary problem of the following system on torus
\begin{align*}
\begin{cases}
\partial_t u + u\sdot \nabla u - \nu \Delta u + \nabla P = \rho \nabla \theta,\\
\partial_t(\tfrac{3}{2}\theta - \rho) + u \sdot\nabla (\tfrac{3}{2}\theta - \rho) - \tfrac{5\kappa}{2}\Delta \theta =0,\\
\divg u = 0,~~ \Delta(\rho + \theta ) = \rho,~~E = \nabla(\rho + \theta),\\
u(0,x) = \mathbf{P} u_0(x),~~~~(\tfrac{3}{2}\theta - \rho)(0,x) = (\tfrac{3}{2}\theta_0(x) - \rho_0(x))
\end{cases}
\end{align*}
Furthermore, the solution sequence $ g_\epsilon  $ converge to $  g  $ with the following rate
\begin{align}
\| \int_0^\infty   (g_\epsilon - g)(s) \bd s \|_{H^\ell_x}^2 + \int_0^\infty \|   (g_\epsilon - g)(s) \|_{H^\ell_x}^2 \bd s  \lesssim  {\epsilon |\ln \epsilon|},~~~~ \ell \le s -2.
\end{align}
\end{theorem}
\begin{remark}
The proof of this theorem is based on the semi-group method. First, we need to caculate the spectrum of linear VPB operator very clearly. While $\epsilon =1$, its spectrum was clearly investigated in \cite{spectrum-arxiv-vpb}. But for the dimensionless linear VPB operator, the spectrum is dependent of $\epsilon$ and will be calculated again in Sec.\ref{sec-spec}.  As mentioned in Remark \ref{remark-v-derivative},  owing to the derivative with respect to $v$ ( term $\nabla_v g_\epsilon \sdot \nabla\phi_\epsilon$) and no uniform estimates of $\nabla_v g$, there exists new difficutly. The idea of dealing with such difficulty will be explained during the proof.
\end{remark}
\begin{remark}
\label{remark-loss}
Compared to the main result for the Boltzmann case, the solution $g_\epsilon$ belong to $H^s$ space, but we only obtain the convergence rate in $H^{s-2}_x$ space. There are two reasons. One is that the term induced by eletric field ($\nabla_v g_\epsilon \sdot \nabla\phi_\epsilon$) where there already exists order one derivative with respect to $v$. The other is explained in Remark \ref{remark-rate-1} and \ref{remark-rate-2}.
\end{remark}

\begin{remark}
In \cite{briant-2015-be-to-ns}, for the Boltzmann case, the convergence rate is obtained by the spectrum analysis based on Hilbert expansion.  We shall explain  where and why the Hilbert expansion is  Remark \ref{remark-hilbert} and our method to avoid Hilbert expansion in  Remark \ref{remark-rate-1} and \ref{remark-rate-2}.
\end{remark}

\begin{remark}
If the initial data do not statisfy \eqref{assumpIni-Con} (i.e. the general initial data), we can not obtain the convergence rates without using Hilbert expansion.  But by employing Hilbert expansion and Remark \ref{remark-hilbert}, we can obtain the same results.
\end{remark}

\subsection{Main results of random part}
While considering the random influence of the VPB equations, we consider its wellposedness and stability.

Before stating the main results, we introduce the notation to be used soon. Let
\[\|f\|_{H^s_{z}}^2 =  \sum\limits_{k=0}^s \sum\limits_{i + j =k}\|\nabla^i_x \nabla^j_v f\|_{L^2 L^{2\cap\infty}_z}^2  + \sum\limits_{k=0}^{s-1} \sum\limits_{i + j =k}\|\nabla^i_x \nabla^j_v  \partial_z f\|_{L^2 L^2_z}^2\]
and
\[\|f(t)\|_{L^2 L^{2\cap\infty}_z}^2 =   \intpsz f^2(t,x,v,z) \bdv \bd x \bd z +  \sup\limits_{z \in \mathbb{I}_z}\intps f^2(t,x,v,z) \bdv \bd x.  \]
The $H^s_{x,z}$ norm only contains the derivative with respect to $x$ and $z$. The detailed definiton of these notations can be found in Sec. \ref{sec-z}.

For the random part, we need the similar version assumption like \eqref{initial-mean-nonlinear-c}, i.e., for each $z \in \mathbb{I}_z$
\begin{align}
\label{initial-mean-nonlinear-recall-1}
\begin{split}
  { \intps g_\epsilon(0,x,v,z) \m \bd v \bd x = 0,     \intps v g_\epsilon(0,x,v,z) \m \bd v \bd x =0,}\\
  \big( 3\intps (\tfrac{|v|^2}{3} - 1) g_\epsilon(0,x,v,z)\m \bd v \bd x +  \epsilon \|\nabla_x \phi_\epsilon(0,z)\|_{L^2}^2 \big)= 0.
  \end{split}
\end{align}

\begin{theorem}[Existence and Stability]
\label{main-result-stable}
(I)(Existence).   Under the assumption on the kernel in Sec.\ref{sec-assump-ran} on  system \eqref{vpb-ns-scaling} and assumption \eqref{initial-mean-nonlinear-recall-1} on intial data ,   there exists constant $d_0>0$ such that for each $0<\epsilon\le 1$, as long as the its initial data $g_\epsilon(0)$ satisfy
\begin{align*}
\|(g_\epsilon,\nabla_x \phi_\epsilon)(0)\|_{H^s_{x,z}}^2  + \epsilon^2 \|\nabla_v g_\epsilon(0)\|_{H^{s-1}_z}^2 \le d_0,
\end{align*}
then system \eqref{vpb-ns-scaling-ns-corrector} adimit a global uniquess classical solution in $H^s_z$ space. furthermore the solutions enjoy exponential decay: for some $\bar{d}_0>0$ and $\bar{d}>0$ (all independent of $\epsilon$ while $\epsilon <1$) such that
\begin{align}
\label{estMainRexistence}
\|(g_\epsilon,\nabla_x \phi_\epsilon)(t)\|_{H^s_{x,z}}^2 + \epsilon^2 \|\nabla_v g_\epsilon(t)\|_{H^{s-1}_z}^2 \le \bar{d}_0 \exp(- \bar{d} t).
\end{align}
(II)(Stability). The solutions are stable in the following sense: Suppose that $g_{\epsilon,1}$ and $g_{\epsilon,2}$ are two solutions to VPB system with initial data $g_{\epsilon,1}(0)$ and $g_{\epsilon,2}(0)$ respectively in $H^s_z$ space for each $\epsilon>0$. For some small enough consant $s_0$,  as long as the initial data satisfy
\begin{align*}
 \|(g_{\epsilon,1}, \nabla_x \phi_{\epsilon,1})\|_{H^s_{x,z}}^2 + \epsilon^2 \|\nabla_v g_{\epsilon,1}\|_{H^{s-1}_z}^2 \le  s_0,~~~~\|(g_{\epsilon,2}, \nabla_x \phi_{\epsilon,2})\|_{H^s_{x,z}}^2 + \epsilon^2 \|\nabla_v g_{\epsilon,2}\|_{H^{s-1}_z}^2 \le  s_0,
\end{align*}
there exist some constant $\bar{s}_0$ and $\bar{s}$  such that
\begin{aligno}
&\|(g_{\epsilon,1} - g_{\epsilon,2},\nabla_x (\phi_{\epsilon,1} -\phi_{\epsilon,1}) )(t)\|_{H^\ell_{x}L^{2\cap \infty}_z }^2 + \epsilon^2 \|\nabla_v g_\epsilon(t)\|_{H^{\ell-1}L^{2\cap \infty}_z}^2 \\
& \le \bar{s}_0 \exp(-\bar{s} t) \left( \|(g_{\epsilon,1} - g_{\epsilon,2},\nabla_x (\phi_{\epsilon,1} -\phi_{\epsilon,1}) )(0)\|_{H^\ell_{x}L^{2\cap \infty}_z }^2 + \epsilon^2 \|\nabla_v g_\epsilon(0)\|_{H^{\ell-1}L^{2\cap \infty}_z}^2 \right), ~~ \ell \le s -1.
\end{aligno}
\end{theorem}
\begin{remark}
From \eqref{estshs}, the constant $s_0$ is smaller than $d_0$. Compared to the existence and similar to Remark \ref{remark-loss}, due to the infulence of the eletric field,  there exists derivative loss ($\ell \le s-1$).   The explanation can be found during \eqref{estStableN1} and \eqref{estDN11}.
\end{remark}

\begin{remark}
Compared to \cite{liujin-2018-kinetic} where the initial data of the Boltzmann equation belong to the orthogonal space of $\mathrm{Ker}\mathcal{L}$,  the initial  data  contain fluid parts for VPB system.
\end{remark}
For fluid limit, the initial data $g_\epsilon(0)$  are assumed to satisfy
\begin{align}
\label{assumpIni-Con-z}
g_\epsilon(0,x,z) \to  g_0(0,x,z)~~( \text{strongly in}~ H^s).
\end{align}

\begin{theorem}[Fluid limits]
\label{theorem-limit}
Under the assumption of Theorem \ref{main-result-stable} and \eqref{assumpIni-Con-z},  for $0< \epsilon \le 1$ and $s\ge 4$, let $(g_\epsilon,\nabla_\epsilon)$ be  $H^s_z$ solutions to system \eqref{vpb-ns-scaling-ns-corrector} constructed in Theorem \ref{main-result-stable}. Then, as $\epsilon \to 0$, in distributional sense,
\[ g_\epsilon \to \rho(t,x,z) + u(t,x,z)\cdot v+ \tfrac{\theta}{2}(|v|^2 -3),   \]
where $\rho,~u,~\theta $ belong to   $L^\infty((0,+\infty);H^s_{x,z})\cap C((0,+\infty) \times \mathbb{T}^3\times \mathbb{I}_z)\cap C((0,+\infty);H^{s-1}_{x,z}) $  and  are solutions to
\begin{align}
\label{nspf-z}
\begin{cases}
\partial_t u + u\sdot \nabla u - \nu \Delta u + \nabla P = \rho \nabla \theta,\\
\partial_t(\tfrac{3}{2}\theta - \rho) + u \sdot\nabla (\tfrac{3}{2}\theta - \rho) - \tfrac{5\kappa}{2}\Delta \theta =0,\\
\divg u = 0,~~ \Delta(\rho + \theta ) = \rho,~~E = \nabla(\rho + \theta).
\end{cases}
\end{align}
with initial data
\[ u(0,x,z) = \mathbf{P} u_0(x,z),~~~~(\tfrac{3}{2}\theta - \rho)(0,x,z) = (\tfrac{3}{2}\theta_0 - \rho_0).\]
Furthermore, $\mathbf{P}u_\epsilon(t),~(\tfrac{3}{2}\theta_\epsilon(t) - \rho_\epsilon(t)) \in C((0,+\infty);H^{s-1}_{x,z}),~~\forall \delta>0 $
\[ \mathbf{P} u_\epsilon \rightrightarrows u,~~\left(\tfrac{3}{5}\theta_\epsilon - \tfrac{2}{5}\rho_\epsilon \right) \rightrightarrows \left(\tfrac{3}{5}\theta - \tfrac{2}{5}\rho  \right),~~\text{in},~~C([\delta,+\infty) \times \mathbb{T}^3\times \mathbb{I}_z),~~\forall \delta>0.  \]
In addition, if the initial data are well prepared, i.e,
\[ g_0 \in \mathrm{Ker}\mathcal{L},~~ \divg u_0 =0,~~\Delta(\rho_0 + \theta_0) = \rho_0,  \]
we have
\[ \mathbf{P} u_\epsilon \rightrightarrows u,~~\left(\tfrac{3}{5}\theta_\epsilon - \tfrac{2}{5}\rho_\epsilon \right) \rightrightarrows \left(\tfrac{3}{5}\theta - \tfrac{2}{5}\rho  \right),~~\text{in},~~C([0,+\infty) \times \mathbb{T}^3\times \mathbb{I}_z).  \]
\end{theorem}
\begin{remark}
The resulting system \eqref{nspf-z} is dependent of $z$. Thus, we have prove the wellposedness of incompressible Navier-Stokes-Poisson system with random input.
\end{remark}
\begin{remark}
THe random variable $z$ are assumed to be one dimension. This results can be generalized to high dimension($N$ dimension) random variable if we assume the initial data belong to $H^sH^N_z$.

Furthermore, during the existence of solution, the $L^2$ bound of $\partial_z g_\epsilon$ is not necessary ( initial data belong to $H^sL^{2\cap\infty}_z$ is enough.). But while verifying the fluid limits, since the VPB system is nonlinear,  we need post more regularity of initial data on $z$.
\end{remark}

The whole Sec.\ref{sec-z} serves to proving this theorem.

\section{Analyse the linear equation}
\label{sec-linear-vpb}
\subsection{Difference of the linear equation compared to the linear Boltzmann equation}
\label{sec-difference}
The linear  VPB equation is
\begin{align}
\label{vpb-ns-scaling-ns-linear}
\begin{cases}
  \partial_t g_\epsilon + \reps \vdot g_\epsilon -  \repst \mathcal{L}(g_\epsilon)  - \tfrac{v}{\epsilon} \sdot \nabla \phi_\epsilon = 0, \\
 \Delta_x \phi_\epsilon =\int g_\epsilon \m \bd v.
\end{cases}
\end{align}

Furthermore, denoting
\[ \vpbg(g_\epsilon):=-\reps \vdot g_\epsilon  +  \repst \mathcal{L}(g_\epsilon)  + \tfrac{v}{\epsilon} \sdot \nabla \phi_\epsilon , \]
different with the $$\mathcal{G}_{\mathrm{BE},\epsilon+}(g_\epsilon):=-\reps \vdot g_\epsilon  +  \repst \mathcal{L}(g_\epsilon)$$ defined in \cite{briant-2015-be-to-ns},
$\vpbg$ is not orthogonal to the fluid parts. Indeed, the new term $\tfrac{v}{\epsilon} \sdot \nabla \phi_\epsilon   $ generated by the electric field lies in the kernel space of $\mathcal{G}$.

The Theorem 2.4 of \cite{briant-2015-be-to-ns} is based on the fact that   $\mathcal{G}$ is orthogonal to the fluids parts. Furthermore, the main difference of the norm $\mathcal{H}_\epsilon^s$ and $\mathcal{H}_{\epsilon\perp}^s$ is on the derivative with respect to $v$.  There exists a coefficient $\epsilon^2$ before the derivative of $g_\epsilon$ with respect to $v$  in $\mathcal{H}_\epsilon^s$ but not in  $\mathcal{H}_{\epsilon\perp}^s$. As $\epsilon$ tends to zero, no useful estimates  containing derivatives of $g_\epsilon$ with respect to $v$ can be obtained. More details can be found in Remark \ref{remark-linear}.

\subsection{Estimates}
Assume that the initial data satisfy for each $\epsilon>0$
\begin{align}
\label{mean-zero-initial}
\begin{split}
\intps g_\epsilon(0) \m \bd v \bd x = \intps (\tfrac{|v|^2}{3} - 1) g_\epsilon(0)\m \bd v \bd x = 0, ~~\intps v g_\epsilon(0) \m \bd v \bd x =0,
\end{split}
\end{align}
then by simple calculation, we can deduce that for any time $t>0$
\begin{align}
\label{mean-zero-all-time}
\begin{split}
\intps g_\epsilon(t) \m \bd v \bd x = \intps (\tfrac{|v|^2}{3} - 1) g_\epsilon(t)\m \bd v \bd x = 0, ~~\intps v g_\epsilon(t) \m \bd v \bd x =0.
\end{split}
\end{align}
This means that the solution of linear equations conserves its mean values on the phase space. This will be very useful during employing the Poincare inequality.

The macroscopic density, velocity and temperature are defined as usually,
\begin{align*}
\rho_\epsilon = \intv g_\epsilon \bdv,~~u_\epsilon = \intv vg_\epsilon \bdv, ~~ \theta_\epsilon = \intv(\tfrac{|v|^2}{3} - 1) g_\epsilon \bdv.
\end{align*}
Define
\begin{align*}
\mathfrak{E}_\epsilon^s(t)&:= \sum\limits_{k=0}^{s-1}\bigg( \lambda_1 \|(\nabla^k g_\epsilon,\nabla^k \nabla_x \phi_\epsilon)\|_{L^2}^2 + \lambda_2 \|(\nabla_x \nabla^k g_\epsilon, \Delta \nabla_x^k \phi_\epsilon)\|_{L^2}^2 \bigg)\\
& + \sum\limits_{k=0}^{s-1}\bigg( \lambda_3 \epsilon^2 ( \|\nabla^k \nabla_v g_\epsilon\|_{L^2}^2 - \|\nabla^k \nabla_x \phi_\epsilon\|_{L^2}^2) + 2 \lambda_4 \epsilon \intps \nabla_x \nabla^k_x g_\epsilon\sdot\nabla_v \nabla_x^k g_\epsilon \bdv \bd x \bigg),
\end{align*}
where $\lambda_i (i=1\cdots4)$ will be given later such that $\mathfrak{E}_\epsilon^1(t)$  is  equivalent to
\[ \mathcal{E}_\epsilon^s(t):= \|(g_\epsilon(t),\nabla_x \phi_\epsilon)\|_{H^s_x}^2 + \epsilon^2 \|\nabla_v g_\epsilon(t)\|_{H^{s-1}}^2. \]

\begin{lemma}
\label{lemma-linear-decay}
Under the assumptions on kernels in Sec. \ref{sec-assump-determin} and assumptions  \eqref{mean-zero-initial} on the initial data, there exists some small enough  constant $c_0$ such that as long as   \[  \|g_\epsilon(0)\|_{H^s_x}^2 + \epsilon^2 \|\nabla_v   g_\epsilon(0)\|_{H^{s-1}}^2   \le  c_0, \] equations \eqref{vpb-ns-scaling-ns-linear} admits a solution  $(g_\epsilon, \nabla \phi_\epsilon)$ satisfying for any $t>0$
\[ \|(g_\epsilon(t),\nabla_x \phi_\epsilon)\|_{H^s_x}^2 + \epsilon^2 \|\nabla_v g_\epsilon(t)\|_{H^{s-1}}^2 < +\infty. \]
Furthermore, there exist $\tilde{c}_0>0$ and $\tilde{c}>0$ (all independent of $\epsilon$ while $\epsilon <1$) such that
\begin{align}
\label{est-lemma-linear-exp-decay}
\|(g_\epsilon(t),\nabla_x \phi_\epsilon)\|_{H^s_x}^2 + \epsilon^2 \|\nabla_v g_\epsilon(t)\|_{H^{s-1}}^2 \le \tilde{c}_0 \exp(- \tilde{c} t).
\end{align}
\end{lemma}
\begin{remark}
\label{remark-linear}
There exists a coefficient $\epsilon^2$ before $\|\nabla_v g_\epsilon(t)\|_{L^2}^2$. Furthermore, the similar version to \cite[Theorem 2.4]{briant-2015-be-to-ns} where there is no $\epsilon^2$ before the derivative of $g_\epsilon$ with respect to $v$ can not be established for \eqref{vpb-ns-scaling-ns-linear}. Indeed, for the linear Boltzmann equation, if the initial data do not contain fluid part, then solutions preserve this property. Thus the underlined term in \eqref{est-linear-derivative-v-1-all} can be absorbed by employing \eqref{est-linear-vpb-ns-basic-l2} even though there exists coefficient $\repst$ before the underlined term. Therefore, there is no $\repst$ before $\|\nabla_v g_\epsilon(t)\|_{L^2}^2$ in the norm used in \cite[Theorem 2.4]{briant-2015-be-to-ns}.   But for linear VPB equation, according Section \ref{sec-difference},  as though the initial data are orthogonal to the fluid part, the solutions always contains fluid parts.  The underlined term can not be controlled while $\epsilon$ is very small. This is why the $\epsilon^2$ is  necessary.
\end{remark}
\begin{remark}
\label{remark-linear-difference-with-nonlinear}
{ For the nonlinear system, there exist new difficulties. Specially, \eqref{mean-zero-all-time} does not hold while the initial data satisfies \eqref{mean-zero-initial}. The Poincare's inequality can not be used. Thus we can not directly recover the $L^2$ of $\bdp g_\epsilon$ from $\nabla_x g_\epsilon$.     }
\end{remark}

\begin{proof}
Here, we only take $s=1$ for example. For any integer $s\ge 2$, after applying $\nabla_x^{s-1}$ to \eqref{vpb-ns-scaling-ns-linear},  the method are  similar.  Multiplying the \eqref{vpb-ns-scaling-ns-linear} by $g_\epsilon \m $, integrating over the phase space, it follows that
\begin{align}
\label{est-linear-vpb-ns-basic-1}
\begin{split}
\dt \|g_\epsilon\|_{L^2}^2 + \repsta\|g_\epsilon - \bdp g_\epsilon\|_{L^2_\Lambda}^2 - \reps \int_{\mathbb{T}} (\nabla_x \phi_\epsilon)\sdot ( \int_{\mathbb{R}^3} v g \m \bd v )\bd x \le 0.
\end{split}
\end{align}
To close the above equation, the term with coeffient $\reps$ should be carefully taken care. In additions, noticing that the resulting equation after multiplying the first equation of \eqref{vpb-ns-scaling-ns-linear} by $\dm$ and then  integrating with respect to velocity is
\begin{align*}
\epsilon \partial_t \int g_\epsilon \m \bd v + \divg \int_{\mathbb{R}^3} v g_\epsilon \m \bd v =0.
\end{align*}
therefore,
\begin{align}
\label{est-linear-basic-l2-d-electric}
\begin{split}
- \reps \int_{\mathbb{T}} (\nabla_x \phi_\epsilon)\sdot ( \int_{\mathbb{R}^3} v g \m \bd v )\bd x & =  \reps \int_{\mathbb{T}} ( \phi_\epsilon)\sdot \divg ( \int_{\mathbb{R}^3} v g_\epsilon \m \bd v )\bd x \\
& = - \int_{\mathbb{T}} \phi_\epsilon \sdot ( \partial_t \int g_\epsilon \m \bd v  )\bd x\\
& = -\int_{\mathbb{T}} \phi_\epsilon \sdot \Delta \partial_t \phi_\epsilon \bd x\\
& = \dt \|\nabla \phi_\epsilon\|_{L^2}^2.
\end{split}
\end{align}
Thus, together with \eqref{est-linear-vpb-ns-basic-1}, we can conclude that
\begin{align}
\label{est-linear-vpb-ns-basic-l2}
\dt ( \|g_\epsilon\|_{L^2}^2 + \|\nabla \phi_\epsilon\|_{L^2}^2)(t) +  \repsta\|g_\epsilon - \bdp g_\epsilon\|_{L^2_\Lambda}^2  \le 0.
\end{align}

Similarly, the $L^2$ estimates of high order  derivatives with respect to $x$ can be obtained.
\begin{align}
\label{est-linear-vpb-ns-x-derivative-l2}
\dt ( \|\nabla^s_x g_\epsilon\|_{L^2}^2 + \|\nabla \nabla^s_x \phi_\epsilon\|_{L^2}^2)(t) +  \repsta\|\nabla^s_x g_\epsilon - \bdp (\nabla^s_x g_\epsilon)\|_{L^2_\Lambda}^2 \le 0.
\end{align}
Now we turn to the derivative with respect to $v$ of $g_\epsilon$.
 \begin{align}
\label{linear-derivative-v}
\partial_t \nabla_v g_\epsilon + \reps \vdot \nabla_v g_\epsilon + \reps \nabla_x g_\epsilon-  \repst \nabla_v \mathcal{L}(g_\epsilon)  - \tfrac{1}{\epsilon}   \nabla \phi_\epsilon    = 0.
\end{align}
Multiplying \eqref{linear-derivative-v} by $\nabla_v g_\epsilon \m $ and then integrating over the phase space, it follows that
\begin{align}
\label{est-linear-derivative-v-0}
\begin{split}
&\dt( \|\nabla_v g_\epsilon\|_{L^2}^2)  + \reps \intps \nabla_x g_\epsilon\sdot\nabla_v g_\epsilon \m \bd v \bd x \\
&- \repst \intps \nabla_v \mathcal{L}(g_\epsilon) \sdot \nabla_v g_\epsilon \m \bd v \bd x -\reps \intps \nabla_x \phi_\epsilon \sdot \nabla_v g_\epsilon \m \bd b \bd x  =0.
\end{split}
\end{align}
Recalling
\[ \mathcal{L}= \mathcal{K} - {\Lambda},\]
thus
\begin{align*}
- \repst \intps \nabla_v \mathcal{L}(g_\epsilon) \sdot \nabla_v g_\epsilon \m\bd v \bd x & = - \repst \intps \nabla_v \mathcal{K}(g_\epsilon) \sdot \nabla_v g_\epsilon \m \bd v \bd x\\
& \quad +  \repst \intps \nabla_v \Lambda(g_\epsilon) \sdot \nabla_v g_\epsilon \m\bd v \bd x.
\end{align*}
According to \eqref{constant-a3-a4}, we can obtain that
\begin{align*}
a_3 \|\nabla_v g_\epsilon\|_{L^2_\Lambda}^2 - a_4\|g_\epsilon\|_{L^2}^2 \le \intps \nabla_v \Lambda(g_\epsilon) \sdot \nabla_v g_\epsilon \m \bd v \bd x,
\end{align*}
and for any $\delta>0$ there exists $C_\delta$ such that
\begin{align*}
-\delta\|\nabla_v g_\epsilon\|_{L^2_\Lambda}^2 - C_\delta\|g_\epsilon\|_{L^2}^2 \le -\intps \nabla_v \mathcal{K}(g_\epsilon) \sdot \nabla_v g_\epsilon \m \bd v \bd x.
\end{align*}
Summing the above two inequalities up, we can infer that
\begin{align}
\label{est-linear-derivative-v-term-dissipation}
 \tfrac{(a_3 -\delta)}{\epsilon^2} \|\nabla_v g_\epsilon\|_{L^2_\Lambda}^2 - \underline{\tfrac{(a_4 + C_\delta)}{\epsilon^2}\|g_\epsilon\|_{L^2}^2}  \le   - \repst \intps \nabla_v \mathcal{L}(g_\epsilon) \sdot \nabla_v g_\epsilon \bdv \bd x.
\end{align}
While $\delta$ is small enough and $6\delta <<a_3$, $a_3 - \delta$ is strictly positive. As we will show, it can absorb other term. Indeed, for the $\delta>0$, by Young's inequality, there exists some constant $C_{\delta,1}$ such that
\begin{align*}
\reps \big\vert\intps \nabla_x g_\epsilon\sdot\nabla_v g_\epsilon \bdv \bd x\big\vert \le C_{\delta,1}\|\nabla_x g_\epsilon\|_{L^2}^2 + \tfrac{\delta}{\epsilon^2}\|\nabla_v g_\epsilon\|_{L^2_\Lambda}^2.
\end{align*}

For the   last term,
\begin{align*}
-\reps \intps \nabla_x \phi_\epsilon \sdot \nabla_v g_\epsilon \bdv \bd x & = \tfrac{1}{\epsilon}\intps \nabla_x \phi_\epsilon \sdot  v \sdot g_\epsilon   \bdv \bd x = - \tfrac{\mathrm{d}}{2\mathrm{d}t} \|\nabla \phi_\epsilon\|_{L^2_x}^2,
\end{align*}
where we have used \eqref{est-linear-basic-l2-d-electric}.

Summing the relative inequalities up, then \eqref{est-linear-derivative-v-0} turns to
\begin{align}
\label{est-linear-derivative-v-1-all}
\begin{split}
& \dt( \|\nabla_v g_\epsilon\|_{L^2}^2 + \|\nabla_x \phi_\epsilon\|_{L^2_x}^2) + \tfrac{(a_3 -3\delta)}{\epsilon^2} \|\nabla_v g_\epsilon\|_{L^2_\Lambda}^2 \\
& \le   \underline{\tfrac{(a_4 + C_\delta)}{\epsilon^2}\|g_\epsilon\|_{L^2}^2} +  C_{\delta,1}\|\nabla_x g_\epsilon\|_{L^2}^2.
\end{split}
\end{align}
Multiplying \eqref{linear-derivative-v} by $\nabla_x g_\epsilon\m$ and then integrating over $\mathbb{T}\times\mathbb{R}^3$, it follows that
\begin{align}
\label{est-linear-derivative-v-x-0}
\begin{split}
&\dt( \intps \nabla_x g_\epsilon \sdot \nabla_v g_\epsilon \bdv \bd x) + \reps \intps \nabla_v(\vdot g_\epsilon)\sdot\nabla_x g_\epsilon \bdv \bd x\\
&  - \repst \intps \nabla_v \mathcal{L}(g_\epsilon) \sdot \nabla_x g_\epsilon \bdv \bd x -\reps \intps \nabla_x \phi_\epsilon \sdot \nabla_x g_\epsilon  \bdv \bd x   =0.
\end{split}
\end{align}
By simple computation,
\begin{align*}
\reps \intps \nabla_v(\vdot g_\epsilon)\sdot\nabla_x g_\epsilon \bd v \bd x = \tfrac{1}{2\epsilon}\|\nabla_x g_\epsilon\|_{L^2}^2 + \tfrac{1}{2\epsilon}\|\vdot g_\epsilon\|_{L^2}^2.
\end{align*}
For the term with coefficient $\repst$,
\begin{align*}
- \repst \intps \nabla_v \mathcal{L}(g_\epsilon) \sdot \nabla_x g_\epsilon \bd v \bd x & = - \repst \intps  \mathcal{L}(\nabla_x g_\epsilon) \sdot \nabla_v g_\epsilon \bd v \bd x\\
& = - \repst \intps  \mathcal{L}((\nabla_x g_\epsilon)^\perp) \sdot \nabla_v g_\epsilon \bd v \bd x\\
& \le \tfrac{C_u}{\epsilon^2} \int_{\mathbb{T}^3}\|(\nabla_x g_\epsilon)^\perp\|_{L^2_{v,\Lambda}}\|\nabla_v g_\epsilon\|_{L^2_{v,\Lambda}}\bd x\\
& \le \tfrac{C_u^2}{2a_6\epsilon^2}\|(\nabla_x g_\epsilon)^\perp\|_{L^2_{\Lambda}}^2  + \tfrac{a_6}{2\epsilon^2}\|\nabla_v g_\epsilon \|_{L^2_{\Lambda}}^2,
\end{align*}
where $a_6$ will be given later. For the terms in the left hand of \eqref{est-linear-derivative-v-x-0},
\begin{align*}
-\reps \intps \nabla_x \phi_\epsilon \sdot \nabla_x g_\epsilon  \bdv \bd x   = \tfrac{1}{\epsilon} \|\rho_\epsilon\|_{L^2}^2.
\end{align*}
Thus, for \eqref{est-linear-derivative-v-x-0}, in the light of the above three inequalities (equalities), we obtain that
\begin{align}
\label{est-linear-nabla-x-dissipation}
\begin{split}
& \dt( \intps \nabla_x g_\epsilon \sdot \nabla_v g_\epsilon \bdv \bd x) + \tfrac{1}{2\epsilon} \big( \|\nabla_x g_\epsilon\|_{L^2}^2 + \|\vdot g_\epsilon\|_{L^2}^2 + 2 \|\rho_\epsilon\|_{L^2}^2\big) \\
& \le \tfrac{C_u^2}{2a_6\epsilon^2}\|(\nabla_x g_\epsilon)^\perp\|_{L^2_{\Lambda}}^2  + \tfrac{a_6}{2\epsilon^2}\|\nabla_v g_\epsilon \|_{L^2_{\Lambda}}^2.
\end{split}
\end{align}
 According to  \eqref{est-linear-vpb-ns-basic-l2}, \eqref{est-linear-vpb-ns-x-derivative-l2}, \eqref{est-linear-derivative-v-1-all} and \eqref{est-linear-nabla-x-dissipation},
\begin{align}
\label{est-linear-h1-all-0}
\begin{split}
 \dt \mathfrak{E}_\epsilon^1(t) & + \tfrac{\lambda_1a_2}{\epsilon^2}\|g_\epsilon^\perp\|_{L^2}^2 + \tfrac{\lambda_2 a_2}{\epsilon^2}\|(\nabla_x g_\epsilon)^\perp\|_{L^2}^2 +  \lambda_3 (a_3 -3\delta) \|\nabla_v g_\epsilon\|_{L^2_\Lambda}^2 \\
& \qquad +   {\lambda_4}  \big( \|\nabla_x g_\epsilon\|_{L^2}^2 + \|\vdot g_\epsilon\|_{L^2}^2 + 2 \|\rho_\epsilon\|_{L^2}^2\big)\\
& \le \lambda_3 \underline{{(a_4 + C_\delta)}\|\bdp g_\epsilon\|_{L^2}^2} + \lambda_3 C_{\delta,1}\epsilon^2\|\nabla_x g_\epsilon\|_{L^2}^2 +  \lambda_3 C_{\delta, 2} \epsilon^2 \|\nabla_x \phi_\epsilon\|_{L^2}^2\\
& \qquad + \lambda_4 \tfrac{C_u^2}{a_6} \|(\nabla_x g_\epsilon)^\perp\|_{L^2_{\Lambda}}^2  + \lambda_4 a_6\|\nabla_v g_\epsilon \|_{L^2_{\Lambda}}^2 + \lambda_3 {{(a_4 + C_\delta)}\|g_\epsilon^\perp\|_{L^2}^2}.
\end{split}
\end{align}
The way of choosing $\lambda_i$ ($i=1,\cdots,4$) is to let the right hand of \eqref{est-linear-h1-all-0} be absorbed by the``dissipation" term of the left hand.

Recalling that $\bdp g_\epsilon $ has zero mean value on torus (by \eqref{mean-zero-all-time}),  by Poincare's inequality, there exists some constant $a_5>0$ such that
\begin{align}
\label{constant-a5}
& \|\bdp g_\epsilon\|_{L^2}^2 \le a_5 \|\nabla_x \bdp g_\epsilon\|_{L^2}^2 \le a_5 \|\nabla_x g_\epsilon\|_{L^2}^2,~~\text{and},~~ \|\nabla_x \phi_\epsilon\|_{L^2}^2 \le a_5 \|\Delta \phi_\epsilon\|_{L^2}^2 = a_5 \|\rho_\epsilon\|_{L^2}^2.
\end{align}

\begin{align}
\label{est-linear-all-h1}
\begin{split}
& \dt \mathfrak{E}_\epsilon^1(t) + \tilde\lambda_1\|g_\epsilon^\perp\|_{L^2}^2 + \tilde\lambda_2\|(\nabla_x g_\epsilon)^\perp\|_{L^2}^2  + \lambda_4 \|\vdot g_\epsilon\|_{L^2}^2 \\
& \qquad +  \lambda_5 \|\nabla_v g_\epsilon\|_{L^2_\Lambda}^2  +  \lambda_6\|\nabla_x g_\epsilon\|_{L^2}^2 +  \lambda_7\|\rho_\epsilon\|_{L^2_x}^2  \le 0.
\end{split}
\end{align}
where
\begin{align*}
\tilde\lambda_1 = \big[ \tfrac{\lambda_1a_2}{\epsilon^2} - \lambda_3(a_4 + C_\delta)\big],\\
\tilde\lambda_2 = \big[\tfrac{\lambda_2a_2}{\epsilon^2} - \lambda_4 \tfrac{C_u^2}{a_6} - \lambda_3 a_5 (a_4 + C_\delta)\big],\\
\lambda_5 = \big[\lambda_3 (a_3 -3\delta) - \lambda_4 a_6\big],\\
\lambda_6 = \big[ \lambda_4 - \lambda_3 C_{\delta,1}\epsilon^2  -  a_5(a_4 + C_\delta)\lambda_3 \big],\\
\lambda_7 = \big[2\lambda_4 - \lambda_3 a_5 C_{\delta, 2} \epsilon^2 \big].
\end{align*}

For any fixed $\lambda_4 \ge 2\delta a_5 + 2 a_5 $,  we can choose small enough $\lambda_3$ first. Noticing that
\begin{align*}
& \qquad \tfrac{2\lambda_4}{a_5}\|\nabla_x \phi_\epsilon\|_{L^2}^2 \le 2 \lambda_4 \|\rho_\epsilon\|_{L^2}^2,\\
& \lambda_3  {{(a_4 + C_\delta)}\|\bdp g_\epsilon\|_{L^2}^2} + \lambda_3 C_{\delta,1}\epsilon^2\|\nabla_x g_\epsilon\|_{L^2}^2 \le \lambda_3 \big(a_5 (a_4 + C_\delta) + C_{\delta,1}\epsilon^2 \big) \|\nabla_x g_\epsilon\|_{L^2}^2.
\end{align*}
Choosing $\lambda_3$ such that,
\[ \tfrac{1}{2}\lambda_4 \ge  \lambda_3 \big(a_5 (a_4 + C_\delta) + C_{\delta,1}\epsilon^2 \big) + \delta, ~~\text{and}, \tfrac{\lambda_4}{a_5} \ge  \lambda_3 C_{\delta, 2} \epsilon^2 + \delta.~~\]
For $\lambda_1$, $\lambda_2$ and $a_6$, they can be chosen as follows.
\[ \lambda_4 a_6 = \tfrac{1}{2}\lambda_3 (a_3 -3\delta),~~\tfrac{\lambda_1- \lambda_3 \epsilon^2 - 1}{\epsilon^2} \ge \lambda_3 {(a_4 + C_\delta)}  + \delta,~~\tfrac{\lambda_2- \lambda_4 }{\epsilon^2} \ge  \lambda_4 \tfrac{C_u^2}{a_6} + \delta . \]
Finally, setting $ c_d = \min\{\delta, \tfrac{1}{2}\lambda_3 (a_3 -3\delta)\}$, we finally deduce that
\begin{align}
\label{est-linear-es-1-1}
\dt \mathfrak{E}_\epsilon^1(t) + c_d (\|g_\epsilon\|_{H^1}^2 + \tfrac{1}{\epsilon^2} \|(g_\epsilon)^\perp\|_{H^1}^2 + \|\nabla_v g_\epsilon\|_{L^2}^2+ \|\rho_\epsilon\|_{H^1}^2 ) \le 0.
\end{align}
Besides, while $\epsilon<1$, $\lambda_i, c_d$ can be chosen independently of $\epsilon$. Furthermore, there exist $c_l$ and $c_u$ such that
\[ c_l \mathcal{E}_\epsilon^1(t) \le \mathfrak{E}_\epsilon^1 \le  c_u \mathcal{E}_\epsilon^1(t). \]
By Gr\"onwall's inequality, we complete the proof  for $s=1$.

For $s \ge 2$, define $\mathfrak{E}_\epsilon^1(t)$
\begin{align}
\label{est-norm-es-g-2}
\mathfrak{E}_\epsilon^s(t) = \mathfrak{E}_{\epsilon,1}^s(t) + \mathfrak{E}_{\epsilon,2}^s(t),
\end{align}
where
\begin{align*}
\mathfrak{E}_{\epsilon,1}^s(t)\approx\|g_\epsilon\|_{H^s}^2+ \sum\limits_{i = 1, i + j =s}\epsilon^2\|\nabla^i_v \nabla_x^j g_\epsilon\|_{L^2}^2,\\
\mathfrak{E}_{\epsilon,2}^s(t)=\sum\limits_{i \ge 2, i + j =s}\epsilon^2\|\nabla^i_v \nabla_x^j g_\epsilon\|_{L^2}^2.
\end{align*}
By the similar tricks of obtaining \eqref{est-linear-es-1-1}, we can deduce that
\begin{align}
\label{est-linear-es-1-s}
\dt \mathfrak{E}_{\epsilon,1}^s(t) + c_d (\|g_\epsilon\|_{H^s}^2 + \tfrac{1}{\epsilon^2} \|(g_\epsilon)^\perp\|_{H^s}^2 + \|\nabla_v g_\epsilon\|_{H^{s-1}}^2+ \|\nabla_x \phi\|_{H^s}^2 ) \le 0.
\end{align}
While the order of derivative to $v$ is greater than $2$ (the second part in \eqref{est-norm-es-g-2}), according to \eqref{linear-derivative-v},  the effect of the electric field vanishes, i.e.,
 \begin{align}
\label{linear-derivative-v-s}
\partial_t \nabla_x^j \nabla_v^i g_\epsilon + \reps \nabla_x^j \nabla_v^i (\vdot  g_\epsilon)    -  \repst \nabla_x^j \nabla_v^i \mathcal{L}(g_\epsilon)     = 0,~~ i \ge 2.
\end{align}
Furthermore, we do not need to split fluid parts out like \eqref{est-linear-nabla-x-dissipation}. Indeed, by the assumption \eqref{assump-h1-h} on the collision kernel, with the similar method of deducing \eqref{est-linear-derivative-v-1-all},  it follows that,
\begin{align}
\label{est-linear-es-2-s-0}
\begin{split}
& \qquad \epsilon^2 \dt\|\nabla_x^j \nabla_v^i g_\epsilon\|_{L^2}^2 +  {(a_3 -3\delta)}  \|\nabla_x^j \nabla_v^i g_\epsilon\|_{L^2}^2  \\
& \le   {(a_4 + C_\delta)} \|\nabla_x^j \nabla_v^{i-1} g_\epsilon\|_{L^2}^2 +  C_{\delta,1} \epsilon^2 \| \nabla_x^{j +1} \nabla_v^{i-1} g_\epsilon\|_{L^2}^2.
\end{split}
\end{align}
If $s=2$, we can conclude that
\begin{align*}
\dt \mathfrak{E}_{\epsilon,2}^2(t) + \tfrac{{(a_3 -3\delta)}}{\epsilon^2} \mathfrak{E}_{\epsilon,2}^2(t) \le (a_4 + C_\delta +  C_{\delta,1} \epsilon^2) \|g_\epsilon\|_{H^1_{x,v}}^2.
\end{align*}
For $s \ge 3$, $\mathfrak{E}_{\epsilon,2}^s(t)$ can be split into two cases: $i=2$ can $ i \ge 3$, i.e.,
\begin{align*}
\mathfrak{E}_{\epsilon,2}^s(t) = \mathfrak{E}_{\epsilon,2,1}^s(t) + \mathfrak{E}_{\epsilon,2,2}^s(t),
\end{align*}
where
\begin{align*}
\mathfrak{E}_{\epsilon,2,1}^s(t)=\sum\limits_{i = 2, i + j =s}\epsilon^2\|\nabla^i_v \nabla_x^j g_\epsilon\|_{L^2}^2,~~\mathfrak{E}_{\epsilon,2,2}^s(t) = \sum\limits_{i \ge 3, i + j =s}\epsilon^2\|\nabla^i_v \nabla_x^j g_\epsilon\|_{L^2}^2.
\end{align*}
According to \eqref{est-linear-es-2-s-0},  we can infer that
\begin{align}
\label{est-e21}
\dt \mathfrak{E}_{\epsilon,2,1}^s(t) + \tfrac{{(a_3 -3\delta)}}{\epsilon^2} \mathfrak{E}_{\epsilon,2,1}^s(t) \le (a_4 + C_\delta +  C_{\delta,1} \epsilon^2) \mathfrak{E}_{\epsilon,1}^s(t).
\end{align}
According to \eqref{est-linear-es-2-s-0} and Poincare's inequality,  noticing that
\begin{align}
\label{est-linear-es-2-s-0-vs}
\begin{split}
& \qquad \epsilon^2 \dt\|\nabla_x^j \nabla_v^i g_\epsilon\|_{L^2}^2 +  {(a_3 -3\delta)}  \|\nabla_x^j \nabla_v^i g_\epsilon\|_{L^2}^2  \\
& \le   {(a_4 + C_\delta)} \|\nabla_x^j \nabla_v^{i-1} g_\epsilon\|_{L^2}^2 +  C_{\delta,1} \epsilon^2 \| \nabla_x^{j +1} \nabla_v^{i-1} g_\epsilon\|_{L^2}^2\\
& \le  c_9 \| \nabla_x^{j +1} \nabla_v^{i-1} g_\epsilon\|_{L^2}^2,
\end{split}
\end{align}
where $c_9 =  \bigg( a_5 {(a_4 + C_\delta)} +  C_{\delta,1} \epsilon^2\bigg) $. Thus,
\begin{align*}
& \epsilon^2 \dt\| \nabla_v^s g_\epsilon\|_{L^2}^2 +  {(a_3 -3\delta)}  \| \nabla_v^s g_\epsilon\|_{L^2}^2   \le  c_9 \| \nabla_x  \nabla_v^{s-1} g_\epsilon\|_{L^2}^2,\\
& \epsilon^2 \dt\| \nabla_x \nabla_v^{s-1} g_\epsilon\|_{L^2}^2 +  {(a_3 -3\delta)}  \| \nabla_x \nabla_v^{s-1} g_\epsilon\|_{L^2}^2   \le  c_9 \| \nabla_x^2 \nabla_v^{s-2} g_\epsilon\|_{L^2}^2,\\
& \qquad \qquad \qquad \qquad \hspace{4cm}\vdots \\
& \epsilon^2 \dt\| \nabla_x^{s-3}\nabla_v^{3} g_\epsilon\|_{L^2}^2 +  {(a_3 -3\delta)}  \| \nabla_x^{s-3} \nabla_v^{3} g_\epsilon\|_{L^2}^2   \le  c_9 \| \nabla_x^{s-2} \nabla_v^{2} g_\epsilon\|_{L^2}^2.
\end{align*}
From the above inequalities, we can obtain that
\begin{align}
\label{est-es22}
 c_f \cdot \dt \mathfrak{E}_{\epsilon,2,2}^s(t) + \tfrac{\delta}{\epsilon^2}  \mathfrak{E}_{\epsilon,2,2}^s(t) \le (a_3 - 4\delta) \|\nabla^{s-2}_x \nabla_v^2 g_\epsilon\|_{L^2}^2,
\end{align}
where $c_f = \tfrac{a_3 - 4\delta}{c_9}$.
Combing \eqref{est-e21} and \eqref{est-es22}, it follows that
\begin{align}
\label{est-linear-e2}
\dt \bigg(  c_f \cdot   \mathfrak{E}_{\epsilon,2,2}^s(t) +  \mathfrak{E}_{\epsilon,2,1}^s(t)\bigg) + \tfrac{\delta}{\epsilon^2}  \mathfrak{E}_{\epsilon,2}^s(t) \le (a_4 + C_\delta +  C_{\delta,1} \epsilon^2) \mathfrak{E}_{\epsilon,1}^s(t).
\end{align}
In the light of \eqref{est-linear-es-1-s}, $c_e = \tfrac{(a_4 + C_\delta +  C_{\delta,1} \epsilon^2) + \delta}{c_d}$, we can obtain that
\begin{align}
\dt \bigg( c_e \cdot   \mathfrak{E}_{\epsilon,1}^s(t) +  c_f \cdot   \mathfrak{E}_{\epsilon,2,2}^s(t) +  \mathfrak{E}_{\epsilon,2,1}^s(t)\bigg) +   \tfrac{c_e}{\epsilon^2} \|(g_\epsilon)^\perp\|_{H^s}^2  + c_e\|\rho_\epsilon\|_{H^s}^2  + \delta\|g_\epsilon\|_{H^s_{x,v}}^2  \le 0.
\end{align}
Obviously, $c_e \cdot   \mathfrak{E}_{\epsilon,1}^s(t) +  c_f \cdot   \mathfrak{E}_{\epsilon,2,2}^s(t) +  \mathfrak{E}_{\epsilon,2,1}^s(t)$ is equivalent to $\mathfrak{E}_\epsilon^s$, that is to say, one can reselect $c_l$ and $c_u$ such that,
\begin{align}
\label{constant-cl-cu}
 c_l \mathfrak{E}_\epsilon^s \le  c_e \cdot   \mathfrak{E}_{\epsilon,1}^s(t) +  c_f \cdot   \mathfrak{E}_{\epsilon,2,2}^s(t) +  \mathfrak{E}_{\epsilon,2,1}^s(t) \le c_u \mathfrak{E}_\epsilon^s.
\end{align}
We complete the proof for $s\ge 2$ for this lemma.
\end{proof}
Let's give a remark on the difference on \eqref{vpb-linear-m} and \eqref{vpb-linear-sqrtm}.
\begin{remark}
\label{remark-m-sqrtm}
While using \eqref{vpb-linear-sqrtm}, the last term in \eqref{est-linear-derivative-v-x-0} is replaced by
$-\reps\intps \nabla_v (\vdot \phi_\epsilon \sdot \sqrt\m) \sdot \nabla_x g_\epsilon\bd v \bd x$, by integration by parts,
\begin{align*}
-\reps\intps \nabla_v (\vdot \phi_\epsilon \sdot \sqrt\m) \sdot \nabla_x g_\epsilon\bd v \bd x = \reps \|\Delta \phi_\epsilon\|_{L^2}^2 - \tfrac{1}{2\epsilon} \sum\limits_{i=1}^3\intps \partial_{x_ix_j}^2 \phi_\epsilon \sdot  v_i v_j \sqrt\m g_\epsilon\bd v \bd x.
\end{align*}
The last term in the above equation can not be controlled directly. If we add some new terms up to the norm used in \cite[Theorem 2.1]{briant-2015-be-to-ns} to cancel it, the new norms is just the one we use.
\end{remark}
\section{Prior estimates of the nonlinear system }
This section is devoted to deducing the same version to Lemma \ref{lemma-linear-decay}.
The non-linear system
\begin{align}
\label{vpb-ns-nonlinear}
\partial_t g_\epsilon + \reps \vdot g_\epsilon +  \repst \mathcal{L}(g_\epsilon)   - \tfrac{v}{\epsilon} \sdot \nabla \phi_\epsilon = N(g_\epsilon),
\end{align}
where
\begin{align*}
N(g_\epsilon)&:= N_1(g_\epsilon) + N_2(g_\epsilon),\\
N_1(g_\epsilon)&:= ({v}  g_\epsilon - \nabla_v g_\epsilon     ) \sdot \nabla\phi_\epsilon,\\
N_2(g_\epsilon)&:= \reps {\Gamma}(g_\epsilon,g_\epsilon).
\end{align*}
For the nonlinear case, the global conservation law is
\begin{align}
\label{conservation-law-nonlinear}
\begin{split}
\tdt { \intps g_\epsilon(t) \m \bd v \bd x = 0,   \tdt \intps v g_\epsilon(t) \m \bd v \bd x =0,}\\
\tdt \big(  \intps (\tfrac{|v|^2}{2} - \tfrac{3}{2}) g_\epsilon(t)\m \bd v \bd x +  \epsilon \|\nabla_x \phi_\epsilon(t)\|_{L^2}^2 \big)= 0.
\end{split}
\end{align}

The initial data are assumed to satisfy
\begin{align}
\label{initial-mean-nonlinear}
\begin{split}
  { \intps g_\epsilon(0) \m \bd v \bd x = 0,     \intps v g_\epsilon(0) \m \bd v \bd x =0,}\\
  \big( 3\intps (\tfrac{|v|^2}{3} - 1) g_\epsilon(0)\m \bd v \bd x +  \epsilon \|\nabla_x \phi_\epsilon(0)\|_{L^2}^2 \big)= 0.
  \end{split}
\end{align}.

\begin{lemma}
\label{lemma-nonlinear-decay}
Under the assumptions on kernels in Sec. \ref{sec-assump-determin} and assumptions  \eqref{initial-mean-nonlinear} on the initial data, there exists some small enough  constant $c_{00}$ such that as long as
\[  \|g_\epsilon(0)\|_{H^s_x}^2 + \epsilon^2 \|\nabla_v   g_\epsilon(0)\|_{H^{s-1}}^2   \le c_{00}, \] equations \eqref{vpb-ns-scaling-ns-corrector} admit a solution  $(g_\epsilon, \nabla \phi_\epsilon)$ satisfying that there exist $\bar{c}_{00}>0$ and $\bar{c_0}>0$ (all independent of $\epsilon$ while $\epsilon <1$) such that
\begin{align}
\label{est-lemma-nonlinear-exp-decay}
\|(g_\epsilon(t)\|_{H^s_x}^2 + \epsilon^2 \|\nabla_v g_\epsilon(t)\|_{H^{s-1}}^2 \le \bar{c}_{00} \exp(- \bar{c_0} t).
\end{align}
\end{lemma}
\begin{remark}
\begin{align}\label{c0}
C_v (0) = \tfrac{\epsilon}{3} \|\nabla_x \phi_\epsilon(0)\|_{L^2}^2   \sdot (\tfrac{|v|^2}{3} - 1)
  + \intt \theta_\epsilon(0) \bd x (\tfrac{|v|^2}{3} - 1) + \intt u_\epsilon(0) \bd x \sdot v + \intt \rho_\epsilon(0)\bd x.
\end{align}
Combing \eqref{initial-mean-nonlinear} and  \eqref{c0}, we can deduce $C_v(0)=0$.  This type assumptions were used in \cite[Page 600]{guo-2003-vmb-invention} which hints that   the   total mass, total momentum and total energy are the same as the steady state.
\end{remark}
\begin{proof}
The proof of the nonlinear case is parallel to the process of Lemma \ref{lemma-linear-decay}. First, about $\mathfrak{E}_{\epsilon,1}^s$, by the similar process as those in the proof of  Lemma \ref{lemma-linear-decay}, we can deduce similar estimates for equation \eqref{vpb-ns-nonlinear} as \eqref{est-linear-h1-all-0}. During the calculation, $N(g_\epsilon)$ can be formally seen as the source term.
\begin{align}
\label{est-nonlinear-h1-all-0}
\begin{split}
& \dt \mathfrak{E}_\epsilon^1(t)  + \tfrac{\lambda_1}{\epsilon^2}\|g_\epsilon^\perp\|_{L^2}^2 + \tfrac{\lambda_2}{\epsilon^2}\|(\nabla_x g_\epsilon)^\perp\|_{L^2}^2 +  \lambda_3 (a_3 -3\delta) \|\nabla_v g_\epsilon\|_{L^2_\Lambda}^2 \\
& \qquad +   {\lambda_4}  \big( \|\nabla_x g_\epsilon\|_{L^2}^2 + \|\vdot g_\epsilon\|_{L^2}^2 + 2 \|\rho_\epsilon\|_{L^2}^2\big)\\
& \le \lambda_3 \underline{{(a_4 + C_\delta)}\|\bdp g_\epsilon\|_{L^2}^2} + \lambda_3 C_{\delta,1}\epsilon^2\|\nabla_x g_\epsilon\|_{L^2}^2 +  \lambda_3 C_{\delta, 2} \epsilon^2 \|\nabla_x \phi_\epsilon\|_{L^2_x}^2\\
& \qquad + \lambda_4 \tfrac{C_u^2}{a_6} \|(\nabla_x g_\epsilon)^\perp\|_{L^2_{\Lambda}}^2  + \lambda_4 a_6\|\nabla_v g_\epsilon \|_{L^2_{\Lambda}}^2 + \lambda_3 {{(a_4 + C_\delta)}\|g_\epsilon^\perp\|_{L^2}^2}\\
& \qquad + \lambda_1\intps N(g_\epsilon) g_\epsilon\bdv \bd x + \lambda_2\intps \nabla_x N(g_\epsilon) \sdot \nabla_x g_\epsilon\bdv \bd x \\
& \qquad + \lambda_3\epsilon^2\intps \nabla_v N(g_\epsilon) \sdot \nabla_v g_\epsilon\bdv \bd x + \lambda_4 \epsilon \intps \nabla_x N(g_\epsilon) \sdot \nabla_x g_\epsilon\bdv \bd x.
\end{split}
\end{align}
The terms in last two lines are triple of $g_\epsilon$. They can be bounded  and absorbed while the initial data are small enough. The difficulty are caused by he underlined term in \eqref{est-nonlinear-h1-all-0}. As mentioned before, its the mean value   on the torus is not equal to zero. According to \eqref{conservation-law-nonlinear},
\begin{align}
\label{est-mean-difference}
\begin{split}
\intt \bdp g_\epsilon(t) \bd x  =C_\Omega^{-1}\sdot C_v(0)  - C_\Omega^{-1} \sdot \tfrac{\epsilon}{3} \|\nabla_x \phi_\epsilon(t)\|_{L^2}^2   \sdot (\tfrac{|v|^2}{3} - 1), ~~ C_\Omega = \intt \bd x.
\end{split}
\end{align}
Thus,
\begin{align}
\label{est-differ}
\begin{split}
\|\bdp g_\epsilon\|_{L^2}^2 & \le  \|\bdp g_\epsilon + C_\Omega^{-1}\tfrac{\epsilon}{3}\|\nabla_x \phi_\epsilon\|_{L^2}^2 \sdot (\tfrac{|v|^2}{3} - 1) - C_\Omega^{-1}\sdot C_v(0) \|_{L^2}^2  \\
& + \| C_\Omega^{-1}\tfrac{\epsilon}{3}\|\nabla_x \phi_\epsilon\|_{L^2}^2 \sdot (\tfrac{|v|^2}{3} - 1) - C_\Omega^{-1}\sdot C_v(0) \|_{L^2}^2\\
& \le a_5 \|\nabla_x \bdp g_\epsilon  \|_{L^2}^2 + \| C_\Omega^{-1}\tfrac{\epsilon}{3}\|\nabla_x \phi_\epsilon\|_{L^2}^2 \sdot (\tfrac{|v|^2}{3} - 1) - C_\Omega^{-1}\sdot C_v(0) \|_{L^2}^2  \\
& \le a_5 \|\nabla_x \bdp g_\epsilon  \|_{L^2}^2 + \| C_\Omega^{-1}\tfrac{\epsilon}{3}\|\nabla_x \phi_\epsilon\|_{L^2}^2 \sdot (\tfrac{|v|^2}{3} - 1)\|_{L^2}^2  +  \|C_\Omega^{-1}\sdot C_v(0) \|_{L^2}^2\\
& \le a_5 \|\nabla_x \bdp g_\epsilon  \|_{L^2}^2 +   \tfrac{\epsilon^2}{9C_\Omega} \|\nabla_x \phi_\epsilon\|_{L^2}^4 +  {C}_{in},
\end{split}
\end{align}
where $C_{in}$ only depends on the fluid part of initial data and
 \begin{align}
 \label{cin}
{C}_{in} = \|C_\Omega^{-1}\sdot C_v(0) \|_{L^2}^2.
\end{align}

Compared to the linear case, there exists two more terms $ \|\nabla_x \phi_\epsilon\|_{L^2}^4$ and ${C}_{in}$. Combing \eqref{constant-a5}, we can conclude that
\begin{align}
\label{est-nonlinear-all-h1}
\begin{split}
& \dt \mathfrak{E}_\epsilon^1(t) + \tilde\lambda_1\|g_\epsilon^\perp\|_{L^2}^2 + \tilde\lambda_2\|(\nabla_x g_\epsilon)^\perp\|_{L^2}^2  + \lambda_4 \|\vdot g_\epsilon\|_{L^2}^2 \\
& \qquad +  \lambda_5 \|\nabla_v g_\epsilon\|_{L^2_\Lambda}^2  +  \lambda_6\|\nabla_x g_\epsilon\|_{L^2}^2 +  \lambda_7\|\nabla_x \phi_\epsilon\|_{L^2_x}^2 \\
& \le    \tfrac{\epsilon^2 \lambda_3 {(a_4 + C_\delta)} }{9C_\Omega} \|\nabla_x \phi_\epsilon\|_{L^2}^4  + \lambda_1\intps N(g_\epsilon) g_\epsilon\bdv \bd x \\
& + \lambda_2\intps \nabla_x N(g_\epsilon) \sdot \nabla_x g_\epsilon\bdv \bd x  + \lambda_3\epsilon^2\intps \nabla_v N(g_\epsilon) \sdot \nabla_v g_\epsilon\bdv \bd x \\
& + \lambda_4\epsilon\intps \nabla_x N(g_\epsilon) \sdot \nabla_x g_\epsilon\bdv \bd x + \lambda_3 {(a_4 + C_\delta)} C_{in}.
\end{split}
\end{align}
By the similar way of choosing $\lambda_i$, it follows that
\begin{align}
\label{est-nonlinear-es1}
\begin{split}
& \dt \mathfrak{E}_{\epsilon,1}^s(t) + c_d (\|g_\epsilon\|_{H^s}^2 + \tfrac{1}{\epsilon^2} \|(g_\epsilon)^\perp\|_{H^s}^2 + \|\nabla_v g_\epsilon\|_{H^{s-1}}^2+ \|\nabla_x \phi\|_{H^s}^2 )\\
& \qquad \le    \tfrac{\epsilon^2 \lambda_3 {(a_4 + C_\delta)} }{9C_\Omega} \|\nabla_x \phi_\epsilon\|_{L^2}^4   + \lambda_3 {(a_4 + C_\delta)} C_{in} \\
&  \qquad +  \lambda_1 \sum\limits_{k=0}^{s-1} \intps \nabla^k_x N(g_\epsilon) \sdot \nabla_x^k g_\epsilon\bdv \bd x \\
& \qquad + \lambda_2 \sum\limits_{k=0}^{s-1} \intps \nabla_x \nabla^k_x N(g_\epsilon) \sdot \nabla_x \nabla^k_x g_\epsilon\bdv \bd x \\
&  \qquad + \lambda_3 \epsilon^2 \sum\limits_{k=0}^{s-1} \intps \nabla_v \nabla^k_x N(g_\epsilon) \sdot \nabla_v \nabla^k_x g_\epsilon\bdv \bd x \\
&  \qquad + \lambda_4 \epsilon \sum\limits_{k=0}^{s-1} \intps \nabla_x \nabla^k_x N(g_\epsilon) \sdot \nabla_v \nabla^k_x g_\epsilon\bdv \bd x.
\end{split}
\end{align}
Except for $\lambda_3 {(a_4 + C_\delta)} C_{in}$, the right hand of \eqref{est-nonlinear-all-hs} can be bounded and absorbed by the left hand while the initial data are small enough. According to the linear case, we still need to consider the estimates of $\nabla_x^j \nabla_v^i g_\epsilon$ for $i \ge 2$, recalling that
\begin{align*}
\mathfrak{E}_{\epsilon,2}^s(g(t)) &= \mathfrak{E}_{\epsilon,2,1}^s(g(t)) + \mathfrak{E}_{\epsilon,2,2}^s(t),\\
\mathfrak{E}_{\epsilon,2,1}^s(t)=\sum\limits_{i = 2, i + j =s}\epsilon^2\|\nabla^i_v \nabla_x^j g_\epsilon\|_{L^2}^2,&~~\mathfrak{E}_{\epsilon,2,2}^s(t) = \sum\limits_{i \ge 3, i + j =s}\epsilon^2\|\nabla^i_v \nabla_x^j g_\epsilon\|_{L^2}^2,
\end{align*}
from the linear case, as long as that the derivative of $g_\epsilon$ with respect to $v$ is greater than $1$, the fluid parts are not needed to split out like \eqref{est-nonlinear-h1-all-0}. Similar to \eqref{est-linear-e2}, its nonlinear version reads as
\begin{align}
\label{est-nonlinear-e2}
\begin{split}
& \dt \bigg(  c_f \cdot   \mathfrak{E}_{\epsilon,2,2}^s(t) +  \mathfrak{E}_{\epsilon,2,1}^s(t)\bigg) + \tfrac{\delta}{\epsilon^2}  \mathfrak{E}_{\epsilon,2}^s(t) \\
& \qquad \le (a_4 + C_\delta +  C_{\delta,1} \epsilon^2) \mathfrak{E}_{\epsilon,1}^s(t)\\
& \qquad + (c_f +1) \epsilon^2  \sum\limits_{i \ge 2\atop i + j =s} \vert\intps \big(\nabla^j_x \nabla^i_v N(g_\epsilon)\big) \sdot \big(\nabla^j_x \nabla^i_v g_\epsilon\big)\bdv \bd x \vert.
\end{split}
\end{align}
Since $c_d \le \delta$, thus combing the relevant estimates,

\begin{align}
\label{est-nonlinear-all-hs}
\begin{split}
& \dt \bigg( c_e \cdot   \mathfrak{E}_{\epsilon,1}^s(t) +  c_f \cdot   \mathfrak{E}_{\epsilon,2,2}^s(t) +  \mathfrak{E}_{\epsilon,2,1}^s(t)\bigg) \\
& \qquad +   \tfrac{c_e}{\epsilon^2} \|(g_\epsilon)^\perp\|_{H^s}^2  + c_e\|\rho_\epsilon\|_{H^s}^2  + c_d\|g_\epsilon\|_{H^s_{x,v}}^2  \\
&\le  \tfrac{\epsilon^2 \lambda_3 {(a_4 + C_\delta)} }{9C_\Omega} \|\nabla_x \phi_\epsilon\|_{L^2}^4   + \lambda_3 {(a_4 + C_\delta)} C_{in} \\
& + c_e(\lambda_1 + \lambda_2) \sum\limits_{k=0}^{s} \vert \intps \nabla^k_x N(g_\epsilon) \sdot \nabla_x^k g_\epsilon\bdv \bd x \vert \\
& + (c_f +1 + c_e \lambda_3) \epsilon^2  \sum\limits_{i \ge 1\atop i + j =s} \vert\intps \big(\nabla^j_x \nabla^i_v N(g_\epsilon)\big) \sdot \big(\nabla^j_x \nabla^i_v g_\epsilon\big)\bdv \bd x \vert\\
& + \lambda_4 \cdot  c_e \epsilon \sum\limits_{k=0}^{s-1} \vert \intps \nabla_x \nabla^k_x N(g_\epsilon) \sdot \nabla_v \nabla^k_x g_\epsilon\bdv \bd x\vert.
\end{split}
\end{align}
For the nonlinear term, recalling
\begin{align*}
N(g_\epsilon)&:= N_1(g_\epsilon) + N_2(g_\epsilon),\\
N_1(g_\epsilon)&:= ({v}  g_\epsilon - \nabla_v g_\epsilon     ) \sdot \nabla\phi_\epsilon,\\
N_2(g_\epsilon)&:= \reps {\Gamma}(g_\epsilon,g_\epsilon),
\end{align*}
According to the definition of $\mathfrak{E}_\epsilon^s(t)$ in \eqref{est-norm-es-g-2}, as long as the order of  derivative of $g_\epsilon$ with respect to $v$ is greater than one,  there exists a coefficient $\epsilon^2$ in the norm. We should carefully calculate the non-linear terms. $N_2(g_\epsilon, g_\epsilon)$ can be controlled by the same way as that in \cite{briant-2015-be-to-ns}. For the term generated by the electric field,
\begin{align*}
& \sum\limits_{ i+j =s}\vert \intps \nabla^j_x \nabla^i_v \big( ({v}  g_\epsilon - \nabla_v g_\epsilon     ) \sdot \nabla\phi_\epsilon \big) \sdot \nabla^j_x \nabla^i_v g_\epsilon\bdv \bd x \vert \\
& \le   \sum\limits_{ i+j =s} \vert \intps \nabla^j_x \nabla^i_v \big(  {v}  g_\epsilon    \sdot \nabla\phi_\epsilon \big) \sdot \nabla^j_x \nabla^i_v g_\epsilon\bdv \bd x \vert \\
& + \sum\limits_{ i+j =s} \vert \intps \nabla^j_x \nabla^i_v \big( \nabla_v g_\epsilon      \sdot \nabla\phi_\epsilon \big) \sdot \nabla^j_x \nabla^i_v g_\epsilon\bdv \bd x \vert.
\end{align*}
For the second term, we can find that
\begin{align*}
& \sum\limits_{ i+j =s} \vert \intps \nabla^j_x \nabla^i_v \big( \nabla_v g_\epsilon      \sdot \nabla\phi_\epsilon \big) \sdot \nabla^j_x \nabla^i_v g_\epsilon\bdv \bd x \vert\\
& \le \sum\limits_{ i+j =s} \vert \intps \nabla^j_x  \big( \nabla_v \nabla^i_v g_\epsilon      \sdot \nabla\phi_\epsilon \big) \sdot \nabla^j_x \nabla^i_v g_\epsilon\bdv \bd x \vert \\
& \le  \sum\limits_{ i+j =s} \vert \intps   \big( \nabla_v \nabla^j_x \nabla^i_v g_\epsilon      \sdot \nabla\phi_\epsilon \big) \sdot \nabla^j_x \nabla^i_v g_\epsilon\bdv \bd x \vert\\
& + \sum\limits_{  i+j =s, i \ge 1 \atop k + l = j, l \ge 1} \vert \intps   \big(   \nabla^k_x \nabla^{ i +1}_v g_\epsilon      \sdot \nabla_x^{l+1}\phi_\epsilon \big) \sdot \nabla^j_x \nabla^i_v g_\epsilon\bdv \bd x \vert\\
& +  \vert \intps   \big(    \nabla_v g_\epsilon      \sdot \nabla_x^{s+1}\phi_\epsilon \big) \sdot \nabla^j_x \nabla^i_v g_\epsilon\bdv \bd x \vert.
\end{align*}
Noticing that $\|\rho_\epsilon\|_{H^s} = \|\phi_\epsilon\|_{H^{s+2}}$, thus we can infer after integration by parts
\begin{align*}
& \sum\limits_{i+j =s} \vert \intps   \big( \nabla_v \nabla^j_x \nabla^i_v g_\epsilon      \sdot \nabla\phi_\epsilon \big) \sdot \nabla^j_x \nabla^i_v g_\epsilon\bdv \bd x \vert\\
& \quad = \tfrac{1}{2}\sum\limits_{i+j =s} \vert \intps     \nabla_v (\nabla^j_x \nabla^i_v g_\epsilon)^2 \sdot \nabla\phi_\epsilon  \bdv \bd x \vert\\
& \quad = \tfrac{1}{2}\sum\limits_{i+j =s} \vert \intps      (\nabla^j_x \nabla^i_v g_\epsilon)^2 \sdot v \sdot \nabla\phi_\epsilon  \bdv \bd x \vert\\
& \quad \lesssim \|\rho_\epsilon\|_{H^s}\|g_\epsilon\|_{H^s_\Lambda}^2.
\end{align*}
Similarly, since $l+ 1 \le s$, thus
\begin{align*}
& \sum\limits_{  i+j =s, i \ge 1 \atop k + l = j, l \ge 1} \vert \intps   \big(   \nabla^k_x \nabla^{ i +1}_v g_\epsilon      \sdot \nabla_x^{l+1}\phi_\epsilon \big) \sdot \nabla^j_x \nabla^i_v g_\epsilon\bdv \bd x \vert \\
& \lesssim \|\nabla_x^{l+1}\phi_\epsilon\|_{L^\infty}  \sum\limits_{  i+j =s, i \ge 1 \atop k + l = j, l \ge 1} \vert \intps   |\nabla^k_x \nabla^{ i +1}_v g_\epsilon| \sdot        |\nabla^j_x \nabla^i_v g_\epsilon| \bdv \bd x \vert \\
& \lesssim \|\rho_\epsilon\|_{H^s}\|g_\epsilon\|_{H^s}^2.
\end{align*}
For the last one,
\begin{align*}
& \vert \intps   \big(    \nabla_v g_\epsilon      \sdot \nabla_x^{s+1}\phi_\epsilon \big) \sdot \nabla^j_x \nabla^i_v g_\epsilon\bdv \bd x \vert \\
& \lesssim \vert \intt | \nabla_x^{s+1}\phi_\epsilon|  \intv      |\nabla_v g_\epsilon|       \sdot |\nabla^j_x \nabla^i_v g_\epsilon|\bdv \bd x \vert\\
& \lesssim \vert \intt | \nabla_x^{s+1}\phi_\epsilon|   \|\nabla_v g_\epsilon\|_{L^2_v}       \sdot \|\nabla^j_x \nabla^i_v g_\epsilon\|_{L^2_v}  \bd x \vert \\
& \lesssim \|\|\nabla_v g_\epsilon\|_{L^2_v}\|_{L^\infty}\|g_\epsilon\|_{H^s_x}\|g_\epsilon\|_{H^s}\\
& \lesssim  \|\rho_\epsilon\|_{H^s}\|g_\epsilon\|_{H^s}^2.
\end{align*}
Similarly, we can infer that
\begin{align*}
\sum\limits_{i+j =s} \vert \intps \nabla^j_x \nabla^i_v \big(  {v}  g_\epsilon    \sdot \nabla\phi_\epsilon \big) \sdot \nabla^j_x \nabla^i_v g_\epsilon\bdv \bd x \vert  \lesssim   \|\rho_\epsilon\|_{H^s}\|g_\epsilon\|_{H^s_\Lambda}^2.
\end{align*}

All together, we can conclude that
\begin{align}
\sum\limits_{\atop i+j =s}\vert \intps \nabla^j_x \nabla^i_v \big( ({v}  g_\epsilon - \nabla_v g_\epsilon     ) \sdot \nabla\phi_\epsilon \big) \sdot \nabla^j_x \nabla^i_v g_\epsilon\bdv \bd x \vert \lesssim \|g_\epsilon\|_{H^s_x}\|g_\epsilon\|_{H^s_\Lambda}^2.
\end{align}

For the second part of $N(g_\epsilon)$  containing nonlinear collision operator in the right hand of \eqref{est-nonlinear-all-hs}, i.e., $N_2(g_\epsilon)$, recalling $N_2(g_\epsilon)  = \reps {\Gamma}(g_\epsilon,g_\epsilon)$ and $\Gamma(g_\epsilon, g_\epsilon)$ belongs to the orthogonal space of $\mathcal{L}$,
\begin{align*}
& c_e(\lambda_1 + \lambda_2)\sum\limits_{k=0}^{s} \vert \intps \nabla^k_x N_2(g_\epsilon) \sdot \nabla_x^k g_\epsilon\bdv \bd x \vert \\
& = c_e(\lambda_1 + \lambda_2)\sum\limits_{k=0}^{s} \vert \intps \nabla^k_x \Gamma(g_\epsilon, g_\epsilon) \sdot \tfrac{1}{\epsilon}\big(\nabla_x^k g_\epsilon\big)^\perp\bdv \bd x \vert \\
& \le \tfrac{c_e^2(\lambda_1 + \lambda_2)^2}{c_e}\|\Gamma(g_\epsilon, g_\epsilon)\|_{H^s_x}^2 + \tfrac{c_e}{2\epsilon^2}\|g_\epsilon^\perp\|_{H^s_x}^2.
\end{align*}
For the last two terms in the right hand of \eqref{est-nonlinear-all-hs},
\begin{align*}
& \epsilon^2  \sum\limits_{i \ge 1\atop i + j =s} \vert\intps \big(\nabla^j_x \nabla^i_v N_2(g_\epsilon)\big) \sdot \big(\nabla^j_x \nabla^i_v g_\epsilon\big)\bdv \bd x \vert \\
& + \epsilon \sum\limits_{k=0}^{s-1} \vert \intps \nabla_x \nabla^k_x N_2(g_\epsilon) \sdot \nabla_v \nabla^k_x g_\epsilon\bdv \bd x\vert \\
& \le  \epsilon \sum\limits_{i \ge 1\atop i + j =s} \vert\intps \big(\nabla^j_x \nabla^i_v \Gamma(g_\epsilon, g_\epsilon)\big) \sdot \big(\nabla^j_x \nabla^i_v g_\epsilon\big)\bdv \bd x \vert\\
& \qquad +  \sum\limits_{k=0}^{s-1} \vert \intps \nabla_x \nabla^k_x \Gamma(g_\epsilon, g_\epsilon) \sdot \nabla_v \nabla^k_x g_\epsilon\bdv \bd x\vert\\
& \lesssim \epsilon \|\Gamma(g_\epsilon, g_\epsilon)\|_{H^s}\|g_\epsilon\|_{H^{s}} +  \|\Gamma(g_\epsilon, g_\epsilon)\|_{H^s_x}\|\nabla_v g_\epsilon\|_{H^{s-1}}.
\end{align*}
By the assumption \eqref{assump-h1-h},
\begin{align*}
\|\Gamma(g_\epsilon, g_\epsilon)\|_{H^s_x} \lesssim \|g_\epsilon\|_{H^s_x} \|g_\epsilon\|_{H^s_{\Lambda_x}},~~\|\Gamma(g_\epsilon, g_\epsilon)\|_{H^s} \lesssim \|g_\epsilon\|_{H^s}\|g_\epsilon\|_{H^s_\Lambda}
\end{align*}
Combing the relevant estimates, there exists some constant $C_n$ such that
\begin{align}
\begin{split}
& \dt \bigg( c_e \cdot   \mathfrak{E}_{\epsilon,1}^s(t) +  c_f \cdot   \mathfrak{E}_{\epsilon,2,2}^s(t) +  \mathfrak{E}_{\epsilon,2,1}^s(t)\bigg) \\
& \qquad +   \tfrac{c_e}{2\epsilon^2} \|(g_\epsilon)^\perp\|_{H^s_x}^2  + c_e\|\rho_\epsilon\|_{H^s}^2  + c_d\|g_\epsilon\|_{H^s_{x,v}}^2  \\
&\le C_n\big(\|g_\epsilon\|_{H^s_x}^2 + \|g_\epsilon\|_{H^s_x} + \epsilon\|g_\epsilon\|_{H^s}\big)\|g_\epsilon\|_{H^s}^2     + \lambda_3 {(a_4 + C_\delta)} C_{in}.
\end{split}
\end{align}
While the initial data satisfy   \eqref{initial-mean-nonlinear},
\[ C_{in} = 0, \]
we can conclude that
\begin{align}
\label{est-re-z}
\begin{split}
& \dt \bigg( c_e \cdot   \mathfrak{E}_{\epsilon,1}^s(t) +  c_f \cdot   \mathfrak{E}_{\epsilon,2,2}^s(t) +  \mathfrak{E}_{\epsilon,2,1}^s(t)\bigg) +  c_d \|g_\epsilon\|_{H^{s}_x}^2 +   \tfrac{c_d}{\epsilon^2} \cdot \epsilon^2 \|\nabla_v g_\epsilon\|_{H^{s-1}}^2 \\
&\lesssim \big(\|g_\epsilon\|_{H^s}^2 + \|g_\epsilon\|_{H^s} + \|\rho_\epsilon\|_{H^s}\big)\|g_\epsilon\|_{H^s}^2     + \lambda_3 {(a_4 + C_\delta)} C_{in}\\
&\lesssim (\|g_\epsilon\|_{H^s_x} +   \epsilon \|\nabla_v g_\epsilon\|_{H^{s-1}})( \|g_\epsilon\|_{H^{s}_x}^2 +   \tfrac{1}{\epsilon^2} \cdot \epsilon^2 \|\nabla_v g_\epsilon\|_{H^{s-1}}^2 )\\
& \lesssim   \sqrt{\mathfrak{E}_\epsilon^s(t)} \|g_\epsilon\|_{H^{s}_x}^2  + \tfrac{1}{\epsilon^2} \sqrt{\mathfrak{E}_\epsilon^s(t)}\cdot  \epsilon^2 \|\nabla_v g_\epsilon\|_{H^{s-1}}^2.
\end{split}
\end{align}
Thus, there exists some constant $C_s$
\begin{align}
\label{est-v-uni}
\begin{split}
& \dt \bigg( c_e \cdot   \mathfrak{E}_{\epsilon,1}^s(t) +  c_f \cdot   \mathfrak{E}_{\epsilon,2,2}^s(t) +  \mathfrak{E}_{\epsilon,2,1}^s(t)\bigg) +  c_d \|g_\epsilon\|_{H^{s}_{\Lambda_x}}^2 +   \tfrac{c_d}{\epsilon^2} \cdot \epsilon^2 \|\nabla_v g_\epsilon\|_{H^{s-1}_\Lambda}^2 + \tfrac{c_e}{2\epsilon^2} \|(g_\epsilon)^\perp\|_{H^s_x}^2 \\
& \le C_s \sqrt{\mathfrak{E}_\epsilon^s(t)}\|g_\epsilon\|_{H^{s}_x}^2 + \tfrac{C_s}{\epsilon^2} \sqrt{\mathfrak{E}_\epsilon^s(t)}\cdot  \epsilon^2 \|\nabla_v g_\epsilon\|_{H^{s-1}}^2.
\end{split}
\end{align}
From \eqref{constant-cl-cu},  since $c_e \cdot   \mathfrak{E}_{\epsilon,1}^s(t) +  c_f \cdot   \mathfrak{E}_{\epsilon,2,2}^s(t) +  \mathfrak{E}_{\epsilon,2,1}^s(t)$ is equivalent to $\mathfrak{E}_\epsilon^s(t)$, i.e.,
\begin{align*}
 c_l \mathfrak{E}_\epsilon^s \le  c_e \cdot   \mathfrak{E}_{\epsilon,1}^s(t) +  c_f \cdot   \mathfrak{E}_{\epsilon,2,2}^s(t) +  \mathfrak{E}_{\epsilon,2,1}^s(t) \le c_u \mathfrak{E}_\epsilon^s.
\end{align*}
thus by the continuous bootstrap method, as long as the initial data
\begin{align}
\label{est-v-uni-2}
\mathfrak{E}_\epsilon^s(0) \le \frac{c_l c_d^2}{4 c_u C_s^2 },
\end{align}
the global existence can be obtained and
\begin{align}
\label{est-v-uni-3}
\mathfrak{E}_\epsilon^s(t) \le \frac{c_d^2}{4  C_s^2 },~~\forall t >0.
\end{align}
Moreover,
\begin{align}
\label{est-to-e}
\begin{split}
& \dt \bigg( c_e \cdot   \mathfrak{E}_{\epsilon,1}^s(t) +  c_f \cdot   \mathfrak{E}_{\epsilon,2,2}^s(t) +  \mathfrak{E}_{\epsilon,2,1}^s(t)\bigg) +  \frac{c_d}{2 c_u}\bigg( c_e \cdot   \mathfrak{E}_{\epsilon,1}^s(t) +  c_f \cdot   \mathfrak{E}_{\epsilon,2,2}^s(t) +  \mathfrak{E}_{\epsilon,2,1}^s(t)\bigg) \le 0.
\end{split}
\end{align}
Thus, we can obtain that
\begin{align}
\label{est-nonlinear-decay-rate}
\mathfrak{E}_\epsilon^s(t) \le \tfrac{1}{c_l}\exp(- \tfrac{c_d}{2c_u} t).
\end{align}
\end{proof}

\section{Construction of approximate solutions}
\label{sec-app}
In this section, we are going to sketch the process of construction of approximate solutions. For any fixed $\epsilon>0$,
\begin{align}
\label{vpb-ns-nonlinear-approx}
\partial_t g_{n,\epsilon} + \reps \vdot g_{n,\epsilon} +  \repst \mathcal{L}(g_{n,\epsilon})   - \tfrac{v}{\epsilon} \sdot \nabla \phi_{n,\epsilon} = N(g_{n-1,\epsilon}),
\end{align}
where
\begin{align*}
\Delta_x \phi_{n-1, \epsilon} & =\int g_{n-1,\epsilon} \m \bd v,\\
N(g_{n-1,\epsilon})&:= N_1(g_{n-1,\epsilon}) + N_2(g_{n-1,\epsilon}),\\
N_1(g_{n-1,\epsilon})&:= ({v}  g_{n,\epsilon} - \nabla_v g_{n,\epsilon}     ) \sdot \nabla\phi_{n-1,\epsilon},\\
N_2(g_{n-1,\epsilon})&:= \reps {\Gamma}(g_{n-1,\epsilon},g_{n-1,\epsilon}).
\end{align*}
We also assume that initial states of $g_{\epsilon,n}$ satisfy that
\begin{align}
\label{initial-mean-nonlinear-approx}
\begin{split}
  { \intps g_{n,\epsilon}(0) \m \bd v \bd x = 0,     \intps v g_{n,\epsilon}(0) \m \bd v \bd x =0,}\\
  \big( 3\intps (\tfrac{|v|^2}{3} - 1) g_{n,\epsilon}(0)\m \bd v \bd x +  \epsilon \|\nabla_x \phi_{n-1,\epsilon}(0)\|_{L^2}^2 \big)= 0.
  \end{split}
\end{align}
Under the settings and \eqref{vpb-ns-nonlinear-approx} and \eqref{initial-mean-nonlinear-approx}, for each $n \ge 1,\, n \in \mathbb{N}^+$,  the first  equation of \eqref{vpb-ns-nonlinear-approx} is linear with source term $N(g_{n-1,\epsilon})$. Provided that there exist estimates (in $\mathfrak{E}_\epsilon^s(t)$ norm) of $g_{n-1, \epsilon}$, the existence of $g_{n, \epsilon}$ is guaranteed by Hahn-Banach theorem.  Thus,  we can prove the existence of $g_{n,\epsilon}$ for each $n\in\mathbb{N}^+$.  From equations \eqref{vpb-ns-nonlinear-approx}, one can establish similar version of \eqref{est-nonlinear-all-hs}. Following the similar process of dealing with  \eqref{est-nonlinear-all-hs} and by induction method, the uniform estimates of $g_{n, \epsilon}$ with respect to $n$ can be achieved. By weak convergence method, the wellposedness of $g_\epsilon$ can be proved. Besides, as there exists one more derivative with respect to $v$ in $N_1(g_\epsilon)$, i.e., $\nabla_v g_\epsilon      \sdot \nabla\phi_\epsilon$, the uniqueness of $g_\epsilon$ is verified only in $\mathfrak{E}_\epsilon^{s-1}(t)$ space.

\section{The convergence rate of $g_\epsilon$}
\label{sec-con}
This section consists of prove the main results Thoerem \ref{main-results-convergence-rate}.

Based on the uniform estimates \eqref{est-lemma-nonlinear-exp-decay} of $g_\epsilon$ with respect to $\epsilon$, i.e.,
\[ \|g_\epsilon(t)\|_{H^s_x}^2 \le  \bar{c}_0,~~~~\forall t \ge 0, \]
together with the formal analysis in \cite{diogosrm-2019-vmb-fluid}, there exists a unique $g$ with
\[ g = \rho(t,x) +  u(t,x)\sdot v + \theta(t,x)\left( \tfrac{|v|^2}{2} - \tfrac{3}{2}\right),~~~~\divg u =0,~~\Delta(\rho+ \theta) = \rho. \]
Specially, $\rho$,~$u$ and $\theta$ is a strong solution to \eqref{nsfp}. We will give a sketch of proof on verifying the fluids limit in Sec. \ref{sec-limits}.

The goal of this section is to obtain the convergence rate of $g_\epsilon$ to $g$ with the average of time. Based on the spectrum analysis and Duhamel's principle, for the Boltzmann equation, from  \cite[Theorem 2.5]{briant-2015-be-to-ns},  the convergence rate is $\sqrt{\epsilon|\ln \epsilon|}$. In the rest part of this section, with the similar  spectrum representation of semigroup as \cite{briant-2015-be-to-ns,convergence-rate-linear-1975}, we will show  the convergence is the same for VPB system, that is to say
\[ \|\int_0^t g_\epsilon(s) \bd s - \int_0^t g(s)\bd s \|_{H^s_x} \approx O(\sqrt{\epsilon|\ln \epsilon|}),~~~~\forall~ t>0. \]
For the linear VPB system \eqref{vpb-ns-scaling-ns-linear}, recalling that
\[ \vpbg(g_\epsilon):= -\reps \vdot g_\epsilon  +  \repst \mathcal{L}(g_\epsilon)  + \tfrac{v}{\epsilon} \sdot \nabla \phi_\epsilon. \]
By Lemma \ref{lemma-linear-decay}, $\vpbg$ generates a semi-group $e^{t\vpbg}$ on $H^s$. By Duhamel's principle, for the nonlinear system \eqref{vpb-ns-nonlinear},
\begin{align}
\label{duhamel-g}
\begin{aligned} g_{\epsilon}(t,x,v) &=e^{t G_{\varepsilon}} g_\epsilon(0,x,v)+\int_{0}^{t}  e^{(t-s) G_{\epsilon}} N(g_\epsilon) \bd s \\ &:=U^{\epsilon}(t) g_\epsilon(0)+\Psi^{\epsilon}\left(g_{\epsilon}\right). \end{aligned}
\end{align}
Denoting the Fourier transform with $x$ on torus  by $\mathcal{F}$, the associate discrete variable by $n \in \mathbb{Z}^3$, i.e,
\[ \mathcal{F}(g_\epsilon) = \hat{g}_\epsilon(t,n,v),  \]
then $\hat{g}_\epsilon$ satisfies the following equations:
\begin{align}
\label{vpb-linear-fourier}
\epsilon^2 \partial_t \hat{g}_\epsilon - \mathcal{L}\hat{g}_\epsilon + \epsilon \mathrm{i}(v \cdot n) \hat{g}_\epsilon + \epsilon \frac{\mathrm{i}(v \cdot n)}{|n|^{2}}  \int_{\mathbb{R}^{3}} \hat{g}_\epsilon \bdv =0.
\end{align}
Denoting the linear operator $\mathcal{B}_\epsilon(n)$ as following:
\[ \mathcal{A}_\epsilon(n) \hat{g}_\epsilon :=   \tfrac{1}{\epsilon^2} \mathcal{L}\hat{g}_\epsilon - \tfrac{1}{\epsilon} \mathrm{i}(v \cdot n) \hat{g}_\epsilon - \tfrac{1}{\epsilon} \frac{\mathrm{i}(v \cdot n)}{|n|^{2}}  \int_{\mathbb{R}^{3}} \hat{g}_\epsilon \bdv,  \]
and
\[ \mathcal{B}_\epsilon(n){h} :=     \mathcal{L}h-   \mathrm{i}(v \cdot n) h -   \epsilon^2 \frac{\mathrm{i}(v \cdot n)}{|n|^{2}}  \int_{\mathbb{R}^{3}} h \bdv,  \]
then \eqref{vpb-linear-fourier} can be rewritten as
\begin{align}
\label{bepsilon}
\partial_t \hat{g}_\epsilon = \tfrac{1}{\epsilon^2} \mathcal{B}_\epsilon(\epsilon n) \hat{g}_\epsilon .
\end{align}
Furthermore,
The spectrum of $\mathcal{B}_1(n)$ was clearly investigated in \cite[Lemma 3.16]{spectrum-vpb-2016-bipolar-decay} and \cite[Theorem 2.11]{spectrum-arxiv-vpb}. The spectrum of $\mathcal{B}_\epsilon$ is stated in Sec. \ref{sec-spec}.  Formally, $\mathcal{B}_\epsilon$ is similar to $\mathcal{B}_1$. But we can not use their spectrum results directly. Some modification should be made to their proof in \cite[Lemma 3.16]{spectrum-vpb-2016-bipolar-decay} and \cite[Theorem 2.11]{spectrum-arxiv-vpb} to adapt to the new  operator  $\mathcal{B}_\epsilon$.   We collect the semi-group structure results below first.  From the previous analysis, according to \eqref{duhamel-g}, we can represent the solution $g_\epsilon$:
\begin{align}
\label{duhamel-g-copy}
\begin{aligned}
g_{\epsilon}(t,x,v) &=e^{t G_{\varepsilon}} g_\epsilon(0,x,v)+ \int_{0}^{t}  e^{(t-s) G_{\epsilon}} N_1(g_\epsilon) \bd s + \int_{0}^{t}  e^{(t-s) G_{\epsilon}} N_2(g_\epsilon) \bd s \\
 &=U^{\epsilon} g_\epsilon(0)+\Psi^{\epsilon}_1\left(g_{\epsilon}\right) + \Psi^{\epsilon}_2\left(g_{\epsilon}\right). \end{aligned}
\end{align}
Recalling that
\[ U^{\epsilon} g(0)  = \mathcal{F}^{-1}\left( \exp\left(\repst \mathcal{B}_\epsilon(\epsilon n)t\right)\hat{g}_0 \right), \]
similar to  \cite{convergence-rate-linear-1975,briant-2015-be-to-ns},
we can represent and decompose $ U^{\epsilon} g_\epsilon(0)$ with the help of \eqref{semi-decompose}
\begin{align*}
\exp\left(\repst \mathcal{B}_\epsilon(\epsilon n)t\right)\hat{g}_\epsilon(0) =  \bds_{1}(\tfrac{t}{\epsilon^2}, \epsilon n) \hat{g}_\epsilon(0)  + \bds_{2}(\tfrac{t}{\epsilon^2}, \epsilon n) \hat{g}_\epsilon(0)
\end{align*}
While $|\epsilon n| \le r_0$, from  Theorem \ref{theorem-decom}, the high frequencies enjoy exponential decay.  From \eqref{semi-decompose},
\begin{align}
\label{low-fre}
\bds_{1}(\tfrac{t}{\epsilon^2}, \epsilon n) \hat{g}_\epsilon(0) = \sum_{j=-1}^{3} e^{\tfrac{t \lambda_{j}(|\epsilon n |)}{\epsilon^2} } \bds_{j}(\epsilon n )  \hat{g} \mathbf{1}_{|\epsilon n| \le r_0}
\end{align}
Plugging the eigenvalues into \eqref{low-fre}, it follows that
\begin{align}
\label{low-ei}
\begin{split}
\bds_{1}(\tfrac{t}{\epsilon^2}, \epsilon n)\hat{g}_\epsilon(0) = &  \sum_{j=\pm 1} e^{\{\tfrac{\pm\mathrm{i}( 6|n| + 5 n^2  ) }{6\epsilon} + \tfrac{a_{11} n^2 }{2} + \frac{1}{\epsilon^2} O(|\epsilon n|^3)\}t } \bdp_{j}(\epsilon n )  \mathbf{1}_{|\epsilon n| \le r_0}\hat{g}_\epsilon(0)\\
& +   e^{\{ \tfrac{a_{44} n^2 }{2} + \frac{1}{\epsilon^2} O(|\epsilon n|^3)\}t } \bdp_0(\epsilon n )   \mathbf{1}_{|\epsilon n| \le r_0} \hat{g}_\epsilon(0) \\
& +   e^{\{ \tfrac{a_{22} n^2 }{2} + \frac{1}{\epsilon^2} O(|\epsilon n|^3)\}t } \bdp_2(\epsilon n )   \mathbf{1}_{|\epsilon n| \le r_0} \hat{g}_\epsilon(0)
\end{split}
\end{align}
with
\begin{equation}
\label{semi-p}
\begin{aligned} \bdp_{2}(\epsilon n) \hat{g}_\epsilon=&\sum_{j=2,3}\left(\hat{m}_\epsilon \cdot w^{j}\right)\left(w^{j} \cdot v\right)  +\sum_{j=2,3}|\epsilon n| \mathrm{T}_{1,j}(w) \hat{g}_\epsilon  + \sum_{j=2,3}|\epsilon n|^2 \mathrm{T}_{2,j}(w) \hat{g}_\epsilon \\
& :=\bdp_{0,2}(n)\hat{g}_\epsilon  + \epsilon |n| \bdp_{1,2}(w)\hat{g}_\epsilon + \epsilon^2 |n|^2 \bdp_{2,2}(w)\hat{g}_\epsilon,  \\
& := \bdp_{0,2}(n)\hat{g}_\epsilon + \epsilon |n| \bdp_{r,2}(w)\hat{g}_\epsilon,  \\
 \bdp_{0}(\epsilon n) \hat{g}_\epsilon=&\left(\hat{\theta}_{\epsilon}(0)-\sqrt{\tfrac{2}{3}} \hat{\rho}_{\epsilon}(0)\right) \chi_{4}+|\epsilon n| \mathrm{T}_{0,1}(w) \hat{g}_\epsilon + |\epsilon n|^2 \mathrm{T}_{0,2}(w) \hat{g}_\epsilon\\
  & :=\bdp_{0,0}(n)\hat{g}_\epsilon(0)+ \epsilon |n| \bdp_{1,0}(w)\hat{g}_\epsilon + \epsilon^2 |n|^2 \bdp_{2,0}(w)\hat{g}_\epsilon\\
 & :=\bdp_{0,0}(n)\hat{g}_\epsilon+ \epsilon |n| \bdp_{r,0}(w)\hat{g}_\epsilon \\
 \bdp_{\pm }(\epsilon n) \hat{g}_\epsilon=& \tfrac{1}{2}\left[\left(\hat{m}_{\epsilon} \cdot \omega\right) \mp \tfrac{1}{|n|} \hat{\rho}_{\epsilon} \mp   (\hat{\rho}_\epsilon + \hat{\theta}_\epsilon) |n|\right](v \cdot \omega) \\ & + \epsilon \left(\tfrac{1}{2} \hat{\rho}_{\epsilon}\left(1+\sqrt{\tfrac{2}{3}} \chi_{4}\right)    + |n| \mathrm{T}_{1,\pm 1}(\xi) \hat{g}_{\epsilon}\right)  + \epsilon^2|n|^2 \mathrm{T}_{2,\pm 1}(\xi) \hat{g}_{\epsilon} \\
 & := \bdp_{0,\pm}(n)\hat{g}_\epsilon+ \epsilon |n|\bdp_{1,\pm}(w)\hat{g}_\epsilon + \epsilon^2 |n|^2\bdp_{2,\pm}(w)\hat{g}_\epsilon\\
 & := \bdp_{0,\pm}(n)\hat{g}_\epsilon+ \epsilon |n|\bdp_{r,\pm}(w)\hat{g}_\epsilon
  \end{aligned}
\end{equation}
In the above equations,
\[w = \tfrac{n}{|n|},~\hat{\rho}_\epsilon(0) = \intv \hat{g}_\epsilon(0) \bdv, \hat{m}_\epsilon(0) = \intv \hat{g}_\epsilon(0) v \bdv, \hat{\theta}_\epsilon(0) = \intv  \tfrac{|v|^2-3}{3}\hat{g}_\epsilon(0) \bdv.   \]
Compared to  \cite{convergence-rate-linear-1975} and \cite{briant-2015-be-to-ns}, the lower frequencies part of the semi-group for the linear Boltzmann operator and linear Vlasov-Poisson- Boltzmann operator are similar. Indeed, their structures are
\begin{align}
\label{structure-semi}
 e^{\{\tfrac{a^{j}}{\epsilon} + { b^j n^2 }  +  \epsilon c^j |n|^3 \}t } \left(\bdp_{0,j}(n )  + \epsilon \bdp_{r,j}(w ) \right),~~j=\pm 1,0, 2,
\end{align}
where all the constants are real number with $a^0=a^2 =0$, defined in \eqref{low-ei}. Besides, $\bdp_{0,j}(n )$ is made up  of the macroscopic variable and the remainder $\bdp_{r,j}(\epsilon n )$ is a bounded operator. The only difference between them is that the constant $a^j$, $b^j$ and $c^j$ are with different values. But the values do not affect the proof. In what follows, we will prove the convergence rates of electric field parts in details as a example.

\subsection{The convergence rate of the electric field}
Similar to \eqref{structure-semi} and \cite[Sec. 8.1.1]{briant-2015-be-to-ns},  we split the low frequencies of semi-group like this:
\begin{align}
\begin{aligned}
\bds_{1}(\tfrac{t}{\epsilon^2}, \epsilon n)   & = \sum_{j=-1}^{2} e^{\tfrac{t \lambda_{j}(|\epsilon n |)}{\epsilon^2} } \bdp_{j}(\epsilon n )    \mathbf{1}_{|\epsilon n| \le r_0} \\
 & =  \sum_{j=-1}^{2} e^{\{\tfrac{a^{j}}{\epsilon} + { b^j n^2 }  +  \epsilon c^j |n|^3 \}t } \left(\bdp_{0,j}(n )  + \epsilon \bdp_{r,j}(w ) \right)\hat{g} \mathbf{1}_{|\epsilon n| \le r_0}\\
& =  \sum_{j=-1}^{2} e^{\{\tfrac{a^{j}}{\epsilon} + { b^j n^2 }   \}t }  \bdp_{0,j}(n )   \\
& ~~ + \sum_{j=-1}^{2} \mathbf{1}_{|\epsilon n| \le r_0}  e^{\{\tfrac{a^{j}}{\epsilon} + { b^j n^2 }   \}t } \left( e^{   \epsilon c^j |n|^3 t }  -1 \right) \bdp_{0,j}(n )\\
& ~~ + \sum_{j=-1}^{2}   e^{\{\tfrac{a^{j}}{\epsilon} + { b^j n^2 }   \}t } \left(\mathbf{1}_{|\epsilon n| \le r_0}  -1 \right) \bdp_{0,j}(n )\\
& ~~ + \sum_{j=-1}^{2} e^{\{\tfrac{a^{j}}{\epsilon} + { b^j n^2 }  +  \epsilon c^j |n|^3 \}t }   \epsilon \bdp_{r,j}(w )  \hat{g}_\epsilon \mathbf{1}_{|\epsilon n| \le r_0}
\end{aligned}
\end{align}
Correspondingly,  we can represent $\Psi^\epsilon_1(g_\epsilon)$ as follows:
\begin{align}
\label{estPsi1Split}
\begin{aligned}
\Psi^{\epsilon}_1\left(g_{\epsilon}\right) & = \int_{0}^{t}  e^{(t-s) G_{\epsilon}} N_1(g_\epsilon)(s)  \bd s \\
& = \mathcal{F}^{-1}\left\lbrace   \int_0^t \left[\bds_{1}(\tfrac{t-s}{\epsilon^2}, \epsilon n)    + \bds_{2}(\tfrac{t-s}{\epsilon^2}, \epsilon n)\right] \mathcal{F}({N_1(g_\epsilon)})(s) \bd s    \right\rbrace \\
& = \sum_{j=-1}^{2}\mathcal{F}^{-1}\left\lbrace   \int_0^t  e^{\{\tfrac{a^{j}}{\epsilon} + { b^j n^2 }   \}(t-s)  }  \bdp_{0,j}(n )  \mathcal{F}({N_1(g_\epsilon)})(s) \bd s    \right\rbrace\\
& + \sum_{j=-1}^{2}\mathcal{F}^{-1}\left\lbrace     \mathbf{1}_{|\epsilon n| \le r_0}  e^{\{\tfrac{a^{j}}{\epsilon} + { b^j n^2 }   \}t } \left( e^{   \epsilon c^j |n|^3 t }  -1 \right) \bdp_{0,j}(n ) \mathcal{F}({N_1(g_\epsilon)})(s) \bd s    \right\rbrace\\
& + \sum_{j=-1}^{2}\mathcal{F}^{-1}\left\lbrace   \int_0^t  e^{\{\tfrac{a^{j}}{\epsilon} + { b^j n^2 }   \}(t-s)  } \left(\mathbf{1}_{|\epsilon n| \le r_0}  -1 \right)  \bdp_{0,j}(n )  \mathcal{F}({N_1(g_\epsilon)})(s) \bd s    \right\rbrace\\
& + \sum_{j=-1}^{2}\mathcal{F}^{-1}\left\lbrace   \int_0^t e^{\{\tfrac{a^{j}}{\epsilon} + { b^j n^2 }  +  \epsilon c^j |n|^3 \}(t-s) }    \mathbf{1}_{|\epsilon n| \le r_0}\epsilon \bdp_{r,j}(w )     \mathcal{F}({N_1(g_\epsilon)})(s) \bd s    \right\rbrace\\
& + \mathcal{F}^{-1}\left\lbrace   \int_0^t \left[ \bds_{2}(\tfrac{t-s}{\epsilon^2}, \epsilon n)\right] \mathcal{F}({N_1(g_\epsilon)})(s) \bd s    \right\rbrace  \\
&:= \sum_{j=-1}^{2}\Psi_{1,1,j}^\epsilon(g_\epsilon) + \sum_{j=-1}^{2}\Psi_{2,1,j}^\epsilon(g_\epsilon)  + \sum_{j=-1}^{2}\Psi_{3,1,j}^\epsilon(g_\epsilon) + \sum_{j=-1}^{2}\Psi_{r,1,j}^\epsilon(g_\epsilon) + \Psi_{1,2}^\epsilon(g_\epsilon).
\end{aligned}
\end{align}
Denote
\[ \Psi_1(t) h = \mathcal{F}^{-1}\left[\int_0^t\left( e^{ \tfrac{a_{44} n^2 }{2} (t-s) } P_{0,0} + e^{ \tfrac{a_{22} n^2 }{2} (t-s) } P_{0,2}\right)N_1(h)(s)\bd s \right]. \]
\begin{lemma}[The Electric field part]
\label{lemmaConvergenceElec}
Under the assumption of Theorem \ref{main-results-convergence-rate},  for any $T>0$, we can prove that
\begin{align}
\label{lemmaElectric-decay-rates-2}
\begin{aligned}
\|\int_0^T \Psi_1^\epsilon(s) g_\epsilon  \bd s -  \int_0^T \Psi_1(s) g_\epsilon \bd s\|_{H^\ell_x}    \lesssim \epsilon \sqrt T,~~\ell \le s-1,
\end{aligned}
\end{align}
and
\begin{align}
\label{lemmaElectric-decay-rates-l2}
\begin{aligned}
& \int_0^T \|\Psi_1^\epsilon(s) g_\epsilon  -   \Psi_1(s) g_\epsilon - \sum\limits_{\pm 1}\Psi_{1,1,\pm 1}^\epsilon(g)\|_{H^\ell_x}^2 \bd s    \\
& \lesssim \epsilon^2 T + \sup\limits_{0 \le s \le T} \|g_\epsilon(s)\|_{H^s_x}^2 \int_0^T \|(\bdp g_\epsilon - g)(s)\|_{L^2}^2\bd s,~~\ell \le s-1.
\end{aligned}
\end{align}
\end{lemma}
\begin{remark}
\label{remark-rate-1}
Compared to the Boltzmann case, \eqref{lemmaElectric-decay-rates-l2} is new. Indeed, according to \cite{briant-2015-be-to-ns}, only by the structure of semi-group, one can not directly deduce that $$\int_0^T \|\Psi_1^\epsilon(s) g_\epsilon  -   \Psi_1(s) g_\epsilon  \|_{H^\ell_x}^2 \bd s \lesssim \epsilon T.$$
Our new observation is that we can estimate $\Psi_1^\epsilon(s) g_\epsilon  -   \Psi_1(s) g_\epsilon - \sum\limits_{\pm 1}\Psi_{1,1,\pm 1}^\epsilon(g)$ to obtain \eqref{lemmaElectric-decay-rates-l2}.  Since the initial data are small enough, the last term on the right hand of \eqref{lemmaElectric-decay-rates-l2} can be absorbed. Furthermore, we shall show
\[ \sum\limits_{\pm 1}\Psi_{1,1,\pm 1}^\epsilon(g) = O(\epsilon), \]
in Lemma \ref{lemma-dd}.
\end{remark}
\begin{proof}
Compared to the Boltzmann case, this step is new. there exists new   difficulty resulting from the electric field. Indeed,  recalling
\[N_1(g_\epsilon)= ({v}  g_\epsilon - \nabla_v g_\epsilon     ) \sdot \nabla\phi_\epsilon, ~~  \Psi^{\epsilon}_1\left(g_{\epsilon}\right) = \int_{0}^{t}  e^{(t-s) G_{\epsilon}} \left( ({v}  g_\epsilon - \nabla_v g_\epsilon     ) \sdot \nabla\phi_\epsilon  \right) \bd s   \]
from the uniform estimates in Lemma \ref{lemma-nonlinear-decay}, i.e.,
\[ \|(g_\epsilon(t)\|_{H^s_x}^2 + \epsilon^2 \|\nabla_v g_\epsilon(t)\|_{H^{s-1}}^2 \le \bar{c}_0 \exp(- \bar{c} t), \]
there is no uniform estimates of $\nabla_v g_\epsilon$. Furthermore, $N_1$  does not belong to the orthogonal space of $\mathrm{Ker} \mathcal{L}$. Thus, the method in step two can not be directly used here.

In \eqref{estPsi1Split}, only
$\Psi_{1,1,0}^\epsilon(g_\epsilon) + \Psi_{1,1,2}^\epsilon(g_\epsilon)$ is what we need. The rest terms will turn to zero in some norm.

From \eqref{structure-semi}, we first compute the kernel part of $N_1$.
\begin{align}
\begin{aligned}
\intv N_1(g_\epsilon) \bdv & = - \nabla \phi_\epsilon \sdot \intv \nabla_v ( g_\epsilon \dm)  \bd v =0,\\
\intv N_1(g_\epsilon) v \bdv & = - \nabla \phi_\epsilon \sdot \intv \nabla_v ( g_\epsilon \dm) v \bd v =3\rho_\epsilon \nabla \phi_\epsilon,\\
\intv N_1(g_\epsilon) \left(  \tfrac{|v|^2 - 3}{3} \right)  \bdv & = - \nabla \phi_\epsilon \sdot \intv \nabla_v ( g_\epsilon \dm) \left(  \tfrac{|v|^2 - 3}{3} \right) \bd v =\tfrac{2}{3} u_\epsilon \nabla \phi_\epsilon.
\end{aligned}
\end{align}
Thus, plugging the kernel parts into \eqref{semi-p}, it follows that
\begin{equation}
\label{semi-n1}
\begin{aligned} \bdp_{2}(\epsilon n) N_1(g_\epsilon)=&\sum_{j=2,3}\left(3\hat\rho_\epsilon\!*\!\nabla \hat\phi_\epsilon \cdot w^{j}\right)\left(w^{j} \cdot v\right)  +\sum_{j=2,3}|\epsilon n| \mathrm{T}_{j}(n) \mathcal{F} \left(N_1(g_\epsilon)\right)    \\
:= &  \bdp_{0,2}(n)N_1(g_\epsilon) + \epsilon \bdp_{r,2}(n)\mathcal{F} \left(N_1(g_\epsilon)\right) , \\
 \bdp_{0}(\epsilon n) N_1(g_\epsilon)=& \hat u_\epsilon\!*\!\nabla \hat\phi_\epsilon \chi_{4}+|\epsilon n| \mathrm{T}_{0}(w) \mathcal{F} \left(N_1(g_\epsilon)\right)\\
 := & \bdp_{0,0}(n)N_1(g_\epsilon) + \epsilon \bdp_{r,0}(n)\mathcal{F} \left(N_1(g_\epsilon)\right), \\
 \bdp_{\pm 1}(\xi) N_1(g_\epsilon)=& \tfrac{1}{2} \left(3\hat\rho_\epsilon\!*\!\nabla \hat\phi_\epsilon \cdot \omega \mp \tfrac{2|n|}{3}\hat{u}_\epsilon\!*\!\nabla \hat\phi_\epsilon\right)\sdot(v\sdot \omega)    +|\epsilon n | \mathrm{T}_{\pm 1}(n) \mathcal{F} \left(N_1(g_\epsilon)\right)\\
:= &  \bdp_{0,\pm 1}(n)N_1(g_\epsilon) + \epsilon \bdp_{r,\pm 1}(n)\mathcal{F} \left(N_1(g_\epsilon)\right). \end{aligned}
\end{equation}

   For $\Psi_{1,1,\pm 1}^\epsilon(g^\epsilon)$, integrating by part with respect to $t$, we can infer that
\begin{equation}
\label{estTemp}
\begin{aligned}
\int_{0}^{T} \Psi_{1,1,\pm 1}^\epsilon(g_\epsilon) \bd t &=\sum_{n \in \mathbb{Z}^{d}-[0]} e^{\mathrm{i} n \cdot x} \int_{0}^{T} \int_0^t\left( e^{\{\tfrac{a^{\pm 1}}{\epsilon} + { b^{\pm 1} n^2 }   \}(t-s)  }  \bdp_{0,\pm 1}(n )  \mathcal{F}({N_1(g_\epsilon)})(s) \bd s \right) \bd t \\
&=\sum_{n \in \mathbb{Z}^{d}-[0]} e^{ \mathrm{i} n \cdot x} \frac{\epsilon}{ \mathrm{i} a^ {\pm 1 } +\epsilon b^{\pm 1}|n|^2}\\
& \times \left[\int_{0}^{T}\left(e^{\{\tfrac{a^{\pm 1}}{\epsilon} + { b^{\pm 1} n^2 }   \}(T-s)  }-1\right)  P_{0,\pm 1}(n )  \mathcal{F}({N_1(g)})(s) \bd s\right].
\end{aligned}
\end{equation}
By Plancherel theorem, its $L^2$ norm can be estimated:
\begin{align*}
\|\int_{0}^{T} \Psi_{1,1,\pm 1}^\epsilon(g^\epsilon) \bd t\|_{L^2}^2& \le 2   \sum_{n \in \mathbb{Z}^{d}-[0]} \frac{\epsilon^2}{  (a^ {\pm 1 })^2 +(\epsilon b^{\pm 1})^2|n|^4} \| \int_{0}^{T}   \bdp_{0,\pm 1}(n )  \mathcal{F}({N_1(g)})(s) \bd s \|_{L^2_v}^2\\
& \lesssim    \epsilon^2 \sum_{n \in \mathbb{Z}^{d}-[0]}  \frac{1}{  (a^ {\pm 1 })^2  } \| \int_{0}^{T}   \bdp_{0,\pm 1}(n )  \mathcal{F}({N_1(g)})(s) \bd s \|_{L^2_v}^2
\end{align*}
where we have used \eqref{low-ei} and \eqref{structure-semi}, i.e.,
\[ a^{\pm 1} = \pm |n| + \tfrac{5}{6} n^2,~~b^{\pm 1} = a_{11}. \]
Again,  from \eqref{semi-n1}, recalling that
\[ \bdp_{0,\pm 1}(n )  \mathcal{F}({N_1(g)})(s) = \tfrac{1}{2} \left(3\hat\rho_\epsilon\!*\!\nabla \hat\phi_\epsilon \cdot \omega \mp \tfrac{2|n|}{3}\hat{u}_\epsilon\!*\!\nabla \hat\phi_\epsilon\right)\sdot(v\sdot \omega)  \]
Furthermore, we can conclude that
\begin{align*}
  \| \int_{0}^{T}   \bdp_{0,\pm 1}(n )  \mathcal{F}({N_1(g_\epsilon)})(s) \bd s \|_{L^2_v}^2& \lesssim    \|\int_{0}^{T}      \tfrac{1}{2} \left(3\hat\rho_\epsilon\!*\!\nabla \hat\phi_\epsilon \cdot \omega \mp \tfrac{2|n|}{3}\hat{u}_\epsilon\!*\!\nabla \hat\phi_\epsilon\right)\sdot(v\sdot \omega) \bd s\|_{L^2_v}^2  \\
& \lesssim     T \left(\int_0^T \vert \hat\rho_\epsilon\!*\!\nabla \hat\phi_\epsilon\vert^2  \bd s \right)  +  |n|^2\left(\int_0^T \vert \hat{u}_\epsilon\!*\!\nabla \hat\phi_\epsilon \vert^2 \bd s \right).
\end{align*}
Thus, we can obtain that
\begin{align*}
\|\int_{0}^{T} \Psi_{1,1,\pm 1}^\epsilon(g^\epsilon) \bd t\|_{L^2}^2& \lesssim  \epsilon^2  T    \sum_{n \in \mathbb{Z}^{d}-[0]}\left(\int_0^T \vert \hat\rho_\epsilon\!*\!\nabla \hat\phi_\epsilon\vert^2  \bd s    +   \int_0^T \vert \hat{u}_\epsilon\!*\!\nabla \hat\phi_\epsilon \vert^2 \bd s \right).
\end{align*}
By Plancherel theorem, since we have obtained the exponential decay of $g_\epsilon$ in Lemma \ref{lemma-nonlinear-decay}, we obtain that
\begin{align}
\label{estConEl-1}
\|\int_{0}^{T} \Psi_{1,1,\pm 1}^\epsilon(g_\epsilon) \bd t\|_{L^2} \lesssim \sqrt{T} \epsilon.
\end{align}
Denoting
\[ d_{\pm 1}(n,s) =   \bdp_{0,\pm 1}(n )  \mathcal{F}({N_1(g_\epsilon)} - N_1(g))(s),  \]
Furthermore, noticing that
\begin{aligno}
   \sum_{\pm 1} (\Psi_{1,1,\pm 1}^\epsilon(g_\epsilon)  - \Psi_{1,1,\pm 1}^\epsilon(g))   = \sum_{n \in \mathbb{Z}^{d}-[0]} e^{\mathrm{i} n \cdot x} \int_0^t\left( e^{\{\tfrac{a^{\pm 1}}{\epsilon} + { b^{\pm 1} n^2 }   \}(t-s)  }  d_{\pm 1}({n,s}) \bd s \right),
\end{aligno}
by the Plancherel theorem and integration by part with respect to $t$,
\begin{aligno}
\label{estl22}
& \int_0^T \|\sum_{\pm 1}(\Psi_{1,1,\pm 1}^\epsilon(g_\epsilon)  - \Psi_{1,1,\pm 1}^\epsilon(g))\|_{L^2}^2 \bd t \\
& \lesssim \sum_{ \pm 1 \atop n \in \mathbb{Z}^{d}-[0]}  \int_0^T \|\int_0^t\left( e^{\{\tfrac{a^{\pm 1}}{\epsilon} + { b^{\pm 1} n^2 }   \}(t-s)  }  d_{\pm 1}(n,s) \bd s \right)\|_{L^2_v}^2 \bd t\\
& \lesssim  \sum_{\pm 1 \atop n \in \mathbb{Z}^{d}-[0]} \int_0^T \|\int_0^t\left( e^{\{  { b^{\pm 1} n^2 }   \}(t-s)  }  |d_{\pm 1}(n,s)| \bd s \right)\|_{L^2_v}^2 \bd t\\
& \lesssim \sum_{\pm 1 \atop n \in \mathbb{Z}^{d}-[0]} \int_0^T  \int_0^t   e^{\{  { b^{\pm 1} n^2 }   \}(t-s)  }  \|d_{\pm 1}(n,s)\|_{L^2_v}^2 \bd s    \bd t \\
&  \lesssim \sum_{\pm 1 \atop n \in \mathbb{Z}^{d}-[0]} \frac{1}{|n|^2} \int_0^T ( e^{\{  { b^{\pm 1} n^2 }   \}(T-s)  }  - 1)\|d_{\pm 1}(n,s)\|_{L^2_v}^2 \bd s    \bd t \\
& \lesssim    \sum_{\pm 1 \atop n \in \mathbb{Z}^{d}-[0]}   \tfrac{1}{|n|^2} \int_0^T      \|d_{\pm 1}(n,s)\|_{L^2_v}^2 \bd s,
\end{aligno}
where we have used
\begin{align*}
\|\int_0^t\left( e^{\{\tfrac{a^{\pm 1}}{\epsilon} + { b^{\pm 1} n^2 }   \}(t-s)  }  d_{\pm 1}(n,s) \bd s \right)\|_{L^2_v}^2 & \le \int_0^t  e^{\{  { b^{\pm 1} n^2 }   \}(t-s)  }  \|d_{\pm 1}(n,s)\|_{L^2_v} \bd s   \\
& \le \int_0^t  e^{\{  { b^{\pm 1} n^2 }   \}(t-s)  }    \bd s \int_0^t  e^{\{  { b^{\pm 1} n^2 }   \}(t-s)  }  \|d_{\pm 1}(n,s)\|_{L^2_v}^2 \bd s\\
& \lesssim \int_0^t  e^{\{  { b^{\pm 1} n^2 }   \}(t-s)  }  \|d_{\pm 1}(n,s)\|_{L^2_v}^2 \bd s.
\end{align*}
From \eqref{semi-n1}, we obtain that
\begin{align*}
\bdp_{0,\pm 1}(n )  \mathcal{F}({N_1(g_\epsilon)}) =   \left(3\hat\rho_\epsilon\!*\!\nabla \hat\phi_\epsilon \cdot \omega \mp \tfrac{2|n|}{3}\hat{u}_\epsilon\!*\!\nabla \hat\phi_\epsilon\right)\sdot(v\sdot \omega).
\end{align*}
Thus, we can obtain that
\begin{align*}
  \sum_{\pm 1 \atop n \in \mathbb{Z}^{d}-[0]} \tfrac{1}{|n|^2}       \|d_{\pm 1}(n,s)\|_{L^2_v}^2 & \lesssim \|\rho_\epsilon \sdot \nabla \phi_\epsilon - \rho \sdot \nabla \phi\|_{L^2}^2 + \|u_\epsilon \sdot \nabla \phi_\epsilon - u \sdot \nabla \phi\|_{L^2}^2 \\
  & \lesssim  \|(g_\epsilon,g)\|_{H^s_x}^2\|\bdp g_\epsilon - g\|_{L^2}^2.
\end{align*}
All together, it follows that
\begin{align}
\label{estddd}
\int_0^T \|\sum_{\pm 1}(\Psi_{1,1,\pm 1}^\epsilon(g_\epsilon)  - \Psi_{1,1,\pm 1}^\epsilon(g))\|_{L^2}^2 \bd t \lesssim  \sup\limits_{0 \le s \le T} \|g_\epsilon(s)\|_{H^s_x}^2 \int_0^T \|(\bdp g_\epsilon - g)(s)\|_{L^2}^2\bd s.
\end{align}
By the similar method, we can prove that
\begin{align}
\int_0^T \|\sum_{\pm 1}(\Psi_{1,0,\pm 1}^\epsilon(g_\epsilon)  - \Psi_{1,0,\pm 1}^\epsilon(g))\|_{H^\ell_x}^2 \bd t \lesssim  \sup\limits_{0 \le s \le T} \|g_\epsilon(s)\|_{H^s_x}^2 \int_0^T \|(\bdp g_\epsilon - g)(s)\|_{H^\ell_x}^2\bd s.
\end{align}

For  $\Psi_{2,1,j}^\epsilon(g^\epsilon)$, noticing that
\[ \|\Psi_{2,1,j}^\epsilon(g_\epsilon)\|_{L^2}^2 = \sum_{n \in \mathbb{Z}^{d}-[0]} e^{\mathrm{i} n \cdot x}   \int_0^t\left( e^{\{\tfrac{a^{j}}{\epsilon} + { b^j n^2 }   \}s } \left( e^{   \epsilon c^j |n|^3 s }  -1 \right)  \bdp_{0,j}(n )  \mathcal{F}({N_1(g_\epsilon)})(t-s) \bd s \right),   \]
thus, we have
\begin{align}
\begin{aligned}
\|\Psi_{2,1,j}^\epsilon(g^\epsilon)\|_{L^2}^2
& \lesssim  \sum_{n \in \mathbb{Z}^{d}-[0]} \epsilon^2   \| \int_0^t s e^{  { \tfrac{b^j}{2} n^2 }    s } |n|^3   \vert \bdp_{0,j}(n )  \mathcal{F}({N_1(g_\epsilon)})(t-s) \vert  \bd s \|_{L^2_v}^2 \\
& \lesssim \sum_{n \in \mathbb{Z}^{d}-[0]} \epsilon^2  \| \int_0^t \tfrac{1}{\sqrt{s}} e^{  { \tfrac{b^j}{4} n^2 }    s }   \vert \bdp_{0,j}(n )  \mathcal{F}({N_1(g_\epsilon)})(t-s) \vert  \bd s \|_{L^2_v}^2\\
& \lesssim \sum_{n \in \mathbb{Z}^{d}-[0]} \epsilon^2    \left(\int_0^t \tfrac{1}{\sqrt{s}} e^{  { \tfrac{b^j}{4} n^2 }    s }   \bd s \right)^2 \sup\limits_{0 \le s \le t}  \| \bdp_{0,j}(n )  \mathcal{F}({N_1(g_\epsilon)})(s) \|_{L^2_v} \\
& \lesssim  \epsilon^2   \sum_{n \in \mathbb{Z}^{d}-[0]}     \sup\limits_{0 \le s \le t}  \| \bdp_{0, j}(n )  \mathcal{F}({N_1(g_\epsilon)})(s) \|_{L^2_v}^2\\
& \lesssim \epsilon^2 \|g_\epsilon(0)  \|_{H^s_x}^4.
\end{aligned}
\end{align}
From the above inequality, we can obtain that
\begin{align}
\label{estConEl-2}
\|\int_{0}^{T} \Psi_{2,1,j}^\epsilon(g_\epsilon) \bd t\|_{L^2} \lesssim \sqrt T \epsilon,~~ \|\int_{0}^{T} \Psi_{2,1,j}^\epsilon(g_\epsilon) \bd t\|_{L^2((0,T);L^2)} \lesssim \sqrt T \epsilon.
\end{align}

For $\Psi_{3,0,j}^\epsilon(g_\epsilon)$,  noticing that
$$\mathbf{1}_{|\epsilon n| \le r_0}  -1 \le \tfrac{\epsilon |n| }{r_0},$$  it follows that
\begin{align}
\begin{aligned}
\|\Psi_{3,1,j}^\epsilon(g_\epsilon)\|_{L^2}^2
& \lesssim  \sum_{n \in \mathbb{Z}^{d}-[0]} \epsilon^2
 \| \int_0^t  e^{  {  {b^j}  n^2 }    s } |n|^2   \vert \bdp_{0,j}(n )  \mathcal{F}({N_1(g_\epsilon)})(t-s) \vert  \bd s \|_{L^2_v}^2 \\
& \lesssim \sum_{n \in \mathbb{Z}^{d}-[0]} \epsilon^2  |n|^2  \left(\int_0^t   e^{  { \tfrac{b^j}{4} n^2 }    s }   \| \bdp_{0,j}(n )  \mathcal{F}({N_1(g_\epsilon)})(t-s)\|_{L^2_v}    \bd s \right)^2 \\
& \lesssim \sum_{n \in \mathbb{Z}^{d}-[0]} \epsilon^2 |n|^2 \left(\int_0^t  e^{  {  {b^j}  n^2 }    s }  \bd s \right)^2    \sup\limits_{0 \le s \le t}  \| \bdp_{0,j}(n )  \mathcal{F}({N_1(g)})(s) \|_{L^2_v}^2 \\
& \lesssim  \epsilon^2  |n|^2 \cdot \tfrac{1}{|n|^4} \sum_{n \in \mathbb{Z}^{d}-[0]}     \sup\limits_{0 \le s \le t}  \| \bdp_{0, j}(n )  \mathcal{F}({N_1(g)})(s) \|_{L^2_v}^2\\
& \lesssim \epsilon^2 \|g_\epsilon(0)  \|_{H^s_x}^4.
\end{aligned}
\end{align}
This means that
\begin{align}
\label{estConEl-3}
\|\int_{0}^{T} \Psi_{3,1,j}^\epsilon(g^\epsilon) \bd t\|_{L^2} \lesssim \sqrt T \epsilon,~~ \|\int_{0}^{T} \Psi_{3,1,j}^\epsilon(g^\epsilon) \bd t\|_{L^2((0,T);L^2)} \lesssim \sqrt T \epsilon.
\end{align}

For the left term in the low frequency, from \eqref{semi-p} and \eqref{semi-n1},
\[ \Psi_{r,1,j}^\epsilon(g^\epsilon)=\mathcal{F}^{-1}\left\lbrace   \int_0^t e^{\{\tfrac{a^{j}}{\epsilon} + { b^j n^2 }  +  \epsilon c^j |n|^3 \}(t-s) }    \mathbf{1}_{|\epsilon n| \le r_0}\epsilon \bdp_{r,j}(w )     \mathcal{F}({N_1(g_\epsilon)})(s) \bd s    \right\rbrace,  \]
and
\[ \epsilon \bdp_{r,j}(w )     \mathcal{F}({N_1(g_\epsilon)})(s)= \mathbf{1}_{|\epsilon n| \le r_0} |\epsilon n| \mathrm{T}_j(w) \mathcal{F}({N_1(g_\epsilon)})(s),    \]
noticing that $T_j(w)$ is a bounded operator, thus it follows that
\begin{align}
\label{estaa}
\begin{aligned}
\|\Psi_{r,1,j}^\epsilon(g^\epsilon)\|_{L^2}^2
& \lesssim  \sum_{n \in \mathbb{Z}^{d}-[0]} \epsilon^2
 \| \int_0^t  e^{  {  \tfrac{b^j}{2}  n^2 }    s } |n|^2   \vert \bdp_{r    ,j}(n )  \mathcal{F}({N_1(g_\epsilon)})(t-s) \vert  \bd s \|_{L^2_v}^2 \\
& \lesssim \sum_{n \in \mathbb{Z}^{d}-[0]} \epsilon^2  |n|^2  \left(\int_0^{t}   e^{  { \tfrac{b^j}{4} n^2 }    s }   \| \mathrm{T}_j(w  )  \mathcal{F}({N_1(g)})(t-s)\|_{L^2_v}    \bd s \right)^2 \\
& \lesssim \sum_{n \in \mathbb{Z}^{d}-[0]} \epsilon^2 |n|^2  \left(\int_0^t  e^{  {  \tfrac{b^j}{4}  n^2 }    s }  \bd s \right)  \left(  \int_0^t   e^{  {  \tfrac{b^j}{4}  n^2 }    s }   \| \mathrm{T}_j(w  ) \mathcal{F}({N_1(g)})(t-s)\|_{L^2_v}^2 \bd s  \right)     \\
& \lesssim  \sum_{n \in \mathbb{Z}^{d}-[0]} \epsilon^2  |n|^2 \cdot \tfrac{1}{|n|^2}      \int_0^t     \|   \mathcal{F}({N_1(g)})(s)\|_{L^2_v}^2 \bd s\\
& \lesssim \epsilon^2 \int_0^t \|N_1(g_\epsilon)\|_{L^2}^2 \bd s.
\end{aligned}
\end{align}
\[  \]
In the above computation, we have used the fact that
\[ e^{  {  {b^j}  n^2 }    s } \le 1. \]
Recalling that
\[ \|N_1(g_\epsilon)\|_{L^2}^2 \le \|v g_\epsilon \nabla \phi_\epsilon\|_{L^2}^2 + \|\nabla_v g_\epsilon \cdot \nabla \phi_\epsilon\|_{L^2}^2\lesssim \|g_\epsilon\|_{H^s_x}^4 + \|g_\epsilon\|_{H^s_x}^2 \|\nabla_v g_\epsilon\|_{L^2}^2. \]
From the main result in Lemma \ref{lemma-nonlinear-decay}, the is no exponential decay for $\|\nabla_v g_\epsilon\|_{L^2}$.  Thus, we need to get a better estimates for $\|\nabla_v g_\epsilon\|_{L^2}$.    It can be dealt with by the following methods. From \eqref{est-v-uni},
\begin{align}
\label{est-v-uni-copy}
\begin{split}
& \dt \bigg( c_e \cdot   \mathfrak{E}_{\epsilon,1}^s(t) +  c_f \cdot   \mathfrak{E}_{\epsilon,2,2}^s(t) +  \mathfrak{E}_{\epsilon,2,1}^s(t)\bigg) +  c_d \|g_\epsilon\|_{H^{s}_{\Lambda_x}}^2 +    c_d\|\nabla_v g_\epsilon\|_{H^{s-1}_\Lambda}^2 \\
& \le C_s \sqrt{\mathfrak{E}_\epsilon^s(t)}\|g_\epsilon\|_{H^{s}_x}^2 + \tfrac{C_s}{\epsilon^2} \sqrt{\mathfrak{E}_\epsilon^s(t)}\cdot  \epsilon^2 \|\nabla_v g_\epsilon\|_{H^{s-1}_\Lambda}^2.
\end{split}
\end{align}
As long as  \eqref{est-v-uni-2} is satisfied, for any $t \ge 0$, we can obtain that
\begin{align}
\label{est-v-uni-copy-1}
\begin{split}
& \dt \bigg( c_e \cdot   \mathfrak{E}_{\epsilon,1}^s(t) +  c_f \cdot   \mathfrak{E}_{\epsilon,2,2}^s(t) +  \mathfrak{E}_{\epsilon,2,1}^s(t)\bigg) +  \tfrac{c_d}{2} \|g_\epsilon\|_{H^{s}_\Lambda}^2     \le 0.
\end{split}
\end{align}
Then, we can obtain that for any $t \ge 0$
\begin{align}
\label{est-v-int-t}
\int_0^t \| g_\epsilon\|_{H^{s}_\Lambda}^2  \bd x \lesssim \|g_\epsilon(0)\|_{H^s_x}^2  + \|\nabla_v g_\epsilon(0)\|_{H^{s-1}_x}^2 \lesssim C_0.
\end{align}

Thus, we obtain that
\begin{align}
\label{estConEl-4}
\|\int_{0}^{T} \Psi_{r,1,j}^\epsilon(g^\epsilon) \bd t\|_{L^2} \lesssim \sqrt T \epsilon,~~ \|\int_{0}^{T} \Psi_{r,1,j}^\epsilon(g^\epsilon) \bd t\|_{L^2((0,T);L^2)} \lesssim \sqrt T \epsilon.
\end{align}
Since the high frequency part $\Psi_{1,2}^\epsilon(g_\epsilon)$ enjoys a exponential decay, i.e.
\begin{align*}
\|\bds_{2}(\tfrac{t-s}{\epsilon^2}, \epsilon n) \mathcal{F}({N_1(g_\epsilon)})(s)\|_{L^2_v } \lesssim \exp(-\sigma\tfrac{t-s}{\epsilon^2})\|\mathcal{F}({N_1(g_\epsilon)})(s)\|_{L^2_v},
\end{align*}
we can obtain that
\begin{align*}
\|\Psi_{1,2}^\epsilon(g_\epsilon)\|_{L^2}^2 & \lesssim  \sum_{n \in \mathbb{Z}^{d}-[0]}
 \left( \int_0^t  \|\bds_{2}(\tfrac{t-s}{\epsilon^2}, \epsilon n) \mathcal{F}({N_1(g_\epsilon)})(s)  \|_{L^2_v}\bd s \right)^2  \\
& \lesssim \sum_{n \in \mathbb{Z}^{d}-[0]}   \left( \|\int_0^{t}  e^{-\sigma \tfrac{s}{\epsilon^2}} \mathcal{F}({N_1(g_\epsilon)})(t-s)   \bd s \|_{L^2_v} \bd s \right)^2 \\
& \lesssim \sum_{n \in \mathbb{Z}^{d}-[0]}  \left(\int_0^t  e^{-\sigma \tfrac{s}{\epsilon^2}} \bd s \right)  \left(  \int_0^t  e^{-\sigma \tfrac{s}{\epsilon^2}}  \| \mathcal{F}({N_1(g_\epsilon)})(t-s)\|_{L^2_v}^2 \bd s  \right)     \\
& \lesssim  \sum_{n \in \mathbb{Z}^{d}-[0]} \epsilon^2    \int_0^t     \|   \mathcal{F}({N_1(g)})(s)\|_{L^2_v}^2 \bd s\\
& \lesssim \epsilon^2 \int_0^t \|N_1(g_\epsilon)\|_{L^2}^2 \bd s.
\end{align*}
Thus, we can obtain that
\begin{align}
\label{estConEl2}
\|\int_{0}^{T} \Psi_{1,2}^\epsilon(g^\epsilon) \bd t\|_{L^2} \lesssim \sqrt T \epsilon,~~ \int_0^T\|\Psi_{1,2}^\epsilon(g^\epsilon) \|^2_{L^2} \bd s  \lesssim   T \epsilon^2.
\end{align}
Combing \eqref{estConEl-1}, \eqref{estConEl-2}, \eqref{estConEl-3}, \eqref{estConEl-4} and \eqref{estConEl2}, we can obtain that
\begin{align}
\label{step-1-1}
\|\int_0^T \Psi_1^\epsilon(s) g_\epsilon  \bd s -  \int_0^T \Psi_1(s) g_\epsilon \bd s\|_{L^2} \lesssim \epsilon \sqrt T.~~
\end{align}
Thus, we can complete the proof for $L^2$ norm case. Similarly, we can prove for $H^\ell_x$ case. Furthermore, since the term induced by the electric field already contains derivative with $v$, the $\ell$ is required to be less than $s-1$.

\end{proof}

\subsection{ The convergence rate of the bilinear part.  }
Recalling that
\[  \Psi^{\epsilon}_2\left(g_{\epsilon}\right) = \int_{0}^{t}  e^{(t-s) G_{\epsilon}} N_2(g_\epsilon) \bd s  =  \reps \int_{0}^{t}  e^{(t-s) G_{\epsilon}} {\Gamma}(g_\epsilon,g_\epsilon)(s) \bd s. \]
noticing that ${\Gamma}(g_\epsilon,g_\epsilon)$ lies in $\left(\mathrm{Ker} \mathcal{L}\right)^\perp$, thus only the remainder parts in \eqref{structure-semi} are to be considered. Thus the bad coefficent $\reps$ will be cancelled by the $\epsilon|n|$ before $\bdp_{1,i}(i=0,\pm 1, 2)$. Similar to the electri field parts, correspondingly,  we can represent $\Psi^\epsilon_1(g_\epsilon)$ as follows:
\begin{align}
\label{estPsi2Split}
\begin{aligned}
\Psi^{\epsilon}_2\left(g_{\epsilon}\right) & = \int_{0}^{t}  e^{(t-s) G_{\epsilon}} N_2(g_\epsilon)(s)  \bd s \\
& = \mathcal{F}^{-1}\left\lbrace   \int_0^t \left[\bds_{1}(\tfrac{t-s}{\epsilon^2}, \epsilon n)    + \bds_{2}(\tfrac{t-s}{\epsilon^2}, \epsilon n)\right] \mathcal{F}({N_2(g_\epsilon)})(s) \bd s    \right\rbrace \\
& = \sum_{j=-1}^{2}\mathcal{F}^{-1}\left\lbrace   \int_0^t  e^{\{\tfrac{a^{j}}{\epsilon} + { b^j n^2 }   \}(t-s)  }  |n|\bdp_{1,j}(n )  \mathcal{F}({\Gamma}(g_\epsilon,g_\epsilon))(s) \bd s    \right\rbrace\\
& + \sum_{j=-1}^{2}\mathcal{F}^{-1}\left\lbrace     \mathbf{1}_{|\epsilon n| \le r_0}  e^{\{\tfrac{a^{j}}{\epsilon} + { b^j n^2 }   \}t } \left( e^{   \epsilon c^j |n|^3 t }  -1 \right) |n|\bdp_{1,j}(n ) \mathcal{F}({\Gamma}(g_\epsilon,g_\epsilon))(s) \bd s    \right\rbrace\\
& + \sum_{j=-1}^{2}\mathcal{F}^{-1}\left\lbrace   \int_0^t  e^{\{\tfrac{a^{j}}{\epsilon} + { b^j n^2 }   \}(t-s)  } \left(\mathbf{1}_{|\epsilon n| \le r_0}  -1 \right) |n| \bdp_{1,j}(n )  \mathcal{F}({\Gamma}(g_\epsilon,g_\epsilon))(s) \bd s    \right\rbrace\\
& + \sum_{j=-1}^{2}\mathcal{F}^{-1}\left\lbrace   \int_0^t e^{\{\tfrac{a^{j}}{\epsilon} + { b^j n^2 }  +  \epsilon c^j |n|^3 \}(t-s) }    \mathbf{1}_{|\epsilon n| \le r_0}\epsilon |n|^2 \bdp_{2,j}(w )     \mathcal{F}({\Gamma}(g_\epsilon,g_\epsilon))(s) \bd s    \right\rbrace\\
& + \mathcal{F}^{-1}\left\lbrace   \int_0^t \left[ \bds_{2}(\tfrac{t-s}{\epsilon^2}, \epsilon n)\right] \mathcal{F}({N_2(g_\epsilon)})(s) \bd s    \right\rbrace  \\
&:= \sum_{j=-1}^{2}\Psi_{1,2,j}^\epsilon(g_\epsilon) + \sum_{j=-1}^{2}\Psi_{2,2,j}^\epsilon(g_\epsilon)  + \sum_{j=-1}^{2}\Psi_{3,2,j}^\epsilon(g_\epsilon) + \sum_{j=-1}^{2}\Psi_{r,2,j}^\epsilon(g_\epsilon) + \Psi_{2,2}^\epsilon(g_\epsilon).
\end{aligned}
\end{align}

Denote
\[ \Psi_2(t) h = \mathcal{F}^{-1}\left[\int_0^t\left( e^{ \tfrac{a_{44} n^2 }{2} (t-s) } |n|\bdp_{1,0} + e^{ \tfrac{a_{22} n^2 }{2} (t-s) } |n|\bdp_{1,2}\right){\Gamma}(h,h)(s)\bd s \right]. \]
\begin{lemma}[The bilinear part]
\label{lemmaConvergenceBilinear}
Under the assumption of of Theorem \ref{main-results-convergence-rate},  for any $T>0$, we can prove that
\begin{align}
\label{lemmaBilinear-decay-rates-2}
\begin{aligned}
\|\int_0^T \Psi_2^\epsilon(s) g_\epsilon  \bd s -  \int_0^T \Psi_2(s) g_\epsilon \bd s\|_{H^\ell_x}   \lesssim \sqrt\epsilon \sqrt T, ~~\ell \le s,
\end{aligned}
\end{align}
and
\begin{align}
\label{lemmaBilinear-decay-rates-l2}
\begin{aligned}
& \int_0^T \|\Psi_2^\epsilon(s) g_\epsilon  -   \Psi_2(s) g_\epsilon - \sum\limits_{\pm 1}\Psi_{1,2,\pm 1}^\epsilon(g)\|_{H^\ell_x}^2 \bd s    \\
& \lesssim \epsilon T + \sup\limits_{0 \le s \le T} \|g_\epsilon(s)\|_{H^s_x}^2 \int_0^T \|(\bdp g_\epsilon - g)(s)\|_{L^2}^2\bd s,~~\ell \le s.
\end{aligned}
\end{align}
\end{lemma}
\begin{proof}
For the \eqref{lemmaBilinear-decay-rates-2},  according to the decomposition \eqref{estPsi2Split}, the left proof are similar to the eletric part. For the high frequency parts, since there exists coefficient $\reps$, we only get $\sqrt{\epsilon}$ other than $\epsilon$ before \eqref{lemmaBilinear-decay-rates-2}.

For the \eqref{lemmaBilinear-decay-rates-l2}, denoting
\[ d_{2, \pm 1}(n,s) =   \bdp_{1,\pm 1}(n )  \mathcal{F}({{\Gamma}(g_\epsilon,g_\epsilon)} - {\Gamma}(g,g))(s),  \]
Furthermore, noticing that
\begin{aligno}
   \sum_{\pm 1} (\Psi_{1,2,\pm 1}^\epsilon(g_\epsilon)  - \Psi_{1,2,\pm 1}^\epsilon(g))   = \sum_{n \in \mathbb{Z}^{d}-[0]} e^{\mathrm{i} n \cdot x} \int_0^t\left( e^{\{\tfrac{a^{\pm 1}}{\epsilon} + { b^{\pm 1} n^2 }   \}(t-s)  }  d_{2,\pm 1}({n,s}) \bd s \right),
\end{aligno}
by the similar method of deducing \eqref{estl22}, we can obtain that
\begin{aligno}
\int_0^T \|\sum_{\pm 1}(\Psi_{1,2,\pm 1}^\epsilon(g_\epsilon)  - \Psi_{1,2,\pm 1}^\epsilon(g))\|_{L^2}^2 \bd t
& \lesssim    \sum_{\pm 1 \atop n \in \mathbb{Z}^{d}-[0]}     \int_0^T      \|d_{2,\pm 1}(n,s)\|_{L^2_v}^2 \bd s\\
& \lesssim \int_0^T \|{{\Gamma}(g_\epsilon,g_\epsilon)}  - {{\Gamma}(g,g)}\|_{L^2}^2 \bd s.
\end{aligno}
Decomposing
\[ g_\epsilon = \bdp g_\epsilon + g_\epsilon^\perp,  \]
we can decompose ${\Gamma}(g_\epsilon,g_\epsilon)(s)$ as follows
\[ {\Gamma}(g_\epsilon,g_\epsilon)(s) = {\Gamma}(\bdp g_\epsilon, \bdp g_\epsilon)(s) + {\Gamma}(\bdp g_\epsilon,  g_\epsilon^\perp)(s) + {\Gamma}( g_\epsilon^\perp, \bdp g_\epsilon)(s) + {\Gamma}(g_\epsilon^\perp, g_\epsilon^\perp)(s),  \]
From \eqref{est-v-uni}, we can obtain that
\begin{align}
\label{estes}
\int_0^\infty \|g_\epsilon^\perp(z)\|_{H^s_{x}}^2 \bd s \lesssim \epsilon^2.
\end{align}
Based on the above useful estimates, we can obtain that
\begin{aligno}
\int_0^T \|{{\Gamma}(g_\epsilon,g_\epsilon)}  - {{\Gamma}(g,g)}\|_{L^2}^2 \bd s     & \lesssim \epsilon^2 + \int_0^T \|{{\Gamma}(\bdp g_\epsilon,\bdp g_\epsilon)}  - {{\Gamma}(g,g)}\|_{L^2}^2 \bd s \\
&   \lesssim \epsilon^2 + \int_0^T \|{\mathcal{L}((\bdp g_\epsilon)^2  - g^2)}\|_{L^2}^2 \bd s\\
& \lesssim  \epsilon^2 + \sup\limits_{0 \le s \le T} \|g_\epsilon(s)\|_{H^s_x}^2 \int_0^T \|(\bdp g_\epsilon - g)(s)\|_{L^2}^2\bd s
\end{aligno}
where we have used   the propertis of $\Gamma$,
\[ \Gamma(\bdp g_\epsilon, \bdp g_\epsilon) = \mathcal{L}(\left(\bdp g_\epsilon\right)^2).  \]
\end{proof}

\subsection{The convergence rates of the linear part}Denoting
\[ U(t)h(0)= \mathcal{F}^{-1}\left[( e^{ \tfrac{a_{44} n^2 }{2} t } P_{0,0} + e^{ \tfrac{a_{22} n^2 }{2} t } P_{0,2})\hat{h}(0),\right] \]
the following lemma is to show
\begin{align*}
U^\epsilon(t) g_\epsilon(0) \to U(t) g_\epsilon(0).
\end{align*}
\begin{lemma}[The linear part]
\label{lemmaConvergenceLinear}
Under the assumption of of Theorem \ref{main-results-convergence-rate},  for any $T>0$, we can prove that
\begin{align}
\label{linear-decay-rates}
\begin{aligned}
\|\int_0^T U^\epsilon(s) g_\epsilon(0) \bd s  - \int_0^T U(s) g_\epsilon(0) \bd s\|_{H^\ell_x} \lesssim \epsilon,~~\ell \le s.
\end{aligned}
\end{align}
If the initial data satisfy \eqref{assumpIni-Con},
\begin{align}
\label{linear-decay-rates-L2}
\begin{aligned}
\int_0^T \|U^\epsilon(s) g_\epsilon(0) -  U(s) g_\epsilon(0)\|_{H^\ell_x}^2 \bd s \lesssim \epsilon^2,~~\ell \le s.
\end{aligned}
\end{align}
\end{lemma}

\begin{proof}
Compared to the nonlinear part, the linear part is easier. Indeed, noticing that the initial data is independent of time, the integration of $U^\epsilon(t) g_\epsilon(0)$ with respect to $t$ only acts on the semigroup. Thus, we can can better convergence rate than nonlinear case (there is no $T$ before $\epsilon$ on the right hand of \eqref{linear-decay-rates}).
For \eqref{linear-decay-rates-L2},  by the proof of Lemma \ref{lemmaConvergenceElec}, except for $\Psi_{1,0,\pm 1}(g^\epsilon)$, we first obtain the convergence rate in $L^2((0,T);L^2)$ norm then the $L^2$ norm of the average in time. From Remark \ref{remark-a2}, while the initial data are well-parepared, $\bdp_\pm \hat g_\epsilon(0) =O(\epsilon)$, then we complete the proof. Furthermore, if the initial data do not satisfy \eqref{assumpIni-Con}, we can not obtain \eqref{linear-decay-rates-L2} (see the remark on \cite[Sec. C.1.1]{briant-2015-be-to-ns}).

\end{proof}

From the previous analysis, owing to  the same structure of the semi-group, the proof of this lemma is the same to the Botlzmann case like \cite[Sec.8.1.2]{briant-2015-be-to-ns}.

Before we complete the proof of Theorem \ref{main-results-convergence-rate}, we prove that
\begin{lemma}
\label{lemma-dd} Under the assumption of of Theorem \ref{main-results-convergence-rate},
\begin{align}
\|\Psi_{1,1,\pm 1}^\epsilon(g)\|_{H^\ell_x}^2 \lesssim  \epsilon^2,~~ \|\Psi_{1,2,\pm 1}^\epsilon(g)\|_{H^\ell_x}^2 \lesssim  \epsilon^2,~~~~\ell \le s-2.
\end{align}
\end{lemma}
\begin{remark}
\label{remark-rate-2}
The goal of this lemma has been explained in Remark \ref{remark-rate-1}. The idea of this lemma is based on  \eqref{est111}. Indeed, after integration by part,
\[ \Psi_{1,1,\pm 1}^\epsilon(g) \approx O(\epsilon) + O(\epsilon)\Psi_{1,1,\pm 1}^\epsilon(g_t), \]
noticing that $g$ is governed by a parabolic type equation where one derivative with respect to $t$ is equal two times derivative with respect to $x$, this is why $ \ell \le s -2$.

\end{remark}
\begin{proof}
By the integration by parts with respect to $s$, we can obtain that
\begin{equation}
\label{est111}
\begin{aligned}
 \Psi_{1,1,\pm 1}^\epsilon(g_\epsilon)  &=\sum_{n \in \mathbb{Z}^{d}-[0]} e^{\mathrm{i} n \cdot x}  \int_0^t\left( e^{\{\tfrac{a^{\pm 1} }{\epsilon}i + { b^{\pm 1} n^2 }   \}(t-s)  }  \bdp_{0,\pm 1}(n )  \mathcal{F}({N_1(g)})(s) \bd s \right)  \\
&=\sum_{n \in \mathbb{Z}^{d}-[0]} e^{ \mathrm{i} n \cdot x} \frac{\epsilon}{ \mathrm{i} a^ {\pm 1 } +\epsilon b^{\pm 1}|n|^2}  \bdp_{0,\pm 1}(n)\mathcal{F}({N_1(g)})(t)   \\
& - \sum_{n \in \mathbb{Z}^{d}-[0]} e^{ \mathrm{i} n \cdot x} \frac{\epsilon}{ \mathrm{i} a^ {\pm 1 } +\epsilon b^{\pm 1}|n|^2}e^{\{\tfrac{a^{\pm 1}}{\epsilon} + { b^{\pm 1} n^2 }   \}(t)  }  \bdp_{0,\pm 1}(n )  \mathcal{F}({N_1(g)})(0) \\
& - \sum_{n \in \mathbb{Z}^{d}-[0]} e^{ \mathrm{i} n \cdot x} \frac{\epsilon}{ \mathrm{i} a^ {\pm 1 } +\epsilon b^{\pm 1}|n|^2} \int_{0}^{t}e^{\{\tfrac{a^{\pm 1}}{\epsilon} + { b^{\pm 1} n^2 }   \}(t-s)  }  P_{0,\pm 1}(n )  \mathcal{F}({N_1(g_t)})(s) \bd s\\
&:= E_1 + E_2 + E_3.
\end{aligned}
\end{equation}
Recalling that
\[
\bdp_{0,\pm 1}(n )  \mathcal{F}({N_1(g_\epsilon)}) =   \left(3\hat\rho_\epsilon\!*\!\nabla \hat\phi_\epsilon \cdot \omega \mp \tfrac{2|n|}{3}\hat{u}_\epsilon\!*\!\nabla \hat\phi_\epsilon\right)\sdot(v\sdot \omega),
  \]
by the Plancherel theorem and noticing that $b^{\pm 1} <0$, we can infer that
\begin{align}
\|E_1\|_{H^\ell_x}^2 + \|E_2\|_{H^\ell_x}^2 \lesssim \epsilon^2.
\end{align}
For the time derivative of $g$ in $E_3$, recalling that $g$ belongs to the kernel space of $\mathcal{L}$ with coefficients $\rho, u, \theta$ satisfying the Navier-Stokes-Poission system. In the parabolic type equation, one derivative with respect to $t$ is equal two times derivative with respect to $x$.  By the similar method of infering \eqref{estddd},  we can obtain that
\begin{align}
\|E_3\|_{H^\ell_x}^2   \lesssim \epsilon^2,~~\ell \le s -2.
\end{align}
For $\Psi_{1,2,\pm 1}^\epsilon(g)$,
\begin{align*}
 \Psi_{1,2,\pm 1}^\epsilon(g)  &=\sum_{n \in \mathbb{Z}^{d}-[0]} e^{\mathrm{i} n \cdot x}  \int_0^t  e^{\{\tfrac{a^{\pm 1}}{\epsilon} + { b^{\pm 1} n^2 }   \}(t-s)  }  |n|\bdp_{1,\pm 1}(n )  \mathcal{F}(\Gamma(g,g))(s) \bd s\\
  &=\sum_{n \in \mathbb{Z}^{d}-[0]} e^{\mathrm{i} n \cdot x}  \int_0^t  e^{\{\tfrac{a^{\pm 1}}{\epsilon} + { b^{\pm 1} n^2 }   \}(t-s)  }  |n|\bdp_{1,\pm 1}(n )  \mathcal{F}(\mathcal{L}(g^2)(s) \bd s
\end{align*}
By the similar method trick of deducing \eqref{estl22} and \eqref{est111}, we complete the proof.
\end{proof}
\subsection{The final proof of Theorem \ref{main-results-convergence-rate}}
From Lemma \ref{lemmaConvergenceElec}, Lemma \ref{lemmaConvergenceBilinear} and Lemma \ref{lemmaConvergenceLinear}, we can obtain that
\begin{align}
g(t) = U(t)g_0 + \Psi_1(t)(g) + \Psi_2(t)(g).
\end{align}
Recalling that
\begin{align}
g_\epsilon(t) = U^\epsilon(t)g_\epsilon(0)+ \Psi_1^\epsilon(t)(g_\epsilon) + \Psi_2^\epsilon(t)(g_\epsilon),
\end{align}
we can decompose $g_\epsilon - g$ as follows
\begin{aligno}
\label{estd}
g_\epsilon(t) - g(t) & = \bigg( U^\epsilon(t)g_\epsilon(0)+ \Psi_1^\epsilon(t)(g_\epsilon) + \Psi_2^\epsilon(t)(g_\epsilon) \\
& - U(t)g_\epsilon(0)- \Psi_1(t)(g_\epsilon) - \Psi_2(t)(g_\epsilon)\bigg)\\
& + \bigg( U(t)g_\epsilon(0) - U(t)g_0 \bigg)+ \bigg( \Psi_1(t)(g_\epsilon)  - \Psi_1(t)(g)\bigg) \\
&  + \bigg(\Psi_2(t)(g_\epsilon) - \Psi_2(t)(g) \bigg)\\
& := \mathbf{D}_\epsilon^1(t) + \mathbf{D}_\epsilon^2(t) + \mathbf{D}_\epsilon^3(t) + \mathbf{D}_\epsilon^4(t).
\end{aligno}
From Lemma \ref{lemmaConvergenceElec}, Lemma \ref{lemmaConvergenceBilinear}, Lemma \ref{lemmaConvergenceLinear} and Lemma \ref{lemma-dd}, we can deduce that
\begin{align}
\label{estd1}
\int_0^T \|\mathbf{D}_\epsilon^1(s)\|_{H^\ell_x}^2 \bd s  \lesssim   \epsilon T + \sup\limits_{0 \le s \le T} \|g_\epsilon(s)\|_{H^s_x}^2 \int_0^T \|(\bdp g_\epsilon - g)(s)\|_{H^\ell_x}^2\bd s, ~~ \ell \le s-2.
\end{align}
According to the settings on the initial data,
\begin{align}
\label{estd2}
\int_0^T \|\mathbf{D}_\epsilon^2(s)\|_{H^\ell_x}^2 \bd s \lesssim \epsilon^2, ~~\ell \le s.
\end{align}
But for $\mathbf{D}_\epsilon^3(t)$ and $\mathbf{D}_\epsilon^4(t)$, it is more complicated.
Recalling that
\[ \Psi_2(t) g_\epsilon = \mathcal{F}^{-1}\left[\int_0^t\left( e^{ \tfrac{a_{44} n^2 }{2} (t-s) } |n|\bdp_{1,0} + e^{ \tfrac{a_{22} n^2 }{2} (t-s) } |n|\bdp_{1,2}\right){\Gamma}(g_\epsilon,g_\epsilon)(s)\bd s, \right]\]
decomposing
\[ g_\epsilon = \bdp g_\epsilon + g_\epsilon^\perp,  \]
we can decompose ${\Gamma}(g_\epsilon,g_\epsilon)(s)$ as follows
\[ {\Gamma}(g_\epsilon,g_\epsilon)(s) = {\Gamma}(\bdp g_\epsilon, \bdp g_\epsilon)(s) + {\Gamma}(\bdp g_\epsilon,  g_\epsilon^\perp)(s) + {\Gamma}( g_\epsilon^\perp, \bdp g_\epsilon)(s) + {\Gamma}(g_\epsilon^\perp, g_\epsilon^\perp)(s),  \]
From \eqref{est-v-uni}, we can obtain that
\begin{align}
\label{estes-2}
\int_0^\infty \|g_\epsilon^\perp(z)\|_{H^s_{x}}^2 \bd s \lesssim \epsilon^2.
\end{align}
Based on the above useful estimates, we can obtain that
\begin{align}
 \int_0^t \|\left( \Psi_2(t) g_\epsilon - \Psi_2(t) \bdp g_\epsilon\right)  \|_{H^\ell_x}^2  \bd s \lesssim \epsilon^2.
\end{align}
By the propertis of $\Gamma$,
\[ \Gamma(\bdp g_\epsilon, \bdp g_\epsilon) = \mathcal{L}(\left(\bdp g_\epsilon\right)^2),  \]
then by the same trick of deducing \eqref{estaa}, we can obtain that
\begin{aligno}
\label{estbb}
& \int_0^T \| \Psi_2(s) (\bdp g_\epsilon -g)  \|_{H^\ell_x}^2 \bd s \\
& \lesssim  \int_0^T \| \mathcal{L} ((\bdp g_\epsilon)^2 -g^2)  \|_{H^\ell_x}^2 \bd s \\
& \lesssim  \sup\limits_{ 0 \le s \le T}   \|  g_\epsilon(s)\|_{H^s_x}^2     \int_0^T \|\left( \bdp  g_\epsilon - g \right)\|_{H^\ell_x}^2\bd s.
\end{aligno}
In summary,
\begin{align}
\label{estd4}
\|\int_0^t \mathbf{D}_4^\epsilon(s) \bd s \lesssim  \epsilon^2 +  \sup\limits_{ 0 \le s \le T}   \|  g_\epsilon(s)\|_{H^s_x}^2     \int_0^T \|\left( \bdp  g_\epsilon - g \right)\|_{H^\ell_x}^2\bd s.
\end{align}
For $\mathbf{D}_\epsilon^3(t)$, recalling that
\[ \Psi_1(t) h = \mathcal{F}^{-1}\left[\int_0^t\left( e^{ \tfrac{a_{44} n^2 }{2} (t-s) } P_{0,0} + e^{ \tfrac{a_{22} n^2 }{2} (t-s) } P_{0,2}\right)N_1(h)(s)\bd s \right], \]
from \eqref{semi-n1}, $\bdp_{0,0} N_1(g_\epsilon)$ and $\bdp_{0,2} N_1(g_\epsilon)$ are nonlinear and only related to the fluid parts, by the similar trick to \eqref{estbb}, we can
\begin{align}
\label{estd3}
\|\int_0^t \mathbf{D}_3^\epsilon(s) \bd s \|_{H^\ell_x} \lesssim  \epsilon^2 +  \sup\limits_{ 0 \le s \le T}   \|  g_\epsilon(s)\|_{H^s_x}^2     \int_0^T \|\left( \bdp  g_\epsilon - g \right)\|_{H^\ell_x}^2\bd s.
\end{align}
From \eqref{estd}, \eqref{estd1}, \eqref{estd2}, \eqref{estd3} and \eqref{estd4}, we finally obtain that for $\ell \le s-2$
\begin{aligno}
 \int_0^T \|( g_\epsilon  - g)(s)\|_{H^\ell_x}^2\bd s & \lesssim  \epsilon T   +  \sup\limits_{ 0 \le s \le T}   \|  g_\epsilon(s)\|_{H^s_x}^2     \int_0^T \|\left( \bdp  g_\epsilon - g \right)\|_{H^\ell_x}^2\bd s\\
 &     \lesssim  \epsilon T   +  \sup\limits_{ 0 \le s \le T}   \|  g_\epsilon(s)\|_{H^s_x}^2     \int_0^T \|\left(    g_\epsilon - g \right)\|_{H^\ell_x}^2\bd s
\end{aligno}
Since the initial data are small enough, we can deduce that
\begin{align}
\int_0^T \|( g_\epsilon  - g)(s)\|_{H^\ell_x}^2 \bd s\lesssim \epsilon T  .
\end{align}
Since $g_\epsilon$ converge weakly to $g$ in $L^\infty((0,T),H^s_x))$ space, it follows that
\begin{align}
\|g_\epsilon(t)- g(t)\|_{H^\ell_x}^2  \le 2\bar{c}_0 \exp(- \bar{c} t).
\end{align}
For any $T>0$,
\begin{align*}
 \int_T^\infty \|g_\epsilon(s) - g(s)\|_{H^\ell_x}^2 \bd s  \le  2\tfrac{ {\bar{c}_0}}{\bar{c}^2} \exp\left({-  {\hat{c}}  T}\right).
\end{align*}
Then by simple computation, as long as
\begin{align*}
 T \ge - \tfrac{1}{\bar{c}} \ln\left( -\tfrac{\bar{c}^2}{2{\bar{c}_0}} \epsilon \right),
\end{align*}
we can obtain that
\begin{align*}
\int_T^\infty \|g_\epsilon(s) - g(s)\|_{H^\ell_x}^2 \bd s \le  \epsilon.
\end{align*}
All together, we complete the proof.

\begin{remark}
\label{remark-hilbert}
We explain why and how the Hilbert expansion is used for the Boltzmann case in \cite{briant-2015-be-to-ns}. Indeed,  without estimates like \eqref{lemmaBilinear-decay-rates-l2}, to get the convergence rate of
\[ \Psi_2(t)(g_\epsilon) - \Psi_2(t)(g), \]
the Hilbert expansion is used. Indeed, by the Hilbert expansion (see \cite{guo2006NSlimit}),
\begin{align}
\label{esthilbert}
g_\epsilon = g  + \epsilon g_1 + \epsilon^2 g_2 +\cdots + \epsilon^n g_{n,\epsilon},
\end{align}
we have
\begin{align*}
\Psi_2(t) g_\epsilon - \Psi_2(t) g & = \mathcal{F}^{-1}\left[\int_0^t\left( e^{ \tfrac{a_{44} n^2 }{2} (t-s) } |n|\bdp_{1,0} + e^{ \tfrac{a_{22} n^2 }{2} (t-s) } |n|\bdp_{1,2}\right)[{\Gamma}(g_\epsilon,g_\epsilon)(s)- {\Gamma}(g,g)(s)] \bd s, \right]\\
&  = \mathcal{F}^{-1}\left[\int_0^t\left( e^{ \tfrac{a_{44} n^2 }{2} (t-s) } |n|\bdp_{1,0} + e^{ \tfrac{a_{22} n^2 }{2} (t-s) } |n|\bdp_{1,2}\right)[{\Gamma}(g_\epsilon - g ,g_\epsilon)(s) ] \bd s \right] \\
& +  \mathcal{F}^{-1}\left[\int_0^t\left( e^{ \tfrac{a_{44} n^2 }{2} (t-s) } |n|\bdp_{1,0} + e^{ \tfrac{a_{22} n^2 }{2} (t-s) } |n|\bdp_{1,2}\right)[{\Gamma}(g ,g_\epsilon  - g)(s) ] \bd s. \right]
\end{align*}
Since $a_{22} <0$  and $a_{44}<0$, by \eqref{esthilbert}, we can prove
\[  \|\Psi_2(t) g_\epsilon - \Psi_2(t) g\|_{L^2} \lesssim \epsilon. \]

\end{remark}
\section{Settings with random inputs.}
\label{sec-z}
This section consists of showing that the main results still hold while the random inputs are involved.  The random may come from both the initial data and collision kernel.  The random settings on the collision kernel are the same to those in \cite[Sec. 5]{liujin-2018-kinetic} where their initial data do not include fluid parts. Thus, the diffusive limit was not considered in their work.  In this work, we verify the fluid limit under random settings.
\subsection{Settings and functional space}
In the following, we introduce the similar assumptions on the kernel. To introduce the similar assumptions on kernel to that in Sec. \ref{sec-assump-determin}, we first introduce  the functional space
\begin{align*}
\|h\|_{\hszt}^2 = \int_{\mathbb{I}_z} \intps f^2 \bdv \bd x \bd  z,~~ \bd z = \bdp(z) d z.
\end{align*}

In the same way, we can define $H^s_\Lambda {L^2_z}$, $H^s_x {L^2_z}$, $H^s_\Lambda {L^{2\cap \infty}_z}$.

\begin{align*}
 \|f\|_{L^2_v}^2 =  \intv f^2 \bdv , ~~\|f\|_{L^2}^2 =  \intps f^2 \bdv \bd x,~~\|f\|_{L^2_\Lambda}^2 =  \intps f^2 \hat{v}\bdv \bd x,~~~~\hat{v}= 1 + |v|,\\
 \|f\|_{H^s_{x,z}}^2 =  \sum\limits_{k=0}^s\|\nabla^k_x f\|_{L^2L^{2\cap \infty}_z}^2 + \sum\limits_{k=0}^{s-1}\|\nabla^k_x  \partial_z f\|_{L^2L^2_z}^2,\\
 ~~\|f\|_{H^s_{\Lambda,z}}^2 =  \sum\limits_{k=0}^s \sum\limits_{i + j =k}\|\nabla^i_x \nabla^j_v f\|_{L^2_\Lambda L^{2\cap \infty}_z}^2  + \sum\limits_{k=0}^{s-1} \sum\limits_{i + j =k}\|\nabla^i_x \nabla^j_v  \partial_z f\|_{L^2_\Lambda L^2_z}^2,\\
 ~~\|f\|_{H^s_{z}}^2 =  \sum\limits_{k=0}^s \sum\limits_{i + j =k}\|\nabla^i_x \nabla^j_v f\|_{L^2 L^{2\cap\infty}_z}^2  + \sum\limits_{k=0}^{s-1} \sum\limits_{i + j =k}\|\nabla^i_x \nabla^j_v  \partial_z f\|_{L^2 L^2_z}^2.
\end{align*}

\subsection{A prior estimate}
In this subsection, we will establish similar results as Lemma \ref{lemma-nonlinear-decay}. The derivative of $g_\epsilon(0)$  with respect with $z$  is also assumed to bounded. As we will show, it is convenient for us to verify the fluid limits.  To achieve this, we assume that
\begin{align}
\label{initial-mean-nonlinear-ran}
\begin{split}
  { \intps \partial_z g_\epsilon(0) \m \bd v \bd x = 0,     \intps v \partial_z g_\epsilon(0) \m \bd v \bd x =0,}\\
  \partial_z\big( 3\intps (\tfrac{|v|^2}{3} - 1) g_\epsilon(0)\m \bd v \bd x +  \epsilon \|\nabla_x \phi_\epsilon(0)\|_{L^2}^2 \big)= 0.
  \end{split}
\end{align}
This assumption   \eqref{initial-mean-nonlinear-ran} plays a key role in deducing the estimates related to $\partial_z g_\epsilon$, specially for the similar process like \eqref{est-nonlinear-h1-all-0} \eqref{est-mean-difference} and \eqref{est-differ}. Indeed,
by global conservation law, we can deduce   the mean value of fluid parts  for $\partial_z g_\epsilon$ on tours.

Recalling \eqref{initial-mean-nonlinear},
\begin{align}
\label{initial-mean-nonlinear-recall}
\begin{split}
  { \intps g_\epsilon(0) \m \bd v \bd x = 0,     \intps v g_\epsilon(0) \m \bd v \bd x =0,}\\
  \big( 3\intps (\tfrac{|v|^2}{3} - 1) g_\epsilon(0)\m \bd v \bd x +  \epsilon \|\nabla_x \phi_\epsilon(0)\|_{L^2}^2 \big)= 0.
  \end{split}
\end{align}
The following two lemma is on the $L^2$ and $L^\infty$ estimates. It can be directly derived from Lemma \ref{lemma-nonlinear-decay}.

\begin{lemma}[$\hszti$]
\label{lemma-nonlinear-decay-ran-infty}
Under the assumption of \eqref{initial-mean-nonlinear-ran}, there exists some small enough constant $e_0$   such that as long as
\[  \|g_\epsilon(0,z)\|_{H^s_x \hszti}^2 + \epsilon^2 \|\nabla_v   g_\epsilon(0,z)\|_{H^{s-1} \hszti}^2   \le  e_0, \quad \forall z \in \mathbb{I}_z, \]
then  \eqref{vpb-ns-scaling-ns-corrector} admit a solution  $(g_\epsilon, \nabla \phi_\epsilon)$ satisfying for some $\bar{e}_0>0$ and $\bar{e}>0$ (all independent of $\epsilon$ while $\epsilon <1$) such that
\begin{align}
\label{est-lemma-nonlinear-exp-decay-ran-infty}
\|(g_\epsilon,\nabla_x \phi_\epsilon)(t)\|_{H^s_x \hszti}^2 + \epsilon^2 \|\nabla_v g_\epsilon(t)\|_{H^{s-1} \hszti}^2 \le \bar{e}_0 \exp(- \bar{e} t).
\end{align}
\end{lemma}
\begin{proof}
For each fixed $z \in \mathbb{I}_z$,  the nonlinear system
the non-linear system
\begin{align}
\label{vpb-ns-nonlinear-c-copy}
\partial_t   g_\epsilon(z) + \reps \vdot   g_\epsilon(z) +  \repst \mathcal{L}(  g_\epsilon(z))   - \tfrac{v}{\epsilon} \sdot \nabla   \phi_\epsilon(z) = N(g_\epsilon(z)),
\end{align}
where
\begin{align*}
N(g_\epsilon)&:=  N_1(g_\epsilon) + \partial_z N_2(g_\epsilon),\\
N_1(g_\epsilon)&:= ({v}  g_\epsilon - \nabla_v g_\epsilon     ) \sdot \nabla\phi_\epsilon,\\
N_2(g_\epsilon)&:= \reps {\Gamma}(g_\epsilon,g_\epsilon),
\end{align*}
different with \eqref{vpb-ns-nonlinear}, the solution $g_\epsilon$ to \eqref{vpb-ns-nonlinear-c-copy} is dependent of $z$. With the same method of obtaining \eqref{est-re-z}, we can deduce that
\begin{align}
\label{est-z-basic}
\begin{split}
& \dt \bigg( c_e \cdot   \mathfrak{E}_{\epsilon,1,z}^s(t,z) +  c_f \cdot   \mathfrak{E}_{\epsilon,2,2,z}^s(t,z) +  \mathfrak{E}_{\epsilon,2,1,z}^s(t,z)\bigg) \\
& +  c_d \|g_\epsilon(t,z)\|_{H^{s}_x L^2_z}^2 +   \tfrac{c_d}{\epsilon^2} \cdot \epsilon^2 \|\nabla_v g_\epsilon(t,z)\|_{H^{s-1}_\Lambda L^2_z}^2 \\
& \lesssim   \tfrac{\epsilon^2 \lambda_3 {(a_4 + C_\delta)} }{9C_\Omega}   \|\nabla_x \phi_\epsilon(t,z)\|_{L^2}^4      + c_e(\lambda_1 + \lambda_2) \sum\limits_{k=0}^{s} \vert \intps \nabla^k_x N(g_\epsilon) \sdot \nabla_x^k g_\epsilon\bdv \bd x \vert \\
& + (c_f +1 + c_e \lambda_3) \epsilon^2  \sum\limits_{i \ge 1\atop i + j =s} \vert\intps \big(\nabla^j_x \nabla^i_v N(g_\epsilon)\big) \sdot \big(\nabla^j_x \nabla^i_v g_\epsilon\big)\bdv \bd x   \vert\\
& + \lambda_4 \cdot  c_e \epsilon \sum\limits_{k=0}^{s-1} \vert \intps \nabla_x \nabla^k_x N(g_\epsilon) \sdot \nabla_v \nabla^k_x g_\epsilon\bdv \bd x  \vert.
\end{split}
\end{align}
with
\begin{align*}
\mathfrak{E}_{\epsilon,1,z}^s(t,z)=\|g_\epsilon(t,z)\|_{H^{s}  }^2+ \sum\limits_{i = 1, i + j =s}\epsilon^2\|\nabla^i_v \nabla_x^j g_\epsilon(z)\|_{L^2  }^2,\\
\mathfrak{E}_{\epsilon,2,1,z }^s(t,z)=\sum\limits_{i = 2, i + j =s}\epsilon^2\|\nabla^i_v \nabla_x^j  g_\epsilon(t,z)\|_{L^2 }^2,~~\mathfrak{E}_{\epsilon,2,2,z }^s(t,z) = \sum\limits_{i \ge 3, i + j =s}\epsilon^2\|\nabla^i_v \nabla_x^j g_\epsilon(t,z)\|_{L^2  }^2.
\end{align*}
With
\[ \mathcal{E}_\epsilon^s(t,z):= \|(g_\epsilon(t,z),\nabla_x \phi_\epsilon(t,z))\|_{H^s_x}^2 + \epsilon^2 \|\nabla_v g_\epsilon(t,z)\|_{H^{s-1}}^2, \]
in the same way of deducing \eqref{est-re-z},  we can conclude that
\begin{align}
\label{est-re-z-c-0}
\begin{split}
& \dt \bigg( c_e \cdot   \mathfrak{E}_{\epsilon,1}^s(t,z) +  c_f \cdot   \mathfrak{E}_{\epsilon,2,2}^s(t,z) +  \mathfrak{E}_{\epsilon,2,1}^s(t,z)\bigg) \\
& + \tfrac{c_e}{2\epsilon^2} \|(g_\epsilon)^\perp\|_{H^s_\Lambda}^2 +  c_d \|g_\epsilon(t,z)\|_{H^{s}_x}^2 +   \tfrac{c_d}{\epsilon^2} \cdot \epsilon^2 \|\nabla_v g_\epsilon(t,z)\|_{H^{s-1}_\Lambda}^2 \\
&\lesssim \big(\|g_\epsilon\|_{H^s}^2 + \|g_\epsilon\|_{H^s} + \|\rho_\epsilon\|_{H^s}\big)\|g_\epsilon\|_{H^s}^2      \\
&\lesssim (\|g_\epsilon\|_{H^s_x} +   \epsilon \|\nabla_v g_\epsilon\|_{H^{s-1}})( \|g_\epsilon\|_{H^{s}_x}^2 +   \tfrac{1}{\epsilon^2} \cdot \epsilon^2 \|\nabla_v g_\epsilon\|_{H^{s-1}}^2 )\\
& \lesssim   \sqrt{\mathfrak{E}_\epsilon^s(t,z)} \|g_\epsilon(t,z)\|_{H^{s}_x}^2  + \tfrac{1}{\epsilon^2} \sqrt{\mathfrak{E}_\epsilon^s(t,z)}\cdot  \epsilon^2 \|\nabla_v g_\epsilon(t,z)\|_{H^{s-1}}^2.
\end{split}
\end{align}
Thus, there exists some $C_{sz}$ such that
\begin{align}
\label{est-re-z-c}
\begin{split}
& \dt \bigg( c_e \cdot   \mathfrak{E}_{\epsilon,1}^s(t,z) +  c_f \cdot   \mathfrak{E}_{\epsilon,2,2}^s(t,z) +  \mathfrak{E}_{\epsilon,2,1}^s(t,z)\bigg) \\
& \le C_{sr}  \sqrt{\mathfrak{E}_\epsilon^s(t,z)} \|g_\epsilon(t,z)\|_{H^{s}_x}^2  +  C_{sr} \tfrac{1}{\epsilon^2} \sqrt{\mathfrak{E}_\epsilon^s(t,z)}\cdot  \epsilon^2 \|\nabla_v g_\epsilon(t,z)\|_{H^{s-1}}^2.
\end{split}
\end{align}

Since \eqref{est-re-z} holds for each $z \in \mathbb{I}_z$, based on the fact that the equivalent norm relation \eqref{constant-cl-cu}, similar to \eqref{est-to-e},  there exists some constant $e_0$ such that as long as
\begin{align}
\label{estsinitial}
\mathcal{E}_\epsilon^s(0,z) \le e_0= \frac{c_l c_d^2}{4 c_u C_{sr}^2 }, \forall z \in \mathbb{I}_z,
\end{align}
we can obtain that
\begin{align}
\label{estsup}
\mathfrak{E}_\epsilon^s(t,z) \le \frac{c_d^2}{4  C_{sr}^2 },~~\forall t >0.
\end{align}
Furthermore,
\begin{align}
\label{est-z-to-e-or}
\begin{split}
& \dt \bigg( c_e \cdot   \mathfrak{E}_{\epsilon,1}^s(t,z) +  c_f \cdot   \mathfrak{E}_{\epsilon,2,2}^s(t,z) +  \mathfrak{E}_{\epsilon,2,1}^s(t,z)\bigg)  \\
&  +  \tfrac{c_d}{2} \|g_\epsilon(t,z)\|_{H^{s}_x}^2 +   \tfrac{c_d}{2}  \|\nabla_v g_\epsilon(t,z)\|_{H^{s-1}_\Lambda}^2  + \tfrac{c_e}{2\epsilon^2} \|(g_\epsilon)^\perp(t,z)\|_{H^s_{\Lambda_x}}^2\le 0,
\end{split}
\end{align}
and
\begin{align}
\label{est-z-to-e}
\begin{split}
& \dt \bigg( c_e \cdot   \mathfrak{E}_{\epsilon,1}^s(t,z) +  c_f \cdot   \mathfrak{E}_{\epsilon,2,2}^s(t,z) +  \mathfrak{E}_{\epsilon,2,1}^s(t,z)\bigg)  \\
& +  \frac{c_d}{2 c_u}\bigg( c_e \cdot   \mathfrak{E}_{\epsilon,1,z}^s(t,z) +  c_f \cdot   \mathfrak{E}_{\epsilon,2,2}^s(t,z) +  \mathfrak{E}_{\epsilon,2,1}^s(t,z)\bigg) \le 0,
\end{split}
\end{align}
Since the inequality \eqref{est-z-to-e-or} and \eqref{est-z-to-e} work for each $z \in \mathbb{I}_z$, we can integrate them with respect $z$ over $\mathbb{I}_z$ and complete the proof.
\end{proof}
The following lemma provides the estimates of $\partial_z g_\epsilon$.
\begin{lemma}
\label{lemma-nonlinear-decay-ran}
Under the assumption of \eqref{initial-mean-nonlinear-ran}, there exists some small enough constant $d_0$ such that as long as
\[  \|(g_\epsilon,\nabla_x \phi_\epsilon)(0)\|_{H^s_{x,z}}^2 + \epsilon^2 \|\nabla_v g_\epsilon(0)\|_{H^{s-1}_z}^2  \le d_0, \]
then  \eqref{vpb-ns-scaling-ns-corrector} admit a solution  $(g_\epsilon, \nabla \phi_\epsilon)$ satisfying for some $\bar{d}_0>0$ and $\bar{d}>0$ (all independent of $\epsilon$ while $\epsilon <1$) such that
\begin{align}
\label{est-lemma-nonlinear-exp-decay-ran-z}
\|(g_\epsilon,\nabla_x \phi_\epsilon)(t)\|_{H^s_{x,z}}^2 + \epsilon^2 \|\nabla_v g_\epsilon(t)\|_{H^{s-1}_z}^2 \le \bar{d}_0 \exp(- \bar{d} t).
\end{align}
\end{lemma}

{\bf The estimate of $\partial_z g_\epsilon$.}

\begin{proof}
For the derivative of $g_\epsilon$ with respect to $z$, the $\partial_z g_\epsilon$ satifsies the following system non-linear system
\begin{align}
\label{vpb-ns-nonlinear-z}
\partial_t \partial_z g_\epsilon + \reps \vdot \partial_z g_\epsilon +  \repst \mathcal{L}(\partial_z g_\epsilon)   - \tfrac{v}{\epsilon} \sdot \nabla \partial_z \phi_\epsilon = N_z(g_\epsilon),
\end{align}
where
\begin{align*}
N_z(g_\epsilon)&:= \partial_z N_1(g_\epsilon) + \partial_z N_2(g_\epsilon) - \repst \mathcal{L}_z(\partial_z g_\epsilon)   ,\\
N_1(g_\epsilon)&:= ({v}  g_\epsilon - \nabla_v g_\epsilon     ) \sdot \nabla\phi_\epsilon,\\
N_2(g_\epsilon)&:= \reps {\Gamma}(g_\epsilon,g_\epsilon).
\end{align*}
The proof follows the idea of Lemma \ref{lemma-nonlinear-decay}.  From \eqref{conservation-law-nonlinear}, we can derive that
\begin{align}
\label{conservation-law-nonlinear-z}
\begin{split}
\tdt { \intps \partial_z g_\epsilon(t) \m \bd v \bd x = 0,   \tdt \intps v \partial_z g_\epsilon(t) \m \bd v \bd x =0,}\\
\tdt \big(  \intps (\tfrac{|v|^2}{2} - \tfrac{3}{2}) \partial_z g_\epsilon(t)\m \bd v \bd x +  \epsilon \partial_z \|\nabla_x \phi_\epsilon(t)\|_{L^2}^2 \big)= 0.
\end{split}
\end{align}
Furthermore, from \eqref{initial-mean-nonlinear}, we can deduce that
\begin{align}
\label{conservation-law-nonlinear-zz}
\begin{split}
  { \intps \partial_z g_\epsilon(0) \m \bd v \bd x = 0,   \tdt \intps v \partial_z g_\epsilon(0) \m \bd v \bd x =0,}\\
 \big(  \intps (\tfrac{|v|^2}{2} - \tfrac{3}{2}) \partial_z g_\epsilon(0)\m \bd v \bd x +  \epsilon \partial_z \|\nabla_x \phi_\epsilon(0)\|_{L^2}^2 \big)= 0.
\end{split}
\end{align}
Noticing that the left hands of \eqref{vpb-ns-nonlinear} and \eqref{vpb-ns-nonlinear-z} share the same structure, thus by the similar tricks in Lemma \ref{lemma-nonlinear-decay}, we can obtain the counterpart as  \eqref{est-nonlinear-all-hs},
\begin{align}
\label{est-z-h1}
\begin{split}
& \dt \bigg( c_e \cdot   \mathfrak{E}_{\epsilon,1,z,1}^s(t) +  c_f \cdot   \mathfrak{E}_{\epsilon,2,2,z,1}^s(t) +  \mathfrak{E}_{\epsilon,2,1,z,1}^s(t)\bigg) \\
& \qquad +   \tfrac{c_e}{\epsilon^2} \|(\partial_z g_\epsilon)^\perp\|_{H^{s-1}_x L^2_z}^2  + c_e\|\partial_z \rho_\epsilon\|_{H^{s-1} L^2_z}^2  + c_d\|\partial_z g_\epsilon\|_{H^{s-1}_\Lambda L^2_z}^2  \\
&\le  \tfrac{\epsilon^2 \lambda_3 {(a_4 + C_\delta)} }{9C_\Omega} \intz \|\nabla_x \phi_\epsilon\|_{L^2  }^2\|\nabla_x \partial_z \phi_\epsilon\|_{L^2  }^2 \bd z  \\
& + c_e(\lambda_1 + \lambda_2) \sum\limits_{k=0}^{s-1} \vert \intpsz \nabla^k_x N_z(g_\epsilon) \sdot \nabla_x^k \partial_z g_\epsilon\bdv \bd x  \bd z\vert \\
& + (c_f +1 + c_e \lambda_3) \epsilon^2  \sum\limits_{i \ge 1\atop i + j =s-1} \vert\intpsz \big(\nabla^j_x \nabla^i_v N_z(g_\epsilon)\big) \sdot \big(\nabla^j_x \nabla^i_v \partial_z g_\epsilon\big)\bdv \bd x \bd z \vert\\
& + \lambda_4 \cdot  c_e \epsilon \sum\limits_{k=0}^{s-2} \vert \intpsz \nabla_x \nabla^k_x N_z(g_\epsilon) \sdot \nabla_v \nabla^k_x \partial_z g_\epsilon\bdv \bd x \bd z\vert,
\end{split}
\end{align}
where
\begin{align*}
\mathfrak{E}_{r,\epsilon,1,z,1}^s(t)=\|\partial_z g_\epsilon\|_{H^{s-1}}^2+ \sum\limits_{i = 1, i + j =s-1}\epsilon^2\|\nabla^i_v \nabla_x^j \partial_z g_\epsilon\|_{L^2}^2,\\
\mathfrak{E}_{r,\epsilon,2,1,z,1}^s(t)=\sum\limits_{i = 2, i + j =s-1}\epsilon^2\|\nabla^i_v \nabla_x^j \partial_z g_\epsilon\|_{L^2}^2,~~\mathfrak{E}_{r,\epsilon,2,2,z,1}^s(t) = \sum\limits_{i \ge 3, i + j =s-1}\epsilon^2\|\nabla^i_v \nabla_x^j \partial_z g_\epsilon\|_{L^2}^2.
\end{align*}
Recalling that
\begin{align*}
N_z(g_\epsilon)&:= \partial_z N_1(g_\epsilon) + \partial_z N_2(g_\epsilon) - \repst \mathcal{L}_z(\partial_z g_\epsilon)   ,\\
N_1(g_\epsilon)&:= ({v}  g_\epsilon - \nabla_v g_\epsilon     ) \sdot \nabla\phi_\epsilon,\\
N_2(g_\epsilon)&:= \reps {\Gamma}(g_\epsilon,g_\epsilon).
\end{align*}
We first estimate  the term induced by eletric field, i.e., $N_1(g_\epsilon)$,
\begin{align*}
& \sum\limits_{ i+j =s-1}\vert \intpsz \nabla^j_x \nabla^i_v \partial_z\big( ({v}  g_\epsilon - \nabla_v g_\epsilon     ) \sdot \nabla\phi_\epsilon \big) \sdot \nabla^j_x \nabla^i_v \partial_z g_\epsilon\bdv \bd x \bd z \vert \\
& \le   \sum\limits_{ i+j =s-1} \vert \intpsz \nabla^j_x \nabla^i_v \partial_z\big(  {v}  g_\epsilon    \sdot \nabla\phi_\epsilon \big) \sdot \nabla^j_x \nabla^i_v \partial_z g_\epsilon\bdv \bd x \bd z\vert \\
& + \sum\limits_{ i+j =s-1} \vert \intpsz \nabla^j_x \nabla^i_v \partial_z\big( \nabla_v g_\epsilon      \sdot \nabla\phi_\epsilon \big) \sdot \nabla^j_x \nabla^i_v \partial_z g_\epsilon\bdv \bd x  \bd z\vert.
\end{align*}
We only show how to control the second term in the above inequality. The first term can be dealt with in the same way. Noticing that
\begin{align*}
 & \sum\limits_{ i+j =s-1} \vert \intpsz \nabla^j_x \nabla^i_v \partial_z\big( \nabla_v g_\epsilon      \sdot \nabla\phi_\epsilon \big) \sdot \nabla^j_x \nabla^i_v \partial_z g_\epsilon\bdv \bd x  \bd z\vert \\
 & = \sum\limits_{ i+j =s-1} \vert \intpsz \nabla^j_x \nabla^i_v \big(  \nabla_v  \partial_z g_\epsilon    \sdot \nabla\phi_\epsilon \big) \sdot \nabla^j_x \nabla^i_v \partial_z g_\epsilon\bdv \bd x \bd z\vert \\
&  + \sum\limits_{ i+j =s-1} \vert \intpsz \nabla^j_x \nabla^i_v \big(  \nabla_v   g_\epsilon    \sdot \nabla \partial_z \phi_\epsilon \big) \sdot \nabla^j_x \nabla^i_v \partial_z g_\epsilon\bdv \bd x \bd z\vert\\
& := T_1 + T_2
\end{align*}
We can split $T_1$ into three groups. Indeed,
\begin{align}
\label{splitterms-z-t1}
\begin{aligned}
T_1 &= \sum\limits_{ i+j =s-1} \vert \intpsz \nabla^j_x \nabla^i_v \big( \nabla_v \partial_z g_\epsilon      \sdot \nabla\phi_\epsilon \big) \sdot \nabla^j_x \nabla^i_v g_\epsilon\bdv \bd x  \bd z\vert\\
& \le \sum\limits_{ i+j =s-1} \vert \intpsz \nabla^j_x  \big( \nabla_v \nabla^i_v \partial_z g_\epsilon      \sdot \nabla\phi_\epsilon \big) \sdot \nabla^j_x \nabla^i_v g_\epsilon\bdv \bd x  \bd z\vert \\
& \le  \sum\limits_{ i+j =s-1} \vert \intpsz   \big( \nabla_v \nabla^j_x \nabla^i_v \partial_z g_\epsilon      \sdot \nabla\phi_\epsilon \big) \sdot \nabla^j_x \nabla^i_v g_\epsilon\bdv \bd x \bd z\vert\\
& + \sum\limits_{  i+j =s-1, i \ge 1 \atop k + l = j, l \ge 1} \vert \intpsz   \big(   \nabla^k_x \nabla^{ i +1}_v \partial_z g_\epsilon      \sdot \nabla_x^{l+1}\phi_\epsilon \big) \sdot \nabla^j_x \nabla^i_v g_\epsilon\bdv \bd x \bd z\vert\\
& +  \vert \intpsz   \big(    \nabla_v \partial_z g_\epsilon      \sdot \nabla_x^{s}\phi_\epsilon \big) \sdot \nabla^j_x \nabla^i_v g_\epsilon\bdv \bd x \bd z \vert\\
& = T_{11} + T_{12} +T_{13}.
\end{aligned}
\end{align}
For $T_{11}$, by integration by part and noticing that $\|\rho_\epsilon(z)\|_{H^s} = \|\phi_\epsilon(z)\|_{H^{s+2}}$, thus we can infer after integration by parts
\begin{align*}
T_{11} &\le  \sum\limits_{i+j =s-1} \vert \intpsz   \big( \nabla_v \nabla^j_x \nabla^i_v \partial_z g_\epsilon      \sdot \nabla\phi_\epsilon \big) \sdot \nabla^j_x \nabla^i_v g_\epsilon\bdv \bd x \bd z \vert\\
& \quad =  \sum\limits_{i+j =s-1} \vert \intpsz   \big(  \nabla^j_x \nabla^i_v \partial_z g_\epsilon      \sdot \nabla\phi_\epsilon \big) \sdot \nabla^j_x \nabla^{i+1}_v g_\epsilon\bdv \bd x \bd z \vert\\
&  \quad +  \sum\limits_{i+j =s-1} \vert \intpsz  v \big(  \nabla^j_x \nabla^i_v \partial_z g_\epsilon      \sdot \nabla\phi_\epsilon \big) \sdot \nabla^j_x \nabla^{i}_v g_\epsilon\bdv \bd x \bd z \vert\\
& \quad \lesssim \|\nabla \phi_\epsilon\|_{L^\infty_x L^\infty_z}\|g_\epsilon\|_{H^s_{\Lambda,z}}^2\\
& \quad \lesssim \|g_\epsilon\|_{H^s_{x}L^\infty_z}\|g_\epsilon\|_{H^s_{\Lambda,z}}^2.
\end{align*}
For $T_{12}$, noticing that $l+1 \le s-1$, it follows that
\begin{align}
\label{phi-z-infty}
\|\nabla^l \phi_\epsilon\|_{L^\infty_x L^\infty_z} \lesssim \|g_\epsilon\|_{H^{l}L^\infty_z} \lesssim\|g_\epsilon\|_{H^{l+1}_{x,z}}.
\end{align}
Then we can directly obtain
\begin{align*}
T_{12} &\le  \sum\limits_{  i+j =s-1, i \ge 1 \atop k + l = j, l \ge 1} \vert \intpsz   \big(   \nabla^k_x \nabla^{ i +1}_v \partial_z g_\epsilon      \sdot \nabla_x^{l+1}\phi_\epsilon \big) \sdot \nabla^j_x \nabla^i_v g_\epsilon\bdv \bd x \bd z\vert\\
& \quad \lesssim \|\nabla^{l +1} \phi_\epsilon\|_{L^\infty_x L^\infty_z}\|g_\epsilon\|_{H^s_{\Lambda,z}}^2\\
& \quad \lesssim \|g_\epsilon\|_{H^s_{x}L^\infty_z}\|g_\epsilon\|_{H^s_{\Lambda,z}}^2.
\end{align*}
With \eqref{phi-z-infty}, for $T_{13}$, we can also obtain that
\begin{align*}
T_{13} &\le  \vert \intpsz   \big(    \nabla_v \partial_z g_\epsilon      \sdot \nabla_x^{s}\phi_\epsilon \big) \sdot \nabla^j_x \nabla^i_v g_\epsilon\bdv \bd x \bd z \vert\\
& \quad \lesssim \|\nabla^{s} \phi_\epsilon\|_{L^\infty_x L^\infty_z}\|g_\epsilon\|_{H^s_{\Lambda,z}}^2\\
& \quad \lesssim \|g_\epsilon\|_{H^s_xL^\infty_z}\|g_\epsilon\|_{H^s_{\Lambda,z}}^2.
\end{align*}
In summary, we can obtain
\begin{align}
\label{est-t1}
T_1  \lesssim \|g_\epsilon\|_{H^s_xL^\infty_z}\|g_\epsilon\|_{H^s_{\Lambda,z}}^2.
\end{align}
Similary, $T_2$ also can be split into two groups
\begin{align}
\label{splitterms-z-t2}
\begin{aligned}
T_2 &= \sum\limits_{ i+j =s-1} \vert \intpsz \nabla^j_x \nabla^i_v \big( \nabla_v  g_\epsilon      \sdot \nabla_x \partial_z \phi_\epsilon \big) \sdot \nabla^j_x \nabla^i_v g_\epsilon\bdv \bd x  \bd z\vert\\
& \le  \sum\limits_{  i+k + l  =s-1 \atop  i+j = s-1, l \le s-2 } \vert \intpsz   \big(   \nabla^k_x \nabla^{ i +1}_v   g_\epsilon      \sdot \nabla_x^{l+1} \partial_z \phi_\epsilon \big) \sdot \nabla^j_x \nabla^i_v g_\epsilon\bdv \bd x \bd z\vert\\
& +  \vert \intpsz   \big(    \nabla_v  g_\epsilon      \sdot \nabla_x^{s} \partial_z\phi_\epsilon \big) \sdot \nabla^j_x \nabla^i_v g_\epsilon\bdv \bd x \bd z \vert\\
& = T_{21} + T_{22}.
\end{aligned}
\end{align}
For $T_{21}$, since $ l + 1 \le s -1 $, it follows that
\begin{align*}
\|\nabla^l \partial_z\phi_\epsilon\|_{L^\infty_x L^\infty_z} \lesssim \|g_\epsilon\|_{H^{l}L^\infty_z} \lesssim\|g_\epsilon\|_{H^{l+1}_{x,z}}.
\end{align*}
With the help of the above inequality, we can deduce that
\begin{align*}
& \vert \intpsz   \big(   \nabla^k_x \nabla^{ i +1}_v   g_\epsilon      \sdot \nabla_x^{l+1} \partial_z \phi_\epsilon \big) \sdot \nabla^j_x \nabla^i_v g_\epsilon\bdv \bd x \bd z\vert  \\
 & \lesssim \vert \int_{\mathbb{I}_z} \int_{\mathbb{T}^3}  |\nabla^{l+1}_x\partial_z\phi_\epsilon|  \|\nabla^k_x \nabla^{ i +1}_v   g_\epsilon(t,x,z)\|_{L^2_v} \|\nabla^j_x \nabla^i_vg_\epsilon\|_{L^2_v} \bd x \bd z\vert\\
&  \lesssim \vert \int_{\mathbb{I}_z} |\nabla^{l+1}_x\partial_z\phi_\epsilon(z)|_{L^\infty_x}  \int_{\mathbb{T}^3}   \|\nabla^k_x \nabla^{ i +1}_v   g_\epsilon(t,x,z)\|_{L^2_v} \|\nabla^j_x \nabla^i_vg_\epsilon\|_{L^2_v} \bd x \bd z\vert  \\
& \lesssim \vert \int_{\mathbb{I}_z} |\nabla^{l+1}_x\partial_z\phi_\epsilon(z)|_{L^\infty_x}  \|g_\epsilon(z)\|_{H^s} \|g_\epsilon(z)\|_{H^{s-1}} \bd z\vert \\
& \lesssim \|g_\epsilon\|_{H^{s}_z} \vert \int_{\mathbb{I}_z} |\nabla^{l+1}_x\partial_z\phi_\epsilon(z)|_{L^\infty_x}  \|g_\epsilon(z)\|_{H^s}  \bd z\vert\\
& \lesssim \|g_\epsilon\|_{H^{s}_z} \vert \int_{\mathbb{I}_z} \| \partial_z\phi_\epsilon(z)\|_{H^{s-1}_x}  \|g_\epsilon(z)\|_{H^s}  \bd z\vert\\
& \lesssim  \|\partial_z g_\epsilon\|_{H^{s-1}_x}\|g_\epsilon\|_{H^{s}_z}^2.
\end{align*}
Noticing, by integration by parts with respect to $x$, we can infer
\begin{align*}
& \vert \intpsz   \big(    \nabla_v  g_\epsilon      \sdot \nabla_x^{s} \partial_z\phi_\epsilon \big) \sdot \nabla^j_x \nabla^i_v g_\epsilon\bdv \bd x \bd z \vert \\
& \le \vert \intpsz   \big(   \nabla_x \nabla_v  g_\epsilon      \sdot \nabla_x^{s-1} \partial_z\phi_\epsilon \big) \sdot \nabla^j_x \nabla^i_v g_\epsilon\bdv \bd x \bd z \vert\\
& + \vert \intpsz   \big(    \nabla_v  g_\epsilon      \sdot \nabla_x^{s-1} \partial_z\phi_\epsilon \big) \sdot \nabla^{j+1}_x \nabla^i_v g_\epsilon\bdv \bd x \bd z \vert.
\end{align*}
Then by employing the same way of dealing  with $T_{21}$,

Combining \eqref{est-z-basic} and \eqref{est-z-h1}, we can infer that
\begin{align*}
\vert \intpsz   \big(    \nabla_v  g_\epsilon      \sdot \nabla_x^{s} \partial_z\phi_\epsilon \big) \sdot \nabla^j_x \nabla^i_v g_\epsilon\bdv \bd x \bd z \vert \lesssim \|\partial_z g_\epsilon\|_{H^{s-1}_x}\|g_\epsilon\|_{H^{s}_z}^2.
\end{align*}
All together, we can infer that
In summary, we can obtain
\begin{align}
\label{est-t2}
T_2  \lesssim \|\partial_z g_\epsilon\|_{H^{s-1}_{x,z}}\|g_\epsilon\|_{H^s_{\Lambda,z}}^2.
\end{align}
According to \eqref{est-t1} and \eqref{est-t2},  we can conclude that
\begin{align}
\label{est-z-n1-1}
\sum\limits_{ i+j =s-1}\vert \intpsz \nabla^j_x \nabla^i_v \partial_z\big( ({v}  g_\epsilon - \nabla_v g_\epsilon     ) \sdot \nabla\phi_\epsilon \big) \sdot \nabla^j_x \nabla^i_v \partial_z g_\epsilon\bdv \bd x \bd z \vert \lesssim \| g_\epsilon\|_{H^{s}_{x,z}}\|g_\epsilon\|_{H^s_{\Lambda,z}}^2.
\end{align}
In the similar way, we can infer that
\begin{align}
&\label{est-z-n1-2} \sum\limits_{k=0}^{s-1} \vert \intpsz \nabla^k_x \partial_z N_1(g_\epsilon) \sdot \nabla_x^k \partial_z g_\epsilon\bdv \bd x  \bd z\vert \lesssim \| g_\epsilon\|_{H^{s}_{x,z}}\|g_\epsilon\|_{H^s_{\Lambda,z}}^2,\\
& \label{est-z-n1-3} \sum\limits_{k=0}^{s-2} \vert \intpsz \nabla_x \nabla^k_x \partial_z N_1(g_\epsilon) \sdot \nabla_v \nabla^k_x \partial_z g_\epsilon\bdv \bd x \bd z\vert \lesssim \| g_\epsilon\|_{H^{s}_{x,z}}\|g_\epsilon\|_{H^s_{\Lambda,z}}^2.
\end{align}
From \eqref{est-z-n1-1}, \eqref{est-z-n1-2} and \eqref{est-z-n1-3}, we can obtain for $N_1(g_\epsilon)$,  we can conclude that
 \begin{aligno}
 \label{estzN1}
& (\lambda_1 + \lambda_2)\sum\limits_{k=0}^{s-1} \vert \intpsz \nabla^k_x \partial_z N_1(g_\epsilon) \sdot \nabla_x^k \partial_z g_\epsilon\bdv \bd x \bd z \vert \\
&  + \epsilon^2 (c_f +1 + c_e \lambda_3) \sum\limits_{i \ge 1\atop i + j =s-1} \vert\intpsz \big(\nabla^j_x \nabla^i_v \partial_z\big(N_1(g_\epsilon)\big)\big) \sdot \big(\nabla^j_x \nabla^i_v \partial_z g_\epsilon\big)\bdv \bd x \bd z \vert \\
& + \lambda_4 c_e \epsilon \sum\limits_{k=0}^{s-2} \vert \intpsz \nabla_x \nabla^k_x \partial_z\big( N_1(g_\epsilon) \big) \sdot \nabla_v \nabla^k_x \partial_z g_\epsilon\bdv \bd x \bd z\vert  \\
& \lesssim    \|  g_\epsilon\|_{H^{s}_{x,z}} \|  g_\epsilon\|_{H^{s}_{\Lambda,z}}^2  .
\end{aligno}

Recalling that the source term $N_z(g_\epsilon)$ in \eqref{est-z-h1} is made up of three parts, i.e.,
$$N_z(g_\epsilon) = \partial_z N_1(g_\epsilon) + \partial_z N_2(g_\epsilon) - \repst \mathcal{L}_z(  g_\epsilon),$$
since \eqref{estzN1}  are only  about $N_1(g_\epsilon)$,
then the left things are to estimate the collision term in \eqref{est-z-h1}. First for $\repst \mathcal{L}_z(  g_\epsilon)$, by the assumption, $\mathcal{L}_z$ is also a linear Boltzmann operator, thus
\begin{align}
\label{est-z-lz-1}
\begin{aligned}
& (c_f +1 + c_e \lambda_3) \sum\limits_{ i+j =s-1}\epsilon^2 \vert \intpsz \nabla^j_x \nabla^i_v  \repst \mathcal{L}_z(g_\epsilon)   \sdot \nabla^j_x \nabla^i_v \partial_z g_\epsilon\bdv \bd x \bd z \vert \\
& \le  (c_f +1 + c_e \lambda_3)\sum\limits_{ i+j =s-1} \vert \intpsz \nabla^j_x \nabla^i_v  \mathcal{L}_z(g_\epsilon)   \sdot \nabla^j_x \nabla^i_v \partial_z g_\epsilon\bdv \bd x \bd z \vert \\
& \le (c_f +1 + c_e \lambda_3) \sum\limits_{ i+j =s-1}  \left( \tfrac{1}{c_d} \|\nabla^j_x \nabla^i_v  \mathcal{L}_z(g_\epsilon)\|_{L^2L^2_z}^2 +    \tfrac{c_d}{4}\|\nabla^j_x \nabla^i_v \partial_z g_\epsilon\|_{L^2L^2_z}^2 \right) \\
&  \le  2(c_f +1 + c_e \lambda_3)^2\tfrac{C_z^2}{c_d} \|   g_\epsilon \|_{H^{s}_{\Lambda,z}}^2 +    \tfrac{c_d}{8}\|  \partial_z g_\epsilon\|_{H^{s-1}_z}^2.
\end{aligned}
\end{align}
While there is no derivative with respect to $v$, it follows that
\begin{align}
\begin{aligned}
&\label{est-z-lz-2} (\lambda_1 + \lambda_2)\sum\limits_{k=0}^{s-1} \repst \vert \intpsz \nabla^k_x \mathcal{L}_z(g_\epsilon) \sdot \nabla_x^k \partial_z g_\epsilon\bdv \bd x  \bd z\vert \\
& = (\lambda_1 + \lambda_2) \sum\limits_{k=0}^{s-1} \repst \vert \intpsz \mathcal{L}_z( (\nabla^k_x  g_\epsilon)^\perp) \sdot \left(\nabla_x^k \partial_z g_\epsilon\right)^\perp \bdv \bd x  \bd z\vert \\
& \le (\lambda_1 + \lambda_2)\sum\limits_{k=0}^{s-1} \left( \tfrac{1}{ c_e\epsilon^2} \|   \mathcal{L}_z( (\nabla^k_x  g_\epsilon)^\perp)\|_{L^2L^2_z}^2 +    \tfrac{c_e}{4 \epsilon^2 }\|\left(\nabla_x^k \partial_z g_\epsilon\right)^\perp\|_{L^2L^2_z}^2 \right) \\
&  \le 4(\lambda_1 + \lambda_2)^2 \tfrac{C_z^2}{ c_e\epsilon^2} \|      ( g_\epsilon)^\perp \|_{H^{s-1}_{\Lambda,x} L^2_z}^2 +    \tfrac{c_e}{16 \epsilon^2 }\|\left(\partial_z g_\epsilon\right)^\perp\|_{H^{s-1}_{x} L^2_z}^2.
\end{aligned}
\end{align}
Similarly, we can obtain
\begin{aligno}
&\label{est-z-lz-3} \lambda_4 c_e\sum\limits_{k=0}^{s-2} \reps \vert \intpsz \nabla^k_x \mathcal{L}_z(g_\epsilon) \sdot \nabla_x^k \partial_z g_\epsilon\bdv \bd x  \bd z\vert \\
& = \lambda_4 c_e\sum\limits_{k=0}^{s-1} \reps \vert \intpsz \mathcal{L}_z( (\nabla^k_x  g_\epsilon)^\perp) \sdot \left(\nabla_x^k \partial_z g_\epsilon\right)^\perp \bdv \bd x  \bd z\vert \\
& \le \sum\limits_{k=0}^{s-1} \left( \tfrac{\lambda_4^2 c_e}{  \epsilon^2} \|   \mathcal{L}_z( (\nabla^k_x  g_\epsilon)^\perp)\|_{L^2L^2_z}^2 +    \tfrac{c_e}{4 \epsilon^2 }\|\left(\nabla_x^k \partial_z g_\epsilon\right)^\perp\|_{L^2L^2_z}^2 \right) \\
&  \le 4\tfrac{C_z^2 \lambda_4^2 c_e}{  \epsilon^2} \|      ( g_\epsilon)^\perp \|_{H^{s-1}_{\Lambda,x} L^2_z}^2 +    \tfrac{c_e}{16 \epsilon^2 }\|\left(\partial_z g_\epsilon\right)^\perp\|_{H^{s-1}_{x} L^2_z}^2.
\end{aligno}
Combining the relevent estimates with $\mathcal{L}_z$, we can obtain that
\begin{aligno}
\label{estzLz}
& (c_f +1 + c_e \lambda_3) \sum\limits_{ i+j =s-1}\epsilon^2 \vert \intpsz \nabla^j_x \nabla^i_v  \repst \mathcal{L}_z(g_\epsilon)   \sdot \nabla^j_x \nabla^i_v \partial_z g_\epsilon\bdv \bd x \bd z \vert \\
& +  (\lambda_1 + \lambda_2)\sum\limits_{k=0}^{s-1} \repst \vert \intpsz \nabla^k_x \mathcal{L}_z(g_\epsilon) \sdot \nabla_x^k \partial_z g_\epsilon\bdv \bd x  \bd z\vert \\
& +  \lambda_4 c_e\sum\limits_{k=0}^{s-2} \reps \vert \intpsz \nabla^k_x \mathcal{L}_z(g_\epsilon) \sdot \nabla_x^k \partial_z g_\epsilon\bdv \bd x  \bd z\vert \\
&  -  \tfrac{c_e}{8 \epsilon^2 }\|\left(\partial_z g_\epsilon\right)^\perp\|_{H^{s-1}_{x} L^2_z}^2 - \tfrac{c_d}{8}\|  \partial_z g_\epsilon\|_{H^{s-1}_z}^2 \\
&  \lesssim \|   g_\epsilon \|_{H^{s}_{\Lambda,z}}^2 + \repst \|      ( g_\epsilon)^\perp \|_{H^{s-1}_{\Lambda,x} L^2_z}^2
\end{aligno}

For $\partial_z N_2(g_\epsilon)$, recalling that
\[ N_2(g_\epsilon)=\reps {\Gamma}(g_\epsilon,g_\epsilon),  \]
thus,
\[ \partial_z \big(\Gamma(g_\epsilon, g_\epsilon)\big)=  {\Gamma}_z(g_\epsilon,g_\epsilon) +   {\Gamma}(\partial_z g_\epsilon,g_\epsilon) +   {\Gamma}_z(g_\epsilon, \partial_z g_\epsilon), \]
noticing that $  {\Gamma}(g_\epsilon,g_\epsilon)$ and $\Gamma_z(g_\epsilon, g_\epsilon)$ belongs to the orthogonal space of $\mathcal{L}$,  thus, we have
\begin{align*}
&  (\lambda_1 + \lambda_2)\sum\limits_{k=0}^{s-1} \vert \intpsz \nabla^k_x \partial_z N_2(g_\epsilon) \sdot \nabla_x^k \partial_z g_\epsilon\bdv \bd x \bd z \vert \\
& =  \sum\limits_{k=0}^{s-1} (\lambda_1 + \lambda_2) \vert \intpsz \nabla^k_x \partial_z \big( \Gamma(g_\epsilon, g_\epsilon) \big)\sdot \tfrac{1}{\epsilon}\big(\nabla_x^k  \partial_z g_\epsilon\big)^\perp\bdv \bd x \bd z\vert \\
& \lesssim \reps \|g_\epsilon\|_{H^s_{x,z}}  \|g_\epsilon\|_{H^s_{\Lambda_x,z}}\|\partial_z g_\epsilon^\perp\|_{H^{s-1}_{\Lambda_x} L^2_z}.
\end{align*}
From the above inequalities, by Young's inequality,  we can obtain that
\begin{align*}
&  (\lambda_1 + \lambda_2)\sum\limits_{k=0}^{s-1} \vert \intpsz \nabla^k_x \partial_z N_2(g_\epsilon) \sdot \nabla_x^k \partial_z g_\epsilon\bdv \bd x \bd z \vert  - \tfrac{c_e}{8\epsilon^2}\|\partial_z g_\epsilon^\perp\|_{H^{s-1}_{\Lambda_x} L^2_z}^2  \\
& \lesssim  \|g_\epsilon\|_{H^s_{x,z}}^2  \|g_\epsilon\|_{H^s_{\Lambda_x,z}}^2.
\end{align*}
For the last two terms in the right hand of \eqref{est-z-basic}, recalling that
\begin{align*}
\|\partial_z \big(\Gamma(g_\epsilon, g_\epsilon)\big)\|_{H^{s-1}L^2_z} \lesssim \|g_\epsilon\|_{H^{s}_{\Lambda,z}}^2,~~\|\partial_z \big(\Gamma(g_\epsilon, g_\epsilon)\big)\|_{H^{s-2}_xL^2_z} \lesssim \|g_\epsilon\|_{H^{s-1}_{\Lambda,x}}^2,
\end{align*}
thus
\begin{align*}
& \epsilon^2 (c_f +1 + c_e \lambda_3)  \sum\limits_{i \ge 1\atop i + j =s-1} \vert\intpsz \big(\nabla^j_x \nabla^i_v \partial_z\big(N_2(g_\epsilon)\big)\big) \sdot \big(\nabla^j_x \nabla^i_v \partial_z g_\epsilon\big)\bdv \bd x \bd z \vert \\
& + \epsilon \lambda_4 \cdot  c_e \sum\limits_{k=0}^{s-2} \vert \intpsz \nabla_x \nabla^k_x \partial_z\big( N_2(g_\epsilon) \big) \sdot \nabla_v \nabla^k_x \partial_z g_\epsilon\bdv \bd x \bd z\vert \\
& \lesssim  \epsilon \sum\limits_{i \ge 1\atop i + j =s-1} \vert\intpsz \big(\nabla^j_x \nabla^i_v \partial_z \big(\Gamma(g_\epsilon, g_\epsilon)\big)\big) \sdot \big(\nabla^j_x \nabla^i_v \partial_z g_\epsilon\big)\bdv \bd x \bd z\vert\\
& \qquad +  \sum\limits_{k=0}^{s-2} \vert \intpsz \nabla_x \nabla^k_x \partial_z \big(\Gamma(g_\epsilon, g_\epsilon)\big) \sdot \nabla_v \nabla^k_x \partial_z g_\epsilon\bdv \bd x \bd z\vert\\
& \lesssim \epsilon \|  g_\epsilon\|_{H^{s}_z} \|  g_\epsilon\|_{H^{s}_{\Lambda,z}}^2 + \|  g_\epsilon\|_{H^{s}_{x,z}} \|  g_\epsilon\|_{H^{s}_{\Lambda,z}}^2.
\end{align*}
In summary, we can conclude that
 \begin{aligno}
 \label{estzN2}
& (\lambda_1 + \lambda_2)\sum\limits_{k=0}^{s-1} \vert \intpsz \nabla^k_x \partial_z N_2(g_\epsilon) \sdot \nabla_x^k \partial_z g_\epsilon\bdv \bd x \bd z \vert \\
&  + \epsilon^2  \sum\limits_{i \ge 1\atop i + j =s-1} \vert\intpsz \big(\nabla^j_x \nabla^i_v \partial_z\big(N_2(g_\epsilon)\big)\big) \sdot \big(\nabla^j_x \nabla^i_v \partial_z g_\epsilon\big)\bdv \bd x \bd z \vert \\
& + \epsilon \sum\limits_{k=0}^{s-2} \vert \intpsz \nabla_x \nabla^k_x \partial_z\big( N_2(g_\epsilon) \big) \sdot \nabla_v \nabla^k_x \partial_z g_\epsilon\bdv \bd x \bd z\vert - \tfrac{c_e}{8\epsilon^2}\|\partial_z g_\epsilon^\perp\|_{H^{s-1}_{\Lambda_x} L^2_z}^2 \\
& \lesssim \epsilon \|  g_\epsilon\|_{H^{s}_z} \|  g_\epsilon\|_{H^{s}_{\Lambda,z}}^2 + \|  g_\epsilon\|_{H^{s}_{x,z}} \|  g_\epsilon\|_{H^{s}_{\Lambda,z}}^2 + \|g_\epsilon\|_{H^s_{x,z}}^2  \|g_\epsilon\|_{H^s_{\Lambda_x,z}}^2.
\end{aligno}
With
\begin{align*}
\mathcal{E}_{\epsilon,z}^s(t):& = \|(g_\epsilon(t),\nabla_x \phi_\epsilon(t))\|_{H^s_{x,z}}^2 + \epsilon^2 \|\nabla_v g_\epsilon(t)\|_{H^{s-1}_z}^2\\
 & \quad +  \|(g_\epsilon(t),\nabla_x \phi_\epsilon(t))\|_{H^s_{x}\lzi}^2 + \epsilon^2 \|\nabla_v g_\epsilon(t,z)\|_{H^{s-1}\lzi}^2,
\end{align*}
finally we are ready to  close the inequalities.  From  \eqref{estzN1}, \eqref{estzLz} and \eqref{estzN2}, \eqref{est-z-h1} turns to
\begin{align}
\label{estzH1-c}
\begin{split}
& \dt \bigg( c_e \cdot   \mathfrak{E}_{\epsilon,1,z,1}^s(t) +  c_f \cdot   \mathfrak{E}_{\epsilon,2,2,z,1}^s(t) +  \mathfrak{E}_{\epsilon,2,1,z,1}^s(t)\bigg) \\
& \qquad +   \tfrac{3c_e}{4\epsilon^2} \|(\partial_z g_\epsilon)^\perp\|_{H^{s-1}_x L^2_z}^2  + c_e\|\partial_z \rho_\epsilon\|_{H^{s-1} L^2_z}^2  + \tfrac{3c_d}{4}\|\partial_z g_\epsilon\|_{H^{s-1}_\Lambda L^2_z}^2  \\
&\lesssim \epsilon \|  g_\epsilon\|_{H^{s}_z} \|  g_\epsilon\|_{H^{s}_{\Lambda,z}}^2 + \|  g_\epsilon\|_{H^{s}_{x,z}} \|  g_\epsilon\|_{H^{s}_{\Lambda,z}}^2 + \|g_\epsilon\|_{H^s_{x,z}}^2  \|g_\epsilon\|_{H^s_{\Lambda_x,z}}^2 + \|   g_\epsilon \|_{H^{s}_{\Lambda,z}}^2 + \repst \|      ( g_\epsilon)^\perp \|_{H^{s-1}_{\Lambda,x} L^2_z}^2\\
& \lesssim \sqrt{\mathfrak{E}_{\epsilon, z}^s(t)} \|g_\epsilon(t,z)\|_{H^{s}_x L^2_z }^2  + \tfrac{1}{\epsilon^2} \sqrt{\mathfrak{E}_{\epsilon, z}^s(t)}\cdot  \epsilon^2 \|\nabla_v g_\epsilon(t,z)\|_{H^{s-1}_{\Lambda,z}}^2 + \|   g_\epsilon \|_{H^{s}_{\Lambda,z}}^2 + \repst \|      ( g_\epsilon)^\perp \|_{H^{s-1}_{\Lambda,x} L^2_z}^2,
\end{split}
\end{align}
where we have used that
\begin{align*}
\intz \|\nabla_x \phi_\epsilon\|_{L^2  }^2\|\nabla_x \partial_z \phi_\epsilon\|_{L^2  }^2 \bd z  \lesssim \|g_\epsilon\|_{H^s_{x,z}}^4.
\end{align*}
With the help of  \eqref{est-z-to-e-or}, the last two terms in the right hand of \eqref{estzH1-c} can be absorbed. Since the skills are the same to \eqref{estsinitial}, we omit the details and complete the proof.

\end{proof}

\subsection{Construction of approximate solutions}
In this section,  we are going to show the existence of solutions based on the prior estimates established in Sec. \ref{sec-z}. Different with  \cite{liujin-2018-kinetic} where the approximate solutions  were obtained by a Galerkin type method,  the iteration methods are employed to obtaining the approximate solution sequence in this work. For each fixed $z \in \mathbb{I}_z$,  the nonlinear system
the non-linear system
\begin{align}
\partial_t   g_\epsilon(z) + \reps \vdot   g_\epsilon(z) +  \repst \mathcal{L}(  g_\epsilon(z))   - \tfrac{v}{\epsilon} \sdot \nabla   \phi_\epsilon(z) = N(g_\epsilon(z)),
\end{align}
where
\begin{align*}
N(g_\epsilon)&:=  N_1(g_\epsilon) + \partial_z N_2(g_\epsilon),\\
N_1(g_\epsilon)&:= ({v}  g_\epsilon - \nabla_v g_\epsilon     ) \sdot \nabla\phi_\epsilon,\\
N_2(g_\epsilon)&:= \reps {\Gamma}(g_\epsilon,g_\epsilon).
\end{align*}
From Sec. \ref{sec-app}, the approximate solutions can be constructed and then the wellposedness of solutions can be verified in $H^s$ space. Furthermore,  from Lemma \ref{lemma-nonlinear-decay-ran-infty} and Lemma \ref{lemma-nonlinear-decay-ran}, under our settings, the solutions are Lipschitz in $z$ and  we have proved the existence of solutions.

\subsection{Stability}
In this part, we consider the stability problem. Suppose that $g_{\epsilon,1}(t)$  and $g_{\epsilon,2}(t)$ are the solutions (in $H^s_z$ space) to
\begin{align}
\partial_t   g_\epsilon  + \reps \vdot   g_\epsilon  +  \repst \mathcal{L}(  g_\epsilon )   - \tfrac{v}{\epsilon} \sdot \nabla   \phi_\epsilon  = N(g_\epsilon ),
\end{align}
with initial data $g_{\epsilon,1}(0)$ and $g_{\epsilon,2}(0)$ respectively.  Letting
\[ h_\epsilon = g_{\epsilon, 1} - g_{\epsilon,2}, ~~~~\Delta_x \phi^h_\epsilon =\rho^h_\epsilon := \intv h_\epsilon \bdv  \]
then
\begin{aligno}
\label{vpb-stable}
\partial_t   h_\epsilon  + \reps \vdot   h_\epsilon  +  \repst \mathcal{L}(  h_\epsilon )   - \tfrac{v}{\epsilon} \sdot \nabla   \phi_\epsilon^h  =:N^d(h_\epsilon) = N^d_1(h_{\epsilon}) + N^d_2(h_{\epsilon}),
\end{aligno}
with
\begin{align*}
N^d_1(h_{\epsilon}) = N_1(g_{\epsilon,1}) - N_1(g_{\epsilon,2}),\\
N^d_2(h_{\epsilon}) = N_2(g_{\epsilon,1}) - N_2(g_{\epsilon,2}).
\end{align*}
\begin{lemma}
\label{lemma-stable}
Suppose that $g_{\epsilon,1}$ and $g_{\epsilon,2}$ are two solutions to VPB system with initial data $g_{\epsilon,1}(0)$ and $g_{\epsilon,2}(0)$ respectively in $H^s_z$ space for each $\epsilon>0$. For some small enough consant $s_0$,  as long as the initial data satisfy
\begin{align*}
 \|(g_{\epsilon,1}, \nabla_x \phi_{\epsilon,1})\|_{H^s_{x,z}}^2 + \epsilon^2 \|\nabla_v g_{\epsilon,1}\|_{H^{s-1}_z}^2 \le  s_0,~~~~\|(g_{\epsilon,2}, \nabla_x \phi_{\epsilon,2})\|_{H^s_{x,z}}^2 + \epsilon^2 \|\nabla_v g_{\epsilon,2}\|_{H^{s-1}_z}^2 \le  s_0,
\end{align*}
there exist some constant $\bar{s}_0$ and $\bar{s}$  such that
\begin{aligno}
&\|(g_{\epsilon,1} - g_{\epsilon,2},\nabla_x (\phi_{\epsilon,1} -\phi_{\epsilon,1}) )(t)\|_{H^\ell_{x}L^{2\cap \infty}_z }^2 + \epsilon^2 \|\nabla_v g_\epsilon(t)\|_{H^{\ell-1}L^{2\cap \infty}_z}^2 \\
& \le \bar{s}_0 \exp(-\bar{s} t) \left( \|(g_{\epsilon,1} - g_{\epsilon,2},\nabla_x (\phi_{\epsilon,1} -\phi_{\epsilon,1}) )(0)\|_{H^\ell_{x}L^{2\cap \infty}_z }^2 + \epsilon^2 \|\nabla_v g_\epsilon(0)\|_{H^{\ell-1}L^{2\cap \infty}_z}^2 \right), ~~ \ell \le s -1.
\end{aligno}

\end{lemma}

\begin{proof}
Noticing that the left part of  \eqref{vpb-stable} enjoys the same structure to that  of \eqref{vpb-ns-nonlinear}. To get the same version to \eqref{est-nonlinear-all-hs}, we shall analyze the mean value of the fluid part of $h_\epsilon(t)$ to obtain the similar result as \eqref{est-mean-difference}. By simple calculation, the global conservation laws  of \eqref{lemma-stable}  are
\begin{align}
\label{conservation-law-nonlinear-stable}
\begin{split}
\tdt { \intps h_\epsilon(t,z) \m \bd v \bd x = 0,   \tdt \intps v h_\epsilon(t,z) \m \bd v \bd x =0,}\\
\tdt \bigg(  \intps (\tfrac{|v|^2}{2} - \tfrac{3}{2}) h_\epsilon(t,z)\m \bd v \bd x +  \epsilon \big( \|\nabla_x \phi_{\epsilon,1}(t,z)\|_{L^2}^2 - \|\nabla_x \phi_{\epsilon,2}(t,z)\|_{L^2}^2 \big) \bigg)= 0.
\end{split}
\end{align}
Thus, we can deduce
\begin{align}
\label{est-mean-difference-stable}
\begin{split}
\intt \bdp h_\epsilon(t,z) \bd x  =- C_\Omega^{-1} \sdot \tfrac{\epsilon}{3} \left(  \|\nabla_x \phi_{\epsilon,1}(t)\|_{L^2}^2  - \|\nabla_x \phi_{\epsilon,2}(t)\|_{L^2}^2 \right)  \sdot (\tfrac{|v|^2}{3} - 1).
\end{split}
\end{align}
With the help of \eqref{est-mean-difference-stable} and the similar trick of obtaining \eqref{est-nonlinear-all-hs}, we can obtain that
\begin{align}
\label{est-stable-hs}
\begin{split}
& \dt \bigg( c_e \cdot   \mathfrak{E}_{\epsilon,1}^\ell(t,z) +  c_f \cdot   \mathfrak{E}_{\epsilon,2,2}^\ell(t,z) +  \mathfrak{E}_{\epsilon,2,1}^\ell(t,z)\bigg) \\
& \qquad +   \tfrac{c_e}{\epsilon^2} \|(h_\epsilon(t,z))^\perp\|_{H^\ell_\Lambda}^2  + c_e\|\rho_\epsilon^h(t,z)\|_{H^\ell}^2  + c_d\|h_\epsilon(t,z)\|_{H^\ell_{\Lambda}}^2  \\
&\le  \tfrac{\epsilon^2 \lambda_3 {(a_4 + C_\delta)} }{9C_\Omega} \left( \|\nabla_x \phi_{\epsilon,1}\|_{L^2}^2 - \|\nabla_x \phi_{\epsilon,2}\|_{L^2}^2 \right)^2  \\
& + c_e(\lambda_1 + \lambda_2) \sum\limits_{k=0}^{\ell} \vert \intps \nabla^k_x N^d_\epsilon \sdot \nabla_x^k h_\epsilon\bdv \bd x \vert \\
& + (c_f +1 + c_e \lambda_3) \epsilon^2  \sum\limits_{i \ge 1\atop i + j =\ell} \vert\intps \big(\nabla^j_x \nabla^i_v N^d_\epsilon\big) \sdot \big(\nabla^j_x \nabla^i_v h_\epsilon\big)\bdv \bd x \vert\\
& + \lambda_4 \cdot  c_e \epsilon \sum\limits_{k=0}^{\ell-1} \vert \intps \nabla_x \nabla^k_x N^d_\epsilon \sdot \nabla_v \nabla^k_x h_\epsilon\bdv \bd x\vert, \quad\quad \ell \le s-1.
\end{split}
\end{align}
The idea is to employ the left hand of \eqref{est-stable-hs} to absorbing its right hand with the help of Lemma \ref{lemma-nonlinear-decay-ran-infty}.

For the source term $N(g_{\epsilon,1})$, first for $N_1$, we can obtain that
\begin{aligno}
\label{estStableN1}
N_1^d(h_\epsilon) & =  N_1(g_{\epsilon, 1}) - N_1(g_{\epsilon, 2})  \\
& = ({v}  g_{\epsilon, 1} - \nabla_v g_{\epsilon, 1}     ) \sdot \nabla\phi_{\epsilon, 1} -  ({v}  g_{\epsilon, 2} - \nabla_v g_{\epsilon, 2}     ) \sdot \nabla\phi_{\epsilon, 2} \\
& = ({v}  h_{\epsilon} - \nabla_v h_{\epsilon}     ) \sdot \nabla\phi_{\epsilon, 1} +  ({v}  g_{\epsilon, 2} - \nabla_v g_{\epsilon, 2}     ) \sdot \nabla\phi_{\epsilon}^h.
\end{aligno}
Noticing that there exists $({v}  g_{\epsilon, 2} - \nabla_v g_{\epsilon, 2}     ) \sdot \nabla\phi_{\epsilon}^h$ in \eqref{estStableN1}, this is why we need $\ell \le s -1$. Indeed, if $\ell =s$,
$  \vert \intps \nabla^s_v ({v}  g_{\epsilon, 2} - \nabla_v g_{\epsilon, 2}     ) \sdot \nabla\phi_{\epsilon}^h  \sdot \nabla_v^s h_\epsilon\bdv \bd x \vert$ can not be bounded.  Furthermore, we can obtain for the differences between  $N_1(g_{\epsilon,1})$ and  $N_1(g_{\epsilon,2})$
\begin{aligno}
\label{estDN11}
& \sum\limits_{ i+j =\ell}\vert \intps \nabla^j_x \nabla^i_v N_1^d(h_\epsilon) \sdot \nabla^j_x \nabla^i_v   h_\epsilon\bdv \bd x   \vert \\
& \le   \sum\limits_{ i+j =\ell} \vert \intps \nabla^j_x \nabla^i_v  \big( ({v}  h_{\epsilon} - \nabla_v h_{\epsilon}     ) \sdot \nabla\phi_{\epsilon, 1} \big) \sdot \nabla^j_x \nabla^i_v   h_\epsilon\bdv \bd x  \vert \\
& +  \sum\limits_{ i+j =\ell} \vert \intps \nabla^j_x \nabla^i_v  \big( ({v}  g_{\epsilon, 2} - \nabla_v g_{\epsilon, 2}     ) \sdot \nabla\phi_{\epsilon}^h \big) \sdot \nabla^j_x \nabla^i_v   h_\epsilon\bdv \bd x  \vert.
\end{aligno}
Since $\ell \le s -1$ and $j +1 \le s \le s+2$, we can infer that
\[ \|\nabla^n \phi_{\epsilon,1}\|_{L^\infty} \lesssim \|g_{\epsilon,1}\|_{H^s}, n \le \ell. \]
With the help of the above inequality, we can infer \begin{aligno}\label{estDN12}
& \sum\limits_{ i+j =\ell} \vert \intps \nabla^j_x \nabla^i_v  \big( ({v}  h_{\epsilon} - \nabla_v h_{\epsilon}     ) \sdot \nabla\phi_{\epsilon, 1} \big) \sdot \nabla^j_x \nabla^i_v   h_\epsilon\bdv \bd x  \vert \\
& \le \sum\limits_{ i+j =\ell} \vert \intps \nabla^j_x \bigg( \nabla^i_v   ({v}  h_{\epsilon} - \nabla_v h_{\epsilon}     ) \sdot \nabla\phi_{\epsilon, 1} \bigg) \sdot \nabla^j_x \nabla^i_v   h_\epsilon\bdv \bd x  \vert \\
& \lesssim \|g_{\epsilon,1}\|_{H^s}\|h_\epsilon\|_{H^\ell_\Lambda}^2,
\end{aligno}
and
\begin{aligno}\label{estDN13}
& \sum\limits_{ i+j =\ell} \vert \intps \nabla^j_x \nabla^i_v  \big( ({v}  g_{\epsilon, 2} - \nabla_v g_{\epsilon, 2}     ) \sdot \nabla\phi_{\epsilon}^h \big) \sdot \nabla^j_x \nabla^i_v   h_\epsilon\bdv \bd x  \vert\\
& \le \|g_{\epsilon,2}\|_{H^s_\Lambda}\|h_\epsilon\|_{H^s}\|h_\epsilon\|_{H^\ell_\Lambda}.
\end{aligno}
Combing \eqref{estDN11}, \eqref{estDN12} and \eqref{estDN13},  we can obtain that for $\epsilon \le 1$
\begin{aligno}
\label{estDN1}
&  c_e(\lambda_1 + \lambda_2) \sum\limits_{k=0}^{\ell} \vert \intps \nabla^k_x N^d_1(h_\epsilon) \sdot \nabla_x^k h_\epsilon\bdv \bd x \vert - \tfrac{c_d}{8} \|h_\epsilon\|_{H^\ell_\Lambda}^2 \\
& + (c_f +1 + c_e \lambda_3) \epsilon^2  \sum\limits_{i \ge 1\atop i + j =\ell} \vert\intps \big(\nabla^j_x \nabla^i_v N^d_1(h_\epsilon)\big) \sdot \big(\nabla^j_x \nabla^i_v h_\epsilon\big)\bdv \bd x \vert\\
& + \lambda_4 \cdot  c_e \epsilon \sum\limits_{k=0}^{\ell-1} \vert \intps \nabla_x \nabla^k_x N^d_1(h_\epsilon) \sdot \nabla_v \nabla^k_x h_\epsilon\bdv \bd x\vert    \\
& \lesssim \|g_{\epsilon,1}\|_{H^s_x}\|h_\epsilon\|_{H^\ell_\Lambda}^2 + \|g_{\epsilon,2}\|_{H^s_\Lambda}^2\|h_\epsilon\|_{H^\ell}^2.
\end{aligno}
The $N_2^d(h_\epsilon)$ is more complicated. First, we can decompose  $N^d_2(h_{\epsilon})$ as follows
\begin{align*}
N^d_2(h_{\epsilon}) & = \reps \Gamma(g_{\epsilon,1}, g_{\epsilon,1})  - \reps \Gamma(g_{\epsilon,2}, g_{\epsilon,2})\\
& = \reps \Gamma(h_{\epsilon}, g_{\epsilon,1})  + \reps \Gamma(g_{\epsilon,2}, h_{\epsilon}).
\end{align*}
Again, noticing that $  \Gamma(h_{\epsilon}, g_{\epsilon,1})$ and $\Gamma(g_{\epsilon,2}, h_{\epsilon})$ belongs to the orthogonal space of $\mathcal{L}$,  thus it follows that
\begin{align*}
&  (\lambda_1 + \lambda_2)\sum\limits_{k=0}^{\ell} \vert \intps \nabla^k_x   N_2^d(h_\epsilon) \sdot \nabla_x^k   h_\epsilon\bdv \bd x   \vert \\
& =  \sum\limits_{k=0}^{\ell} (\lambda_1 + \lambda_2) \vert \intps \nabla^k_x N_2^d(h_\epsilon)\sdot \tfrac{1}{\epsilon}\big(\nabla_x^k    h_\epsilon\big)^\perp\bdv \bd x  \vert \\
& \le \sum\limits_{k=0}^{\ell} (\lambda_1 + \lambda_2) \vert \intps \nabla^k_x \Gamma(h_{\epsilon}, g_{\epsilon,1})\sdot \tfrac{1}{\epsilon}\big(\nabla_x^k    h_\epsilon\big)^\perp\bdv \bd x  \vert \\
& + \sum\limits_{k=0}^{\ell} (\lambda_1 + \lambda_2) \vert \intps \nabla^k_x \Gamma(g_{\epsilon,2}, h_{\epsilon})\sdot \tfrac{1}{\epsilon}\big(\nabla_x^k    h_\epsilon\big)^\perp\bdv \bd x  \vert \\
& \lesssim \reps \|g_{\epsilon,1}\|_{H^\ell_{x}}  \|h_\epsilon\|_{H^\ell_{\Lambda_x}}\|  h_\epsilon^\perp\|_{H^{\ell}_{\Lambda_x}  }.
\end{align*}
From the above inequalities, by Young's inequality,  we can obtain that
\begin{align*}
&  (\lambda_1 + \lambda_2)\sum\limits_{k=0}^{\ell} \vert \intps \nabla^k_x   N_2^d(h_\epsilon) \sdot \nabla_x^k   h_\epsilon\bdv \bd x   \vert  - \tfrac{c_e}{8\epsilon^2}\|  h_\epsilon^\perp\|_{H^{\ell}_{\Lambda_x}  }^2  \\
& \lesssim  \left(  \|g_{\epsilon,1}\|_{H^\ell_{x}}^2  + \|g_{\epsilon,2}\|_{H^\ell_{x}}^2 \right) \|h_\epsilon\|_{H^\ell_{\Lambda_x}}^2.
\end{align*}
In the similar way, we can infer that
\begin{align*}
& \epsilon^2 (c_f +1 + c_e \lambda_3)  \sum\limits_{i \ge 1\atop i + j =\ell} \vert\intps \big(\nabla^j_x \nabla^i_v  \big(N_2^d(h_\epsilon)\big)\big) \sdot \big(\nabla^j_x \nabla^i_v h_\epsilon\big)\bdv \bd x   \vert \\
& + \epsilon \lambda_4 \cdot  c_e \sum\limits_{k=0}^{\ell - 1} \vert \intps \nabla_x \nabla^k_x  \big( N_2^d(h_\epsilon) \big) \sdot \nabla_v \nabla^k_x   h_\epsilon\bdv \bd x  \vert \\
& \lesssim  \epsilon \sum\limits_{i \ge 1\atop i + j =\ell} \vert\intps \big(\nabla^j_x \nabla^i_v   \big(\Gamma(h_\epsilon, g_{\epsilon,1})\big)\big) \sdot \big(\nabla^j_x \nabla^i_v   h_\epsilon\big)\bdv \bd x  \vert\\
& + \epsilon \sum\limits_{i \ge 1\atop i + j =\ell} \vert\intps \big(\nabla^j_x \nabla^i_v   \big(\Gamma(g_{\epsilon,2}, h_\epsilon)\big)\big) \sdot \big(\nabla^j_x \nabla^i_v   h_\epsilon\big)\bdv \bd x  \vert\\
&  +  \sum\limits_{k=0}^{\ell -1} \vert \intps \nabla_x \nabla^k_x  \big(\Gamma(h_\epsilon, g_{\epsilon,1})\big) \sdot \nabla_v \nabla^k_x   h_\epsilon\bdv \bd x  \vert\\
& + \sum\limits_{k=0}^{\ell -1} \vert \intps \nabla_x \nabla^k_x \partial_z \big(\Gamma(g_{\epsilon,2}, h_\epsilon)\big) \sdot \nabla_v \nabla^k_x   h_\epsilon\bdv \bd x  \vert\\
& \lesssim \epsilon \|  g_{\epsilon,1}\|_{H^{\ell}} \|  h_\epsilon\|_{H^{s}_{\Lambda,z}}^2 + \epsilon \|  g_{\epsilon,2}\|_{H^{\ell}} \|  h_\epsilon\|_{H^{\ell}_{\Lambda}}^2  \\
& \qquad + \|  g_{\epsilon,1}\|_{H^{\ell}_{x}} \|  h_\epsilon\|_{H^{s}_{\Lambda}}^2 + \|  g_{\epsilon,2}\|_{H^{\ell}_{x}} \|  h_\epsilon\|_{H^{s}_{\Lambda}}^2.
\end{align*}
In summary, we can conclude that
 \begin{aligno}
 \label{estszN2}
& (\lambda_1 + \lambda_2)\sum\limits_{k=0}^{\ell} \vert \intps \nabla^k_x  N_2^d(h_\epsilon) \sdot \nabla_x^k   g_\epsilon\bdv \bd x \bd z \vert \\
&  + \epsilon^2 (c_f +1 + c_e \lambda_3) \sum\limits_{i \ge 1\atop i + j =\ell-1} \vert\intps \big(\nabla^j_x \nabla^i_v N_2^d(h_\epsilon)\big) \sdot \big(\nabla^j_x \nabla^i_v  h_\epsilon\big)\bdv \bd x \  \vert \\
& + \epsilon \sum\limits_{k=0}^{\ell-2} \vert \intpsz \nabla_x \nabla^k_x  \big( N_2^d(h_\epsilon) \big) \sdot \nabla_v \nabla^k_x h_\epsilon\bdv \bd x \bd z\vert - \tfrac{c_e}{8\epsilon^2}\|h_\epsilon^\perp\|_{H^{\ell}_{\Lambda_x}  }^2 \\
& \lesssim \epsilon \|  (g_{\epsilon,1}, g_{\epsilon,2})\|_{H^{\ell}} \|  h_\epsilon\|_{H^{s}_{\Lambda,z}}^2 + \|  (g_{\epsilon,1}, g_{\epsilon,2})\|_{H^{\ell}_{x}} \|  h_\epsilon\|_{H^{s}_{\Lambda}}^2 + \|(g_{\epsilon,1}, g_{\epsilon,2})\|_{H^s_{x}}^2  \|h_\epsilon\|_{H^s_{\Lambda_x}}^2.
\end{aligno}
Combing \eqref{est-stable-hs}, \eqref{estDN1} and \eqref{estszN2}, we can infer that
\begin{align}
\label{estshs}
\begin{split}
& \dt \bigg( c_e \cdot   \mathfrak{E}_{\epsilon,1}^\ell(t,z) +  c_f \cdot   \mathfrak{E}_{\epsilon,2,2}^\ell(t,z) +  \mathfrak{E}_{\epsilon,2,1}^\ell(t,z)\bigg) \\
& \qquad +   \tfrac{7c_e}{8\epsilon^2} \|(h_\epsilon(t,z))^\perp\|_{H^\ell_{\Lambda_x}}^2  + c_e\|\rho_\epsilon^h(t,z)\|_{H^\ell}^2  + c_d\|h_\epsilon(t,z)\|_{H^\ell_{\Lambda}}^2  \\
& \lesssim  \|g_{\epsilon,2}(t,z)\|_{H^s_\Lambda}^2\|h_\epsilon(t,z)\|_{H^\ell_x}^2 + \epsilon \|  (g_{\epsilon,1}, g_{\epsilon,2})(t,z)\|_{H^{s}} \|  h_\epsilon(t,z)\|_{H^{\ell}_{\Lambda}}^2 \\
& + \|  (g_{\epsilon,1}, g_{\epsilon,2})(t,z)\|_{H^{s}_{x}} \|  h_\epsilon(t,z)\|_{H^{\ell}_{\Lambda}}^2 + \|(g_{\epsilon,1}, g_{\epsilon,2})(t,z)\|_{H^s_{x}}^2  \|h_\epsilon(t,z)\|_{H^\ell_{\Lambda}}^2 \\
& \lesssim  \|g_{\epsilon,2}(t,z)\|_{H^s_\Lambda}^2\|h_\epsilon(t,z)\|_{H^\ell_x}^2 +  \sqrt{\mathfrak{E}_{\epsilon,1+2}^s(t,z)} \|h_\epsilon(t,z)\|_{H^\ell_{\Lambda}}^2, \ell \le s-1,
\end{split}
\end{align}
with
\begin{align*}
\mathfrak{E}_{\epsilon,1+2}^s(t,z) = \mathfrak{E}_{\epsilon,1}^s(t,z) + \mathfrak{E}_{\epsilon,2}^s(t,z).
\end{align*}
Here $ \mathfrak{E}_{\epsilon,1}^s(t,z)$ and $ \mathfrak{E}_{\epsilon,2}^s(t,z)$ are the $H^s$ norms of $g_{\epsilon,1}$ and $g_{\epsilon,2}$ respectively.

Based on \eqref{estshs} which shares the similar structure as \eqref{est-v-uni},   the term $\|g_{\epsilon,2}(t,z)\|_{H^s_\Lambda}^2\|h_\epsilon(t,z)\|_{H^\ell}^2$ brings new difficulty while following the  trick of obtaining estimates like \eqref{est-to-e}. Indeed,  with the notations,  we can only obtain
\[  \|g_{\epsilon,2}(t,z)\|_{H^s}\lesssim \|g_{\epsilon,2}(t,z)\|_{H^s_\Lambda}.  \]
The idea is to employ the bound of type \eqref{est-v-int-t} to absorb $\|g_{\epsilon,2}(t,z)\|_{H^s_\Lambda}^2\|h_\epsilon(t,z)\|_{H^\ell}^2$.

From \eqref{estshs}, there exists some constant $C_\ell$ such that
\begin{align}
\label{estshs-simple}
\begin{split}
& \dt \bigg( c_e \cdot   \mathfrak{E}_{\epsilon,1}^\ell(t,z) +  c_f \cdot   \mathfrak{E}_{\epsilon,2,2}^\ell(t,z) +  \mathfrak{E}_{\epsilon,2,1}^\ell(t,z)\bigg)  + c_d\|h_\epsilon(t,z)\|_{H^\ell_{\Lambda}}^2  \\
& \le   C_\ell  \|g_{\epsilon,2}(t,z)\|_{H^s_\Lambda}^2\|h_\epsilon(t,z)\|_{H^\ell_x}^2 +    C_\ell \sqrt{\mathfrak{E}_{\epsilon,1+2}^s(t,z)} \|h_\epsilon(t,z)\|_{H^\ell_{\Lambda}}^2.
\end{split}
\end{align}
Reselecting the constant $d_0$ in \eqref{estsinitial} as
\begin{align}
\label{estsinitial-ss}
  e_0\le \frac{c_l c_d^2}{16 c_u \left( C_{sr} + C_\ell + 1\right)^2 },
\end{align}
it follows that
\begin{aligno}
\label{est12up}
\mathfrak{E}_{\epsilon,1+2}^s(t,z) \le \frac{  c_d^2}{4   \left( C_{sr} + C_\ell + 1\right)^2 }.
\end{aligno}
Based on \eqref{estshs-simple} and  \eqref{est12up},
\begin{align}
\label{eststable0}
\begin{split}
 & \frac{\bd }{\bd t} \bigg( c_e \cdot   \mathfrak{E}_{\epsilon,1}^\ell(t,z) +  c_f \cdot   \mathfrak{E}_{\epsilon,2,2}^\ell(t,z) +  \mathfrak{E}_{\epsilon,2,1}^\ell(t,z)\bigg) \\
 &  +  \tfrac{c_d}{c_u} \bigg( c_e \cdot   \mathfrak{E}_{\epsilon,1}^\ell(t,z) +  c_f \cdot   \mathfrak{E}_{\epsilon,2,2}^\ell(t,z) +  \mathfrak{E}_{\epsilon,2,1}^\ell(t,z)\bigg)   \\
& \lesssim     \|g_{\epsilon,2}(t,z)\|_{H^s_\Lambda}^2 \bigg( c_e \cdot   \mathfrak{E}_{\epsilon,1}^\ell(t,z) +  c_f \cdot   \mathfrak{E}_{\epsilon,2,2}^\ell(t,z) +  \mathfrak{E}_{\epsilon,2,1}^\ell(t,z)\bigg).
\end{split}
\end{align}
To complete the proof, we only need to show $\int_0^\infty  \|g_{\epsilon,2}(t,z)\|_{H^s_\Lambda}^2 \bd t $ is finite. From  \eqref{est-z-to-e-or}, by simple calculation, we can obtain that there exists some $C$ such that
\begin{align*}
\int_0^\infty  \|g_{\epsilon,2}(t,z)\|_{H^s_\Lambda}^2 \bd t \le C.
\end{align*}
Then by Gr\"onwall's inequality, we complete the proof.

\end{proof}

\section{Remarks on the fluids limits}
\label{sec-limits}

This section is devoted to verifying the Fluid limits. The formal analysis was clearly performed in \cite{diogosrm-2019-vmb-fluid}.  Based on the uniform estiamtes fluctuations (i.e. $g_\epsilon$), the diffusive limits of the Boltzmann equation was verified in \cite{ns-limit-2018}. Since the process of verifying fluid limits are similar (except $z$),  we sketch the proof. As the fluctuation $g_\epsilon$ is dependent of $z$, we shall explain how to improve the regularity of the limit of $g_\epsilon$ on $z$.

 First, we collect the estimates to be used later,  from \eqref{est-z-to-e-or} and \eqref{est-lemma-nonlinear-exp-decay-ran-z}, we can obtain that
\begin{align}
\label{estLimits}
\int_0^\infty \|g_\epsilon^\perp(z)\|_{H^s_{x.z}}^2 \bd s \lesssim \epsilon^2,~~~\|g_\epsilon\|_{H^s_{x,z}}^2  \le C_0.
\end{align}
From \eqref{estLimits}, we can obtain that
\[  g^\perp_\epsilon \to 0, ~~~~\text{in},~~ L^2\left((0,+\infty); H^s_{x,z}\right),  \]
and there exists some $g \in H^s_{x,z}$ such that
\begin{align*}
g_\epsilon(z) \to g(z),~~~~\text{in},~~~~ H^{s-1}_xL^2_z.
\end{align*}
In summary, we can obtain that
\[  g \in \mathrm{Ker} \mathcal{L}. \]
Now, we try to deduce the Navier-Stokes-Poisson-Fourier system. Letting
\[ \rho_\epsilon(t,x,z) = \intv g_\epsilon(t,z)\bdv,~~ u_\epsilon(t,x,z) = \intv g_\epsilon v\bdv,~~ \theta_\epsilon(t,x,z) = \intv g_\epsilon \left( \tfrac{|v|^2}{3} -1\right) \bdv,    \]
from the estimates \eqref{estLimits}, we can obtain that
\begin{align}
\label{estLimit3}
\rho_\epsilon(t,x,z),~~u_\epsilon(t,x,z),~~\theta_\epsilon(t,x,z) \in L^\infty([0,+\infty); H^s_{x,z}).
\end{align}
Furthermore, for any fixed $t>0$,  by the Sobolev embedding inequaity, $\rho_\epsilon,~~u_\epsilon$ and $\theta_\epsilon$ are H\"older continuous on $\mathbb{T}^3 \times \mathbb{I}_z$. In what follows, we first verify the fluid limit in the distributional sense, then by Arzelà–Ascoli  thoerem to improve the regularity.

We copy \eqref{vpb-ns-scaling-ns-corrector} below,
\begin{align}
\label{vpb-limit}
\partial_t g_\epsilon + \reps \vdot g_\epsilon +  \repst \mathcal{L}(g_\epsilon) +  (\nabla_v g_\epsilon   -  {v}  g_\epsilon) \sdot \nabla\phi_\epsilon - \tfrac{v}{\epsilon} \sdot \nabla \phi_\epsilon = \reps {\Gamma}(g_\epsilon,g_\epsilon).
\end{align}
Multiplying the above equation by $\dm,~v\dm~ \text{and}~\left(\tfrac{|v|^2}{3} -1\right)\dm$ respectively, then the evolution equation of $\rho_\epsilon,~u_\epsilon,~\theta_\epsilon$ can be obtained:
\begin{aligno}
\label{nsp-app-0}
\partial_t \rho_\epsilon + \tfrac{1}{\epsilon} \divg u_\epsilon &=0,\\
\partial_t u_\epsilon + \reps \divg \intv \hat{A} \mathcal{L}g_\epsilon \bdv + \reps \nabla_x (\rho_\epsilon + \theta_\epsilon  - \phi_\epsilon) & = \rho_\epsilon \nabla_x \phi_\epsilon,\\
\partial_t \theta_\epsilon + \tfrac{2}{3\epsilon} \divg \intv \hat{B} \mathcal{L}g_\epsilon \bdv + \tfrac{2}{3\epsilon} \divg u_\epsilon & = \tfrac{2}{3} u_\epsilon \cdot \nabla \phi_\epsilon.
\end{aligno}
where
\begin{align*}
A(v) = v \otimes v - \tfrac{|v|^2}{3}\mathbb{I},~~B(v) = v ( \tfrac{|v|^2}{2} - \tfrac{5}{2}),\\
\mathcal{L}\hat{A}(v) = A(v),~~\mathcal{L}\hat{B}(v) = B(v).
\end{align*}

Based on the estimates \eqref{estLimits} and the first equation of \eqref{nsp-app-0},
\begin{align*}
\divg u_\epsilon \weakc \divg u,~\text{in}~H^{s-1}_{x,z},~ \text{and},~ \divg u =0.
\end{align*}

From the first and third equation of \eqref{nsp-app-0}, we can deduce that
\begin{aligno}
\label{nsp-app}
\partial_t u_\epsilon + \reps \divg \intv \hat{A} \mathcal{L}g_\epsilon \bdv + \reps \nabla_x (\rho_\epsilon + \theta_\epsilon  - \phi_\epsilon) & = \rho_\epsilon \nabla_x \phi_\epsilon,\\
\partial_t \left(\tfrac{3}{5}\theta_\epsilon - \tfrac{2}{5}\rho_\epsilon \right) + \tfrac{2}{5\epsilon} \divg \intv \hat{B} \mathcal{L}g_\epsilon \bdv   & = \tfrac{2}{5} u_\epsilon \sdot \nabla \phi_\epsilon.
\end{aligno}
Based on \eqref{estLimits}, it follows that  $\reps \divg \intv \hat{A} \mathcal{L}g_\epsilon \bdv$ has uniform upper bound with resepct to $\epsilon$  in $H^{s-1}_x$ space. Thus, we can obtain that in the distribution sense that
\[ \nabla_x (\rho_\epsilon + \theta_\epsilon  - \phi_\epsilon) \to 0,  \]
and
\[ \nabla_x(\rho + \theta) = \nabla \phi. \]
Furthermore, letting $\mathbf{P}$ be the Leray projection operator on torus and applying $\mathbf{P}$ to the first equation of \eqref{nsp-app}, it follows that
\begin{aligno}
\label{nsp-app-4}
\partial_t \mathbf{P} u_\epsilon + \reps \mathbf{P}\left( \divg \intv \hat{A} \mathcal{L}g_\epsilon \bdv \right)   & = \mathbf{P} \left(  \rho_\epsilon \nabla_x \phi_\epsilon \right),\\
\partial_t \left(\tfrac{3}{5}\theta_\epsilon - \tfrac{2}{5}\rho_\epsilon \right) + \tfrac{2}{5\epsilon} \divg \intv \hat{B} \mathcal{L}g_\epsilon \bdv   & = \tfrac{2}{5} u_\epsilon \sdot \nabla \phi_\epsilon.
\end{aligno}
While verifying diffusive limit of the Boltzmann equation (see \cite{ns-limit-2018} for instance, specially Sec. 4.2), the similar system to \eqref{nsp-app-4} was established (without the right hand of \eqref{nsp-app-4}). From \eqref{vpb-limit},
\begin{align*}
     \reps \mathcal{L}(g_\epsilon)     & =  {\Gamma}(g_\epsilon,g_\epsilon) + \vdot(\phi_\epsilon - g_\epsilon) - \epsilon \left( (\nabla_v g_\epsilon   -  {v}  g_\epsilon) \sdot \nabla\phi_\epsilon + \partial_t g_\epsilon \right) \\
     & =   {\Gamma}(g_\epsilon,g_\epsilon) + \vdot(\phi_\epsilon - g_\epsilon) + O(\epsilon).
\end{align*}
Since there is a coefficient $\epsilon$ in $O(\epsilon)$,  in the distributional sense,
\[ O(\epsilon) \to 0. \]
By simple calculation (see \cite{diogosrm-2019-vmb-fluid,bgl1993convergence,bgl1991formal}),
\begin{align}
\label{aaa}
  \intv \hat{A} \sdot \reps \mathcal{L}g_\epsilon \bdv = u_\epsilon \otimes u_\epsilon - \tfrac{|u_\epsilon|^2}{3} \mathbf{I} - \mu \left(\nabla_x u_\epsilon + \nabla^T_\epsilon u_\epsilon -\tfrac{2}{3} \divg u_\epsilon \mathbf{I}\right)- R_1(\epsilon)
\end{align}
with
\[R_1(\epsilon) :=  \intv A \sdot \left( O(\epsilon) -  \vdot g^\perp_\epsilon +  {\Gamma}(g_\epsilon^\perp,g_\epsilon) + {\Gamma}(g_\epsilon,g_\epsilon^\perp) \right) \bdv.  \]
and
\[  \mu = \tfrac{1}{15} \sum\limits_{ 1 \le i \le 3 \atop 1 \le j \le 3}\intv A_{ij}\hat{A}_{ij}\bdv. \]

Based on the estimate \eqref{estLimits} on $g^\perp_\epsilon$, one can obtain that in the distributional sense
\[  R_1(\epsilon) \to 0.  \]

 Thus, in the light of  estimates \eqref{estLimits}, \eqref{estLimit3} and \eqref{aaa}, one can finally obtain in distributional sense:
\begin{aligno}
& \mathbf{P} u_\epsilon \to u,~~\reps \mathbf{P}\left( \divg \intv \hat{A} \mathcal{L}g_\epsilon \bdv \right)  \to u \sdot \nabla u - \mu \Delta u.
\end{aligno}
For the temperatrue equation, it is slightly different with that of the Boltzmann case where there is no term $\tfrac{2}{5} u_\epsilon \sdot \nabla \phi_\epsilon$. By the similar way of deducing \eqref{aaa}, we have
\begin{align}
\label{bbb}
  \tfrac{2}{5}   \intv \hat{B} \sdot \reps \mathcal{L}g_\epsilon \bdv = u_\epsilon\sdot \theta_\epsilon - \kappa \nabla \theta_\epsilon -  R_b(\epsilon)
\end{align}
with
\[R_b(\epsilon) :=  \tfrac{5}{2}\intv B \sdot \left( O(\epsilon) -  \vdot g^\perp_\epsilon +  {\Gamma}(g_\epsilon^\perp,g_\epsilon) + {\Gamma}(g_\epsilon,g_\epsilon^\perp) \right) \bdv.  \]
Plugging \eqref{bbb} into the second equation of \eqref{nsp-app-4}, we obtain that
\begin{align*}
\partial_t \left(\tfrac{3}{5}\theta_\epsilon - \tfrac{2}{5}\rho_\epsilon \right) + \mathbf{P}u_\epsilon\sdot \nabla_x \left(\tfrac{3}{5}\theta_\epsilon - \tfrac{2}{5}\rho_\epsilon \right)- \kappa \Delta \theta_\epsilon   & = R_2(\epsilon),
\end{align*}
with
\[ R_2(\epsilon) = \tfrac{2}{5} \divg R_b(\epsilon) + \tfrac{2}{5}\mathbf{P}_\perp u_\epsilon \sdot \nabla \phi_\epsilon - \divg (\mathbf{P}_\perp u_\epsilon \sdot \theta_\epsilon) - \tfrac{2}{5} \mathbf{P} u_\epsilon \sdot \nabla(\rho_\epsilon + \theta_\epsilon  - \phi_\epsilon),\]
and
\[ \mathbf{P}_\perp u_\epsilon = u_\epsilon - \mathbf{P} u_\epsilon,~~ \kappa = \tfrac{2}{15} \sum\limits_{ 1 \le i \le 3  }\intv B_{i}\hat{B}_{i}\bdv. \]
Noticing that $$u_\epsilon \in H^s_{x,z}, $$ by the Sobolev embedding theorem, without loss generality, we can obtain that
\[ u_\epsilon \to u,~~\text{in},~~H^{s-1}_xL^2_z.\]
Thus, we can deduce that
\[ \mathbf{P}_\perp u_\epsilon \to 0,~~\text{in},~~L^2.  \]
Based on the estimates \eqref{estLimits} and the above fact, one can verify that in the distributional sense
\begin{align*}
R_2(\epsilon) \to 0,
\end{align*}
and
\begin{align}
\partial_t(\tfrac{3}{2}\theta - \rho) + u \sdot\nabla (\tfrac{3}{2}\theta - \rho) - \tfrac{5\kappa}{2}\Delta \theta =0.
\end{align}
Thus, we have verify that in distributional sense
\begin{align}
\label{limit-g}
g_\epsilon \to g = \rho + u \sdot v + \tfrac{\theta}{2}(|v|^2 - 3),
\end{align}
with $\rho,~u,~\theta \in  L^\infty([0,+\infty);H^s_{x,z})$  and satisfying
\begin{align}
\label{nsfp-c}
\begin{cases}
\partial_t u + u\sdot \nabla u - \nu \Delta u + \nabla P = \rho \nabla \theta,\\
\partial_t(\tfrac{3}{2}\theta - \rho) + u \sdot\nabla (\tfrac{3}{2}\theta - \rho) - \tfrac{5\kappa}{2}\Delta \theta =0,\\
\divg u = 0,~~ \Delta(\rho + \theta ) = \rho,~~E = \nabla(\rho + \theta).
\end{cases}
\end{align}

{\bf Improving the regularity of $z$ and $t$.} The above process of verifying \eqref{limit-g} is established in distributional sense. In fact, the convergence of $g_\epsilon$ can be improved to strong convergence (at least in H\"older space).

From \eqref{estLimit3} ($s\ge 5$) and \eqref{estLimits}, with the help of \eqref{nsp-app-4}, we can obtain that
\begin{align}
\partial_t \mathbf{P} \mathbf{u}_\epsilon,~~\partial_t \left(\tfrac{3}{5}\theta_\epsilon - \tfrac{2}{5}\rho_\epsilon \right) \in H^{s-1}_{x,z}.
\end{align}
Then, by the Aubin-Lions-Simon theorem (see \cite{tools-ns}),
\[ \mathbf{P} \mathbf{u}_\epsilon,  \left(\tfrac{3}{5}\theta_\epsilon - \tfrac{2}{5}\rho_\epsilon \right) \in C((0,+\infty;H^{s-1}_{x,z}), \]
and
\[ \mathbf{P} \mathbf{u},  \left(\tfrac{3}{5}\theta  - \tfrac{2}{5}\rho  \right) \in C((0,+\infty;H^{s-1}_{x,z}). \]

Furthermore, by the Sobolev embedding theorem, for any $(t,x,z) \in (0,+\infty) \times \mathbb{T}^3\times \mathbb{I}_z$, there exists some constant $C_0$ ( only dependent of the initial data) such that
\begin{align*}
| \mathbf{P} \mathbf{u}_\epsilon| +  |\left(\tfrac{3}{5}\theta_\epsilon - \tfrac{2}{5}\rho_\epsilon \right)| \le C_0.
\end{align*}
Furthermore, for any $\delta>0$,  $\mathbf{P} u_\epsilon$ and $\left(\tfrac{3}{5}\theta_\epsilon - \tfrac{2}{5}\rho_\epsilon \right)$ are Lispchtiz continuous on $ [\delta,+\infty) \times \mathbb{T}^3\times \mathbb{I}_z$. This means that
\[ \mathbf{P} u_\epsilon, \left(\tfrac{3}{5}\theta_\epsilon - \tfrac{2}{5}\rho_\epsilon \right) \text{ are uniformly bounded and equicontinuous on } [\delta,+\infty) \times \mathbb{T}^3\times \mathbb{I}_z. \]
Then by the Arzel\`a–Ascoli theorem, up to a subsequence,
\[ \mathbf{P} u_\epsilon \xrightarrow{~~~~\text{uniformly converge}~~~~ } u,~~\left(\tfrac{3}{5}\theta_\epsilon - \tfrac{2}{5}\rho_\epsilon \right) \xrightarrow{~~~~\text{uniformly converge} ~~~~} \left(\tfrac{3}{5}\theta - \tfrac{2}{5}\rho  \right).  \]
As a consequence,
\[ \rho,~~u,~~\theta \in C((0,+\infty) \times \mathbb{T}^3\times \mathbb{I}_z). \]
If the initial data are well prepared,
\[ \rho,~~u,~~\theta \in C([0,+\infty) \times \mathbb{T}^3\times \mathbb{I}_z). \]

\appendix
\section{Analysis of spectrum}
\label{sec-spec}
This section consists of calculating the eigenvalue of $\mathcal{B}_\epsilon$ and its eigenvector. The following theorem is a counterpart of \cite[Theorem 2.11]{spectrum-arxiv-vpb}.

\begin{theorem}
\label{theorem-decom}
For some constant $r_0>0$, the spectrum set of $\mathcal{B}_1(\xi)$ is made up of five eigenvalues:$$\{\lambda_j(|\xi|),\ j=-1,0,1,2,3\},~~|\xi| \le r_0, ~~\mathrm{Re}\lambda_i >-\hat a_2 /2.$$ Furthermore, the spectrum $\lambda_j(|\xi|)$ and its associate eigenfunction $\psi_j(s,\omega)$ ($s=|\xi|,~\omega = \tfrac{\xi}{|\xi|}$) are $C^\infty$ functions of $s$ while $|s|\leq r_0$. In additions, if $|s|\leq r_0$, then $\lambda_i$ enjoys  asymptotic expansion:
\begin{equation}
\left\{\begin{aligned} \lambda_{\pm 1}(s) &=\pm \epsilon\mathrm{i}+\left(a_{11} \pm \mathrm{i} \tfrac{5}{3\epsilon}\right) s^{2}+O\left(s^{3}\right), \quad \overline{\lambda_{1}}=\lambda_{-1} \\ \lambda_{0}(s) &=a_{44} s^{2}+O\left(s^{3}\right), \\ \lambda_{2}(s) &=\lambda_{3}(s)=a_{22} s^{2}+O\left(s^{3}\right). \end{aligned}\right.
\end{equation}
Furthermore, the semigroup  $e^{\mathcal{B}_\epsilon(\xi)t}$ has the following low and high frequencies decomposition$\colon$
\begin{equation}
\label{semi-decompose}
\begin{aligned}
e^{\mathcal{B}_\epsilon(\xi)t}\hat{g}_\epsilon & = \bds_{1}(t, \xi) \hat{g}_\epsilon  + \bds_{2}(t, \xi) \hat{g}_\epsilon , \\
& = \sum_{j=-1}^{3} e^{t \lambda_{j}(|\xi|)} \bdp_{j}(\xi) \hat{g}_\epsilon(0)\mathbf{1}_{|\xi| \le r_0} +
\bds_{2}(t, \xi) \hat{g}_\epsilon(0),
\end{aligned}
\end{equation}
with
\begin{equation}
\label{projectionP}
\begin{aligned} \bdp_{j}(\xi) \hat{g}_{\epsilon} =&\left(\hat{m}_\epsilon  \cdot w^{j}\right)\left(w^{j} \cdot v\right)  +|\xi| \mathrm{T}_{1,j}(\xi) \hat{g}_{\epsilon}  + |\xi|^2 \mathrm{T}_{2,j}(\xi) \hat{g}_{\epsilon}, \quad j=2,3 \\ \bdp_{0}(\xi) \hat{g}_{\epsilon} =&\left( \sqrt{\tfrac{3}{2}} \hat{\theta}_{\epsilon} -\sqrt{\tfrac{2}{3}} \hat{\rho}\right) \chi_{4}+|\xi| \mathrm{T}_{1,0}(\xi) \hat{g}_{\epsilon}   + |\xi|^2 \mathrm{T}_{2,0}(\xi) \hat{g}_{\epsilon} \\ \bdp_{\pm 1}(\xi) \hat{g}_{\epsilon}=& \tfrac{1}{2}\left[\left(\hat{m}_{\epsilon}  \cdot \omega\right) \mp \tfrac{\epsilon}{|\xi|} \hat{\rho} \mp \reps (\hat{\rho}_\epsilon+ \hat{\theta}_\epsilon) |\xi|\right](v \cdot \omega) \\ & +\tfrac{\epsilon}{2} \hat{\rho}_{\epsilon} \left(1+\sqrt{\tfrac{2}{3}} \chi_{4}\right)    +|\xi| \mathrm{T}_{1,\pm 1}(\xi) \hat{g}_{\epsilon}  + |\xi|^2 \mathrm{T}_{2,\pm 1}(\xi) \hat{g}_{\epsilon}, \end{aligned}
\end{equation}
and
\begin{equation}
\chi_{0}=1, \quad \chi_{j}=v_{i}  (i=1,2,3), \quad \chi_{4}=\tfrac{\left(|v|^{2}-3\right) }{\sqrt{6}}.
\end{equation}
Here, in \eqref{projectionP},  $\bdp_{j}(\xi) \hat{g}_{\epsilon}(0)$ is the projection of $\hat{g}_{\epsilon}(0)$ onto the space spanned by the   eigenfunctions related to $\lambda_j$ and the right hand are their   Taylor expansion.

Furthermore, the high frequency part enjoys a exponential decay, i.e., there exists some $C_{r_0}$ and $\sigma>0$, for any $\xi$,
\[ \|\bds_2(t,\xi)\hat{g}\|_{L^2_v} \le C_{r_0} e^{-\sigma t}\|g\|_{L^2_v}.  \]
\end{theorem}
\begin{remark}
\label{remark-a2}
Compared to the results of \cite{spectrum-arxiv-vpb}, the main difference happens on the projection operator $\bdp_{\pm}$.  Here we make some comments on $\reps (\hat{\rho}_\epsilon(0) + \hat{\theta}_\epsilon(0)) |\xi|$. While the $\epsilon=1$, owing to  the coefficent $|\xi|$, this one is contained in $|\xi| \mathrm{T}_{\pm 1}(\xi) \hat{g}_{\epsilon}(0)$ in \cite{spectrum-arxiv-vpb}. Here, since there exists coefficent $\epsilon^{-1}$,  we deal it in a different way. From \eqref{bepsilon},    the $\xi$ will be replaced by $\epsilon n$ in Sec. \ref{sec-con} and the coefficient $\epsilon^{-1}$ will not bring bad effect. Furthermore, the coefficent before $v\sdot \omega$, i.e.,
\[ \left[\left(\hat{m}_{\epsilon}(0) \cdot \omega\right) \mp \tfrac{\epsilon}{|\xi|} \hat{n}_{0} \mp \reps (\hat{\rho}_\epsilon(0) + \hat{\theta}_\epsilon(0)) |\xi|\right]\]
hints the well-prepared initial data which means
\[ \divg u_\epsilon(0) =0, \Delta\left(\rho_\epsilon(0) + \theta_\epsilon(0)\right) = \rho_\epsilon(0).  \]
\end{remark}
\begin{proof}
This theorem serves to calculate the eigenvalues $\lambda$ and eigenfunctions $\psi$ of $\mathcal{B}_\epsilon$ for each $\epsilon>0$, i.e.,
\begin{align}
\label{eigen}
\lambda \psi= \mathcal{B}_\epsilon \psi.
\end{align}
While $\epsilon =1$, this theorem is the same to the Theorem 2.1 in \cite{spectrum-arxiv-vpb} where they considered the spectrum properties of $\mathcal{B}_1$. Since the only difference between $\mathcal{B}_\epsilon$ and $\mathcal{B}_1$ is the coefficient $\epsilon^2$ before the electric field, their proof can be directly adapted to our case with some modification. Owing to  the coefficient $\epsilon^2$, there exists the same $\epsilon^2$ in the definition of the inner product $(f,g)_\xi$, i.e.,
\[ (f,g)_\xi = (f,g) + \tfrac{\epsilon^2}{|\xi|^2} (\mathrm{P}_{\mathrm{d}} f, \mathrm{P}_{\mathrm{d}} g),  ~~(f,g) = \intv f \bar{g} \bdv,~~ \mathrm{P}_{\mathrm{d}} f =  \intv f \bdv. \]
The existence of eigenvalues is guaranteed by the strict semigroup theory and was clearly verified in \cite{spectrum-vpb-2016-bipolar-decay,spectrum-arxiv-vpb,convergence-rate-linear-1975}.  In what follows, we sketch the idea of calculating eigenvalues and eigenfunctions for low frequency. The proof of the  high frequency ($\bds_2$) are the same to \cite[Lemma 2.4]{spectrum-arxiv-vpb}.

For any $h$ satisfying \eqref{eigen}, it can be orthogonally split into fluid part $h$ and microscopic part:
\[ h =:\bdp h    + \bdpv h = h_0 + h_1,  \]
where
\[h_0 = \intv h_0 \bdv   + v_i \intv v_i h \bdv + \tfrac{\left(|v|^{2}-3\right) }{\sqrt{6}} \intv \tfrac{\left(|v|^{2}-3\right) }{\sqrt{6}} h \bdv.   \]
Plugging this ansta into \eqref{eigen}, we can find that
\begin{equation}
\label{eigen-decom}
\begin{aligned}
&\lambda h_{0}=-\bdp\left[\mathbf{i}(v \cdot \xi)\left(h_{0}+h_{1}\right)\right]-\epsilon^2 \tfrac{\mathrm{i}(v \cdot \xi) }{|\xi|^{2}} \int_{\mathbb{R}^{3}} h  \bdv,\\
&\lambda h_{1}=\mathcal{L}h_{1}-\bdpv\left[\mathbf{i}(v \cdot \xi)\left(h_{0}+h_{1}\right)\right].
\end{aligned}
\end{equation}
From the second equation of \eqref{eigen-decom} on $h_1$, we can find that
\[ \left(\lambda \bdpv- \mathcal{L} + \mathrm{i}\bdpv (v \sdot \xi) \bdpv\right) h_{1}=-\mathrm{i} \bdpv(v \cdot \xi) h_{0}. \]
From \cite{spectrum-vpb-2016-bipolar-decay,spectrum-arxiv-vpb}, while $\mathrm{Re}\lambda> -a_2$, the operator before $h_1$ is reversible. Denoting
\[ \mathbf{R}(\lambda, \xi) = -\left(\lambda \bdpv- \mathcal{L} + \mathrm{i}\bdpv (v \sdot \xi) \bdpv\right)^{-1},  \]
thus
\[ h_1 = \mathrm{i} \mathbf{R}(\lambda,\xi)\bdpv(v \cdot \xi) h_{0}. \]
Substituting $h_1$ into the first equation of \eqref{eigen-decom}, it follows that
\[ \lambda h_0 = -\mathrm{i}\bdp(v\sdot \xi)h_0 + \bdp(v\sdot \xi)\mathbf{R}(\lambda,\xi)\bdpv(v\sdot \xi) h_0  -\epsilon^2 \tfrac{\mathrm{i}(v \cdot \xi) }{|\xi|^{2}} \int_{\mathbb{R}^{3}} h_0  \bdv.  \]
This  equation are only related to the fluid parts. Applying the above equation to $\hat{g}_\epsilon$, recalling that $\hat{g}_\epsilon$ admits the following type decomposition:
\[\bdp \hat{g}_\epsilon = (\hat\rho_\epsilon,\hat u_{\epsilon,1},\hat u_{\epsilon,2},\hat u_{\epsilon,3},\hat \Theta_\epsilon)\sdot(\chi_0,\chi_1,\chi_2,\chi_3,\chi_4)^T, ~~\hat u_\epsilon=(\hat u_{\epsilon,1},\hat u_{\epsilon,2},\hat u_{\epsilon,3}).   \]
thus, it follows that
\begin{equation}
\label{ei-matrix}
\begin{aligned}
\lambda \hat\rho_\epsilon=&-\mathrm{i}(\hat u_\epsilon \sdot \xi), \\
\lambda \hat u_{\epsilon,i}=&-\mathrm{i} \hat\rho_\epsilon\left(\xi_{i}+\epsilon^2 \tfrac{\xi_{i}}{|\xi|^{2}}\right)-\mathrm{i} \sqrt{\tfrac{2}{3}} \hat\Theta_\epsilon \xi_{i}+\sum_{j=1}^{3} \hat u_{\epsilon,j}b_{i,j}
+\hat\Theta_\epsilon b_{i,4}, \\
\lambda \hat\Theta_\epsilon=&-\mathrm{i} \sqrt{\tfrac{2}{3}}(\hat u_{\epsilon,i} \cdot \xi)+\sum_{j=1}^{3} \hat u_{\epsilon,j}b_{4,j}
+ \hat\theta_\epsilon b_{4,4}
\end{aligned}
\end{equation}
where $a_{i,j}$ is defined as follows:
\[ b_{i,j} = \intv \mathbf{R}(\lambda, \xi) \bdpv\left((v \cdot \xi) \chi_{j}\right)\sdot (v \sdot \xi) \chi_{i}\bdv.   \]
Denoting the unit vector of $\xi$ by $\omega$, i.e,
\[ \xi = s \omega, \]
and $U_\epsilon=(\hat\rho_\epsilon, \hat u_\epsilon\sdot \omega, \hat\theta_\epsilon)$,
according to \cite[Lemma 2.7]{spectrum-arxiv-vpb} and \cite[Lemma 2.8]{spectrum-arxiv-vpb}, there are five eigenvalues. Two of them are
\begin{align*}
\lambda_2(s)=\lambda_3(s)= a_{22}(\lambda_2)s^2 + o(s^2).
\end{align*}
The rest   can be obtained by solving the following eigenvalue problem:
\begin{equation}
\left(\begin{array}{ccc}
{\lambda} & {\text{i} s } & {0} \\
{\mathrm{i}\left(s+\frac{\epsilon^2}{s}\right)} & {\lambda-s^{2} a_{11}} & {\text { is } \sqrt{\frac{2}{3}}-s^{2} a_{41}} \\
{0} & {\text{i}s \sqrt{\frac{2}{3}}-s^{2} a_{14}} & {\lambda-s^{2} a_{44}}
\end{array}\right)U_\epsilon^t =0,
\end{equation}
with
\[ a_{i,j} = \intv \mathbf{R}(\lambda, s e_1) \bdpv\left(v_1 \chi_{i}\right)\sdot v_1 \chi_{j}\bdv.   \]
With
\[ D_\epsilon(\lambda,s)= \left\vert \begin{array}{ccc}
{\lambda} & {\text{i} s } & {0} \\
{\mathrm{i}\left(s+\frac{\epsilon^2}{s}\right)} & {\lambda-s^{2} a_{11}} & {\text { is } \sqrt{\frac{2}{3}}-s^{2} a_{41}} \\
{0} & {\text{i}s \sqrt{\frac{2}{3}}-s^{2} a_{14}} & {\lambda-s^{2} a_{44}}
\end{array} \right\vert  \]
by simple computation, we can deduce
\begin{equation}
\label{ei-det-s}
\begin{aligned}
D_\epsilon(\lambda, s)=\lambda^{3}-\lambda^{2} s^{2}\left( a_{11} + a_{44}\right) -\left(s^{2}\epsilon^2+s^{4}\right)a_{44} \\
+\lambda\left(\epsilon^2+\tfrac{5}{3} s^{2}+\mathrm{i} \sqrt{\tfrac{2}{3}} s^{3} ( a_{41} + a_{14})
+s^{4} (a_{44}a_{11} - a_{41}a_{14})\right).
\end{aligned}
\end{equation}
For $D_\epsilon(\lambda, 0)=\lambda(\lambda^2 + \epsilon^2)$, there exist three eigenvalues $\pm\epsilon \mathrm{i},0$. Furthermore,
\[\partial_s D_\epsilon(  \epsilon \mathrm{i}\sdot k, 0) =0,~~ \partial_\lambda D_\epsilon( \epsilon \mathrm{i}\sdot k, 0) = \epsilon^2 -3k^2\epsilon^2 = (1-3k^2)\cdot\epsilon^2,~~~~k=0,\pm 1. \]
By the implicit function theorem, for some $r_0>0$ each eigenvalue is a $C^\infty$ function of $s$ such that
\[
D_\epsilon(\lambda_k(s), s) =0,~~\lambda_k(0)= \epsilon\mathrm{i}\cdot k \quad \text { and } \quad \lambda_k^{\prime}(0)=0, \quad k=0, \pm 1,~~ -r_0 \le s \le r_0.
\]
Thus, the eigenvalue enjoys the following Talyor expansion:
\[ \lambda_k(s) = \epsilon \mathrm{i} \sdot k + \tfrac{\lambda^{\prime\prime}_k(0)}{2}s^2 + o(s^2).  \]
Furthermore, from \eqref{ei-det-s}, noticing that
\begin{align*}
\partial_s^2 D_\epsilon(\lambda_k(0), 0)=  2\epsilon^2 \sdot k^2(a_{11}+ a_{44})  -2 \epsilon^2 a_{44} + \tfrac{10\epsilon \mathrm{i}}{3}\sdot k,~~~~k =0,\pm,
\end{align*}
then the second derivative of $\lambda_k(s)$ at $s=0$ are
\begin{align}
\lambda^{\prime\prime}_k(0) = -\tfrac{\partial_s^2 D_\epsilon(\epsilon \mathrm{i} \sdot k, 0)}{\partial_\lambda D_\epsilon( \epsilon \mathrm{i} \sdot k, 0)} = -\tfrac{ 2\epsilon^2 \sdot k^2(a_{11}+ a_{44}) -2 a_{44}\epsilon^2 + \tfrac{10\epsilon \mathrm{i}}{3}\sdot k}{(1-3k^2)\cdot\epsilon^2}.
\end{align}
Here the $a_{11}$ and $a_{44}$ are dependent of eigenvalues.
From the above equation, the three eigenvalues admit the following expansion
\begin{equation}
\begin{split}
\lambda_0(s) & =  {a_{44}}(\lambda_0)s^2 + o(s^2),\\
\lambda_{\pm 1}(s) & = \pm\epsilon \mathrm{i} + \tfrac{1}{2} \left( {a_{11}}(\lambda_{\pm 1})  \pm \tfrac{5\mathrm{i}}{3 \epsilon}   \right)s^2 + o(s^2).
\end{split}
\end{equation}

Since the eigenvalues have been figured out, for fixed $\lambda_j$, the eigenfunctions satisfy the following relation:
\[ \lambda_j \psi_j=\mathcal{L}\psi_j-   \mathrm{i}(v \cdot n) \psi_j -   \epsilon^2 \frac{\mathrm{i}(v \cdot n)}{|n|^{2}}  \int_{\mathbb{R}^{3}} \psi_j \bdv.\]
The method  of constructing the eigenfunction is similar to that of calculating the eigenvalues. Indeed, the eigenfunction can be decomposed into fluids parts and microscopic parts as \eqref{eigen-decom}. Similar matrix like \eqref{ei-matrix} can be deduced too. By using the Taylor expansion,
\[
\psi_{j}(s, \omega)=\psi_{j, 0}(\omega)+\psi_{j, 1}(\omega) s+\psi_{j, 2}(\omega) s^{2}+o\left(s^{2}\right), \quad|s| \leq r_{0}
  \]
the eigenfunction can be calculated. Here we omit the details. For more details, we refer to \cite[Theorem 2.8]{spectrum-arxiv-vpb}. The eigenfunctions are
\begin{equation}
\left\{\begin{array}{l}
{\psi_{0,0}=\chi_{4}, \quad \psi_{0,1}=\mathrm{i} \mathcal{L}^{-1} \bdpv(v \cdot \omega) \chi_{4}, \quad\left(\psi_{0,2}, 1 \right)=-\repst \sqrt{\frac{2}{3}}} \\
{\psi_{\pm 1,0}=\frac{\sqrt{2} }{2}(v \cdot \omega)  , \quad\left(\psi_{\pm 1,2}, 1\right)=0} \\
{\psi_{\pm 1,1}=\mp \reps \left(\frac{\sqrt{2}}{2}   + \frac{\sqrt{3}}{3} \chi_{4} \right)+\frac{\sqrt{2}}{2} \mathrm{i}\left(\mathcal{L} \mp \mathrm{i} \bdpv\right)^{-1} \bdpv(v \cdot \omega)^{2}  } \\
{\psi_{j, 0}=\left(v \cdot c^{j}\right)  , \quad\left(\psi_{j, n}, 1\right)=\left(\psi_{j, n}, \chi_{4}\right)=0 \quad(n \geq 0)} \\
{\psi_{j, 1}=\mathrm{i} \mathcal{L}^{-1} \bdpv\left[(v \cdot \omega)\left(v \cdot c^{j}\right)\right], \quad j=2,3}
\end{array}\right.
\end{equation}
With the eigenvalue and eigenfunction at our disposal, the low and high frequency decomposition  is  defined as follows:
\begin{align*}
\begin{aligned}
e^{\mathcal{B}_\epsilon(\xi)t}\hat{g}_\epsilon & = \bds_{1}(t, \xi) \hat{g}_\epsilon  + \bds_{2}(t, \xi) \hat{g}_\epsilon.
\end{aligned}
\end{align*}
Here, $\bds_1$ denotes the low frequency part, specially,
\begin{align*}
 \bds_1(t,\xi) \hat{g}_\epsilon =  \sum_{j=-1}^{3} e^{t \lambda_{j}(|\xi|)} \bdp_{j}(\xi) \hat{g}_\epsilon\mathbf{1}_{|\xi| \le r_0}.
\end{align*}
Here $\bdp_j$ denotes the projection onto the $j$-$\mathrm{th}$ eigenfunction under the inner product $(,)_\xi$
\begin{align*}
\bdp_{j}(\xi)\hat{g}_\epsilon \mathbf{1}_{|\xi| \le r_0}  = (\hat{g}_\epsilon,\psi_j)_\xi\sdot \psi_j \mathbf{1}_{|\xi| \le r_0}:= \bdp_{j,0}(\xi) + |\xi|\mathrm{T}_j(\xi)\hat{g}_\epsilon.
\end{align*}
Here, the $\bdp_{j,0}(\xi) + |\xi|\mathrm{T}_j(\xi)\hat{g}_\epsilon$ is the Taylor expansion of $\bdp_{j}(\xi)\hat{g} \mathbf{1}_{|\xi| \le r_0}$. Nocticing that usually
\[  \theta_\epsilon:= \intv \tfrac{|v|^2 -1 }{3} g_\epsilon \bdv, \]
there
\begin{align*}
\hat{\Theta}_\epsilon = \intv \chi_4 \hat{g}_\epsilon\bdv = \sqrt{\frac{3}{2}} \hat{\theta}_\epsilon.
\end{align*}
Then we complete the proof.

\end{proof}

\end{document}